\documentclass{article}
\usepackage{amsfonts, amsmath, amssymb, amscd, amsthm, graphicx,enumerate,comment}
\usepackage{hyperref,xypic}
\usepackage[usenames]{color}
\usepackage{tikz}
\usepackage{tikz-cd}
\usepackage{calc}
\usepackage{braket}
\usepackage{mathrsfs}
\usepackage[marginpar=3cm]{geometry}
\usepackage{overpic}
\usepackage[shortlabels]{enumitem}
\usepackage{mathtools}
\usepackage{thmtools}
\usepackage{thm-restate}
\usepackage{subcaption}
\usepackage{titlesec}

\makeatletter
\def\blfootnote{\xdef\@thefnmark{}\@footnotetext}
\makeatother

\newtheorem{thm}{Theorem}[section]
\newtheorem{thma}{Theorem}

\newtheorem{prop}[thm]{Proposition}
\newtheorem{lem}[thm]{Lemma}
\newtheorem{cor}[thm]{Corollary}

\theoremstyle{remark}
\newtheorem{rem}[thm]{Remark}

\newtheorem{quest}[thm]{Question}
\newtheorem{ex}[thm]{Example}

\theoremstyle{definition}
\newtheorem{defn}[thm]{Definition}

\newtheorem{conv}[thm]{Convention}

\newcommand{\R}{\mathbb{R}}

\renewcommand{\H}{\mathbb{H}}

\newcommand{\N}{\mathbb{N}}
\newcommand{\Ext}{\operatorname{Ext}}

\numberwithin{equation}{section}

\title{Transverse measures to infinite type laminations}
\author{Mladen Bestvina \and Alexander J. Rasmussen}
\date{}

\begin{document}

\maketitle

\begin{abstract}
We study the cone of transverse measures to a fixed geodesic lamination on an infinite type hyperbolic surface. Under simple hypotheses on the metric, we give an explicit description of this cone as an inverse limit of finite-dimensional cones. We study the problem of when the cone of transverse measures admits a base and show that such a base exists for many laminations. Moreover, the base is a (typically infinite-dimensional) simplex (called a \emph{Choquet simplex}) and can be described explicitly as an inverse limit of finite-dimensional simplices. We show that on any fixed infinite type hyperbolic surface, every Choquet simplex arises as a base for \emph{some} lamination. We use our inverse limit description and a new construction of geodesic laminations to give other explicit examples of cones with exotic properties.
\end{abstract}
	
\section{Introduction}

Geodesic laminations on infinite type surfaces are currently poorly understood. However, they promise to be valuable tools in the study of the mapping class groups and Teichm\"{u}ller theory of infinite type surfaces. As an example, understanding geodesic laminations would help to advance the study of hyperbolic graphs associated to infinite type type surfaces such as the \emph{ray graph} (\cite{ray},\cite{ray_boundary}). What is missing is a structure theory for laminations on infinite type surfaces.

Lacking such a structure theory, one may attempt to understand the \emph{ergodic theory} of infinite type laminations; i.e. the theory of transverse measures to infinite type laminations. Such a goal has been undertaken recently (\cite{thurston_boundary}, \cite{length_spectrum}, \cite{measured_lams}). Both the structure and the ergodic theory of laminations on finite type surfaces are well understood. Geodesic laminations on finite type surfaces consist of finitely many minimal sub-laminations together with finitely many isolated leaves (see e.g. \cite[Theorem I.4.2.8]{fundamentals}). The cone of transverse measures to a finite type lamination is a finite-dimensional simplicial cone. Its base is a simplex which embeds projectively into the \emph{Thurston boundary} of Teichm\"{u}ller space. Moreover, the space of all measured laminations on a fixed closed hyperbolic surface has a natural piecewise-linear structure. Our goals in this paper are to give a very explicit description of the cone of transverse measures to an infinite type lamination and to highlight similarities and differences with the finite type theory as well as connections with ergodic theory and functional analysis.

Fix a complete hyperbolic surface $X$ of infinite type, without boundary, and a geodesic lamination $\Lambda$ on $X$. We will assume that $X$ is \emph{of the first kind}, meaning that the limit set of $\pi_1(X)$ acting on the universal cover $\widetilde X$ is the entire circle $\partial \widetilde X$. A \emph{transverse measure} to $\Lambda$ assigns to each arc transverse to $\Lambda$ a finite Borel measure and these measures are invariant under isotopies respecting $\Lambda$. Transverse measures may be compared to \emph{invariant measures} of dynamical systems. The space of all transverse measures to $\Lambda$ has the structure of a topological convex cone with the \emph{weak${}^*$ topology}. We denote it by $\mathcal M(\Lambda)$. 

Our initial result gives a rough description of $\mathcal M(\Lambda)$. This will be refined momentarily into a much more explicit description of $\mathcal M(\Lambda)$ as an inverse limit of finite-dimensional cones. Here $\R_+=[0,\infty)\subset \R$.

\begin{thma}
\label{mainthm:subcone}
The cone $\mathcal M(\Lambda)$ is linearly homeomorphic to a closed sub-cone of the product $\R_+^\N$ cut out by countably many linear equations.
\end{thma}

\noindent This appears to be related to the main result of \cite{measured_lams}, which describes the cone of measured laminations carried by a train track via equations. However, Theorem \ref{mainthm:subcone} does not follow immediately from this.

To understand the cone $\mathcal M(\Lambda)$ better, we fix an \emph{exhaustion} of $X$, $X_1\subset X_2\subset \ldots$, by surfaces with geodesic boundary which are compact, minus finitely many punctures. The intersection $\Lambda \cap X_n$ consists of finitely many compact minimal sub-laminations contained in the interior of $X_n$, geodesics spiraling onto these minimal sub-laminations, plus some collection of proper arcs. Moreover, there are finitely many proper arcs in $\Lambda \cap X_n$ up to homotopy. We attach to each $X_n$ a finite-dimensional cone $C(X_n)$, which records all the transverse measures to the compact minimal sub-laminations, plus assignments of non-negative numbers to all the homotopy classes of proper arcs. There are natural \emph{transition maps} $\pi_n:C(X_{n+1})\to C(X_n)$ which record how the arcs and minimal sub-laminations in $X_{n+1}$ traverse those in $X_n$. This leads to our explicit description: 

\begin{thma}
\label{mainthm:inverselim}
The cone $\mathcal M(\Lambda)$ is linearly homeomorphic to the inverse limit of the cones $C(X_n)$ together with the transition maps $\pi_n$.
\end{thma}

\noindent The advantage of Theorem \ref{mainthm:inverselim} is that the finite-dimensional cones $C(X_n)$ and transition maps $\pi_n$ are easily computable in practice, so that the theorem gives a very explicit description of the cone $\mathcal M(\Lambda)$. As first examples, we construct a lamination with a single non-zero transverse measure up to scaling (Example \ref{ex:choquetexample3}), and another lamination with no non-zero transverse measures at all:

\begin{thma}
\label{mainthm:zeromeasure}
Let $X$ be a complete, infinite type hyperbolic surface of the first kind. Then there exists a geodesic lamination $\Lambda$ on $X$ that has no non-zero transverse measures.
\end{thma}

We next study the problem of when the cone $\mathcal M(\Lambda)$ admits a convex, compact cross section (a \emph{base}). We show that such bases do exist in many examples and are examples of \emph{Choquet simplices}. Choquet simplices are infinite-dimensional versions of finite-dimensional simplices, familiar from dynamics and functional analysis. As is well known, the space of invariant probability measures of a homeomorphism of a compact metric space is always a Choquet simplex. 

\begin{thma}
\label{mainthm:choquetbase}
Suppose that there is a compact subsurface of $X$ which intersects every leaf of $\Lambda$. Then $\mathcal M(\Lambda)$ has a base which is a compact metrizable Choquet simplex. Further, there is an exhaustion $X_1\subset X_2\subset \ldots$ of $X$ for which this Choquet simplex is the inverse limit of bases of the cones $C(X_n)$ with the restrictions of the maps $\pi_n$.
\end{thma}

\noindent In particular, this theorem applies to any \emph{minimal} lamination. Choquet simplices can have exotic spaces of extreme points. For example, in Example \ref{ex:choquetexample1} the space of extreme points is homeomorphic to the ordinal $\omega+1$. In Example \ref{ex:choquetexample2} the space of extreme points is not closed. An even more exotic example is the \emph{Poulsen simplex} (\cite{poulsen_simplex}), which has a dense set of extreme points. Our next results show that cones of transverse measures can be arbitrarily strange. Namely, there is no obstruction whatsoever to the Choquet simplex that can appear as a base:

\begin{thma}
\label{mainthm:choquetrealization}
Let $X$ be a complete, infinite type hyperbolic surface of the first kind. Let $\Delta$ be a compact metrizable Choquet simplex. Then there exists a minimal geodesic lamination $\Lambda$ on $X$ for which the cone $\mathcal M(\Lambda)$ has a base which is affinely homeomorphic to $\Delta$.
\end{thma}

Realization theorems of this type for Choquet simplices are familiar from dynamics and algebra (\cite{downarowicz, gjerde_johansen, blackader, goodearl}). For instance, \cite{downarowicz} shows that every Choquet simplex arises as the space of invariant probability measures of a minimal compact dynamical system.

Our main tool for proving Theorem \ref{mainthm:choquetrealization} is a construction of laminations as inverse limits of arcs on compact subsurfaces, together with a construction of such inverse limits using planar maps of intervals. These constructions recover \emph{every} geodesic lamination without compact sub-laminations or leaves asymptotic to punctures, and we anticipate that they can be used to construct examples of laminations with other exotic properties. 

Unfortunately, the cones $\mathcal M(\Lambda)$ do not always admit compact bases, the obstruction being sub-laminations disjoint from any given compact subsurface. We give examples in Section \ref{sec:nobase}. One such example is a lamination with cone of transverse measures $\R_+^\N$. 

It would be interesting to connect the methods of this paper with Teichm\"{u}ller theory. In \cite{thurston_boundary}, Bonahon-\v{S}ari\'{c} produce a Thurston boundary for the quasi-conformal deformation space of an infinite type hyperbolic surface. This is the space of projective \emph{bounded} measured laminations. It would be interesting to know whether the cone of bounded transverse measures to an infinite type geodesic lamination (with the \emph{uniform} weak${}^*$ topology) admits an explicit description as an inverse limit, similar to Theorem \ref{mainthm:inverselim}. One could then study bases for such cones and ask:

\begin{quest}
Is there a Choquet simplex which does not embed projectively into the Thurston boundary of the quasi-conformal deformation space of some infinite type hyperbolic surface $X$?
\end{quest}

\paragraph{Structure of the paper.} 
In Section \ref{sec:borelmeasures}, we study the cone of finite measures on a compact totally disconnected metrizable space. We show that it may be described as a closed sub-cone of $\R_+^\N$ cut out by countably many linear equations. This fact is probably well known to the experts but we couldn't find it in the literature. The techniques of Section \ref{sec:borelmeasures} foreshadow those of Section \ref{sec:equations}, where we prove Theorem \ref{mainthm:subcone} describing the cone of transverse measures $\mathcal M(\Lambda)$ to a lamination $\Lambda$ via equations. In Section \ref{sec:inverselim} we prove Theorem \ref{mainthm:inverselim} describing $\mathcal M(\Lambda)$ as an inverse limit. Namely, in Section \ref{sec:transmaps} we give a more precise version of Theorem \ref{mainthm:inverselim}, in Section \ref{sec:examplecones} we give explicit descriptions of certain cones using Theorem \ref{mainthm:inverselim}, and in Section \ref{sec:completingproof} we complete the proof of Theorem \ref{mainthm:inverselim}. In Section \ref{sec:bases} we prove Theorem \ref{mainthm:choquetbase} giving an explicit description of bases for $\mathcal M(\Lambda)$ for certain laminations $\Lambda$. We also give several explicit examples of bases that arise easily. In Section \ref{sec:inverselimlams} we give a construction of laminations on infinite type surfaces as ``inverse limits'' of finite systems of arcs on compact subsurfaces. We use this construction to prove Theorem \ref{mainthm:choquetrealization} in Section \ref{sec:realization}. Finally in Section \ref{sec:nocompactbase} we prove Theorem \ref{mainthm:zeromeasure} and give some examples of laminations $\Lambda$ for which $\mathcal M(\Lambda)$ has no compact base.

\paragraph{Acknowledgements.} 
The authors thank Leonel Robert for helpful email correspondence and the anonymous referee for helpful suggestions that improved the exposition of the paper. The first author was partially supported by NSF grant DMS-1905720. The second author was partially supported by NSF grants DMS-1840190 and DMS-2202986.

  \section{Borel measures on compact totally disconnected metrizable
    spaces}
    \label{sec:borelmeasures}

Before getting started we set up a few definitions. A \emph{cone} $C$ is a set endowed with operations of addition and multiplication by scalars in $\R_+=[0,\infty)$ such that addition is associative and commutative and $c\cdot (v+w)=c\cdot v+c\cdot w$ for $c\in \R_+$ and $v,w\in C$. A particular type of cone is a \emph{convex cone}, which is a subset of a real vector space which is closed under the ambient operations of addition and multiplication by scalars in $\R_+$. A map between convex sets $f:C\to D$ is \emph{affine} if $f(rv+sw)=rf(v)+sf(w)$ for $r,s\in \R_+$ with $r+s=1$ and $v,w\in C$. If $C$ and $D$ are convex cones then $f:C\to D$ is \emph{linear} if additionally $f(0)=0$. We introduce the following convention:

\begin{conv}
Unless stated otherwise, all cones will be assumed to be convex cones. All maps between cones will be assumed to be linear. All maps between convex subsets of cones will be assumed to be affine.
\end{conv}

\noindent An \emph{$n$-dimensional simplicial cone} is a sub-cone of $\R^m$ spanned by $n$ linearly independent vectors.

  Our first theorem previews Theorem \ref{mainthm:subcone}, and illustrates many of the techniques that we use to prove it.
  Let $X$ be a compact, totally disconnected metrizable space (e.g. the Cantor set).
  The theorem below is presumably known to the
  experts, but we couldn't find it in the literature. Let $\mathcal
  M(X)$ be the space of finite Borel measures on $X$ with the weak${}^*$
  topology. This is the weakest topology such that for every
  continuous function $f:X\to\R$, the function $\mathcal M(X)\to\R$,
  $\mu\mapsto \int_Xf~ d\mu$ is continuous. By $\mathcal K(X)$
  denote the (finite or countable) collection of all clopen subsets of
  $X$. Applying the definition to the characteristic function of any
  $K\in \mathcal K(X)$, we see that the function $\mathcal
  M(X)\to\R_+$, $\mu\mapsto \mu(K)$ is continuous. Putting all these
  maps together gives a continuous linear map
  $$\Phi:\mathcal M(X)\to \prod_{K\in \mathcal K(X)}\R_+ .$$
  This product is homeomorphic to $\R_+^n$ for some $n\geq 0$ or to $\R_+^\N$, depending on whether $\mathcal K$ is finite or not.
  However, the map $\Phi$ is usually not surjective. For $K\in \mathcal K(X)$ and a point $x\in\prod_{\mathcal K(X)} \R_+$, we let $x_K$ denote the coordinate of $x$ corresponding to $K$. If $K$ is the disjoint union of $K_1,\ldots,K_r$ then we have $x_{K_1}+\cdots+x_{K_r}=x_K$ for any $x$ in the image of $\Phi$. We endow $\prod_{\mathcal K(X)} \R_+$ with the product topology.

\begin{thm}
\label{thm:measurecone}
    The map
  $\Phi$
  is a linear homeomorphism onto the closed sub-cone
  $C_X\subset\prod_{K\in\mathcal K(X)}\R_+$ cut out by the linear
  equations $x_K=x_{K_1}+\cdots+x_{K_r}$ whenever
  $K=\bigsqcup_{j=1}^r K_j$.
  \end{thm}

  To prove the theorem, we apply the following \emph{Portmanteau Theorem}:
  
  \begin{thm}[{\cite[Theorem 2.3]{billingsley_conv}}]
  \label{thm:portmanteau}
  Let $X$ be a compact, totally disconnected metrizable space. Let $\{\mu_n\}_{n=1}^\infty$ and $\mu$ be finite Borel measures on $X$. Then $\mu_n\xrightarrow{\text{weak}^*} \mu$ if and only if $\mu_n(K)\to \mu(K)$ for every clopen subset $K$ of $X$.
  \end{thm}

\begin{proof}[Proof of Theorem \ref{thm:measurecone}]
  The image is clearly contained in $C_X$.  That $\Phi$ is a bijection
  to this cone follows from the Carath\'eodory Extension Theorem, which states
  that any finite, additive measure on the algebra of sets $\mathcal K(X)$ uniquely
  extends to a Borel measure on $X$ (see e.g. \cite[Section 7.4.2]{einsiedler_ward}). Finally, that $\Phi^{-1}:C_{X}\to
  \mathcal M(X)$ is continuous follows from the Portmanteau Theorem \ref{thm:portmanteau}.
\end{proof}

In practice, one can use much smaller collections of clopen sets and
explicitly compute the cone $C_X$.
We fix a sequence $\mathcal A_i$, $i=1,2,\ldots$
of finite families of clopen subsets of $X$ so that:
\begin{enumerate}[(i)]
\item $\mathcal A_1=\{X\}$;
  \item for each $i>1$, $\mathcal A_i$ forms a finite partition of $X$
    that refines $\mathcal A_{i-1}$; and
    \item for some (any) metric on $X$ the mesh of $\mathcal A_i$ goes
      to 0 as $i\to\infty$.
\end{enumerate}
Also let $\mathcal A=\bigcup_i \mathcal A_i$.

\begin{ex}
  When $X$ is the middle thirds Cantor set we can take $\mathcal A_i$
  to consist of the $2^{i-1}$ clopen sets obtained by intersecting $X$
  with the defining intervals at stage $i$.
  That is, \[\mathcal A_2 = \left\{\left[0,\frac{1}{3}\right]\cap X,\left[\frac{2}{3},1\right]\cap X\right\},\ \ \ \mathcal A_3=\left\{\left[0,\frac{1}{9}\right]\cap X, \left[\frac{2}{9},\frac{1}{3}\right]\cap X, \left[\frac{2}{3},\frac{7}{9}\right]\cap X, \left[\frac{8}{9},1\right]\cap X\right\}, \ \ \ \text{ etc.}\]
  When $X=\{1/n:
  n=1,2,\ldots\}\cup\{0\}$ we can take $\mathcal A_i$ for $i>1$ to consist of the
  singletons $\{1\},\{1/2\},\ldots,\{\frac 1{i-1}\}$ and the set $\{1/n:
  n=i,i+1,\ldots\}\cup\{0\}$.
\end{ex}

Then one obtains a linear map $\mathcal
M(X)\to\prod_{A\in\mathcal A}\R_+$ which is a homeomorphism onto the
sub-cone cut out by the equations $x_A=x_{A_1}+\cdots+x_{A_r}$ when
$A\in\mathcal A_i$, $A_j\in \mathcal A_{i+1}$ and $A=\bigsqcup_{j=1}^r A_j$. The
proof is the same as that of Theorem \ref{thm:measurecone}, since both the Carath\'eodory and
Portmanteau theorems hold for $\mathcal A$.

\begin{ex}
  When $X=\{1/n:
  n=1,2,\ldots\}\cup\{0\}$, after removing redundant coordinates and
  keeping only those corresponding to $\{1/n:
  n=i,i+1,\ldots\}\cup\{0\}$, we see that $\mathcal M(X)$ can be
  identified with the sub-cone of $\R_+^\N$ defined by the
  inequalities $x_1\geq x_2\geq x_3\geq\ldots$.
\end{ex}

\subsection{Bases and Choquet simplices}

Let $B$ be a compact convex set in a metrizable locally convex
topological vector space, such as
$\R^{\N}$ with the product topology. Recall that an {\it extreme point} of $B$ is a point
$x\in B$ that is not contained in the interior of any interval in
$B$. The Krein-Milman Theorem states that $B$ is the smallest closed
convex set that contains the set $\Ext(B)$ of all extreme points of
$B$ (which form a Borel set by \cite[Proposition 1.3]{phelps}).
A stronger version of the Krein-Milman Theorem is Choquet's
Theorem, that for every point $c\in B$ there is a Borel probability
measure $\nu$ supported on the set of extreme points
such that,
formally, $$c=\int_{\Ext(B)} x~ d\nu .$$ This means that for every affine
function $f:B\to\R$, we have
$$f(c)=\int_{\Ext(B)} f(x)~ d\nu$$
(see \cite[Sections 3, 4]{phelps} for all this).
This expression is a generalization of a convex combination.
In general, this measure $\nu$ is not unique. For example, the center
of the square can be written as the midpoint of opposite vertices in
two ways. A compact convex set $B$ as above is a {\it Choquet simplex} if the
measure $\nu$ is unique, for every $c\in B$. A compact convex set in
$\R^n$ is a Choquet simplex if and only if it is a simplex.

A {\it base} of a cone $C$ is a compact convex set
that doesn't
contain 0 and intersects every ray in $C$ based at the origin in exactly one point. For
example, the space of probability measures $\mathcal P(X)$ on $X$
(where $X$ is compact, totally disconnected, metrizable, as
above) is compact by the Banach-Alaoglu Theorem, and so it is a base
for $\mathcal M(X)$. A probability measure on $X$ is extreme in
$\mathcal P(X)$ if and only if it is supported on one point, and the
space $\Ext(\mathcal P(X))$ can be identified with $X$. We now see that
$\mathcal P(X)$ is a Choquet simplex, since for a probability measure
$\mu$ on $X$, the required measure $\nu$ on $\Ext(\mathcal P(X))=X$ is
the measure $\mu$ itself.
As a simpler example, a base for the simplicial cone $\R_+^{n+1}$ is the standard $n$-dimensional simplex.

In finite dimensions Choquet simplices are just the standard
simplices, but in infinite dimensions they can be quite
pathological. The best behaved are {\it Bauer simplices}, whose extreme
points form a closed subset, but there are also {\it Poulsen
  simplices}, whose extreme points are dense
  (see \cite{alfsen, poulsen_simplex} for more on these examples).

\section{$\mathcal M(\Lambda)$ as a sub-cone of $\R_+^{\N}$}
\label{sec:equations}

Let $\Lambda$ be a geodesic lamination on a complete hyperbolic
surface $X$.
In the special case that $X$ is finite type we allow $X$ to have geodesic boundary and we allow the leaves of $\Lambda$ to intersect the boundary transversely. All the definitions below will apply to this special sub-case. Such surfaces with boundary will come up only when we consider an exhaustion of a larger surface. If $X$ is infinite type then we assume that it is without boundary.
In the case that $X$ \emph{does not} have boundary, the universal cover $\widetilde X$ is homeomorphic to the hyperbolic plane $\H^2$ and $\pi_1(X)$ acts on the compactification $\widetilde X \cup \partial_\infty \widetilde X$, where $\partial_\infty \widetilde X$ is the Gromov boundary, i.e. a circle. Fixing any $x\in \widetilde X$, the \emph{limit set} of $X$ is the closure of the orbit of $x$ in $\widetilde X \cup \partial_\infty \widetilde X$ intersected with the Gromov boundary. I.e. the limit set is $\overline{\pi_1(X)\cdot x}\cap \partial_\infty \widetilde{X}$. We will assume throughout the paper that $X$ is of the \emph{first kind}, meaning that the limit set is all of $\partial_\infty \widetilde{X}$.
We have the following theorem of \v{S}ari\'{c}:
\begin{thm}[{\cite[Theorem 1.1]{measured_lams}}]
\label{thm:nowheredense}
Let $X$ be a complete hyperbolic surface of the first kind and $\Lambda$ a geodesic lamination on $X$. Then $\Lambda$ is nowhere dense in $X$.
\end{thm}
A \emph{transversal} or \emph{transverse arc} is an embedded smooth arc $\tau \subset X$ with endpoints
in $X\setminus \Lambda$ such that $\tau$ is transverse to every leaf of $\Lambda$. Two
transversals $\sigma,\tau$ are {\it homotopic} if there is a smooth 
map $F:[0,1]\times [0,1]\to X$ so that the restrictions to
$\{0\}\times [0,1]$ and $\{1\}\times [0,1]$ are diffeomorphisms onto
$\sigma$ and $\tau$, respectively, and the pre-image $F^{-1}(\Lambda)$
consists of horizontal segments $[0,1]\times \{t\}$. Such a map $F$ is a {\it homotopy}
between $\sigma$ and $\tau$.
Denote by $f_\sigma:[0,1]\to \sigma$ the map $f_\sigma(\cdot)=F(0,\cdot)$ and $f_\tau:[0,1]\to \tau$ the map $f_\tau(\cdot)=F(1,\cdot)$. There is an induced diffeomorphism $f=f_\tau \circ f_\sigma^{-1}:\sigma\to\tau$
that preserves intersections with $\Lambda$.

A {\it
  transverse measure} to $\Lambda$ is a function $\mu$ that to each
transversal $\tau$ associates a finite Borel measure $\mu_\tau$ on
$\tau$ subject to the conditions:
\begin{itemize}
  \item if $\tau'\subset\tau$ is a subarc which is also a transversal,
    then $\mu_{\tau'}$ is the restriction of $\mu_\tau$, and
  \item if $F$ is a homotopy from $\sigma$ to $\tau$ and $f$ is the induced diffeomorphism $f=f_\tau\circ f_\sigma^{-1}$, then $\mu_\tau$ is equal to the push-forward measure $f_*(\mu_\sigma)$.
\end{itemize}

\noindent It follows from the definition that $\mu_\tau$ is supported on $\Lambda \cap \tau$.

Let $\mathcal M(\Lambda)$ be the set of transverse measures to
$\Lambda$.
Transverse measures may be added and multiplied by scalars in $\R_+$ simply by performing these operations to each measure $\mu_\tau$. Thus $\mathcal M(\Lambda)$ is a cone.
We endow $\mathcal M(\Lambda)$ with the weakest topology
such that the maps $\mathcal M(\Lambda)\to \mathcal M(\Lambda\cap\tau)$,
$\mu\mapsto\mu_\tau$ are continuous for every transversal $\tau$.
This is called the \emph{weak${}^*$ topology}.
The addition and scalar multiplication operations are continuous in this topology.

\begin{ex}
\label{ex:product}
Let $X$ be a complete hyperbolic surface with finite area and non-empty totally geodesic boundary. Consider a nowhere dense lamination $\Lambda$ consisting of a family of proper arcs which are homotopic through homotopies preserving $\partial X$ setwise. For instance, $\Lambda$ could consist of homotopic compact arcs from the boundary $\partial X$ to itself. We may view $\Lambda$ as an embedding of $A\times I$ where $A$ is compact and totally disconnected, and $I$ is an interval in $\R$ (possibly infinite) with $A\times \partial I$ mapping to $\partial X$. There is a transversal $\tau_0$ which intersects each leaf of $\Lambda$ exactly once. We have $\tau_0\cap \Lambda\cong A$ and thus $\mu\mapsto \mu_{\tau_0}$ defines a linear map $\mathcal M(\Lambda)\to \mathcal M(A)$. This map is a homeomorphism, since any other transversal to $\Lambda$ may be partitioned into sub-transversals which are homotopic to sub-transversals of $\tau_0$. Thus, any measure $\mu_\sigma$ is determined entirely by the measure $\mu_{\tau_0} \in \mathcal M(A)$. These laminations will turn up extensively in Section \ref{sec:inverselim}.
\end{ex}

\noindent In this section we prove Theorem \ref{mainthm:subcone} from the introduction.

\begin{proof}[Proof of Theorem \ref{mainthm:subcone}]
  Fix a family of transversals $\tau_1,\tau_2,\ldots$
  such that every leaf intersects at least one $\tau_j$.
  Recall that for a totally disconnected compact metrizable space $X$, $\mathcal K(X)$ denotes the set of clopen subsets of $X$.
  For each
  $\tau_j$, let $\mathcal K_j\vcentcolon=\mathcal K(\Lambda\cap\tau_j)$.
Sending a transverse measure $\mu$ to the restrictions $\mu_{\tau_i}$ defines a map
  $$\Phi:\mathcal M(\Lambda)\to \prod_j
  \mathcal M(\Lambda\cap\tau_j)\subseteq \prod_j \prod_{K\in\mathcal
    K_j}\R_+=\R_+^{\N}$$ which is linear and continuous. The image is
  contained in the sub-cone $C_\Lambda$ of $\prod_j
  \mathcal M(\Lambda\cap\tau_j)$ cut out by the following linear
  equations: $x_K=x_L$ whenever there are $K\in \mathcal K_i$ and $L\in \mathcal K_j$ and transversals $\sigma\subset \tau_i$ and $\tau\subset \tau_j$ with $K=\Lambda\cap\sigma$,
  $L=\Lambda\cap\tau$, such that $\sigma$ is homotopic to $\tau$.
  
  We now
  argue that $\Phi$ is a homeomorphism onto $C_\Lambda$. We utilize the following basic fact about homotopies. See e.g. \cite[Section 4.2]{calegari} for the argument.

\begin{lem}
\label{lem:flowboxes}
Let $\Lambda$ be a geodesic lamination on the hyperbolic surface $X$ of the first kind. Let $p_1,p_2$ be points lying on a common leaf of $\Lambda$. Let $\sigma_i$ be transversals to $\Lambda$ through the points $p_i$. Then there are sub-transversals $\sigma_i'\subset \sigma_i$ containing $p_i$ for $i=1,2$ such that $\sigma_1'$ is homotopic to $\sigma_2'$.
\end{lem}

  If $\mu,\mu'\in \mathcal M(\Lambda)$ with $\mu\neq \mu'$,
  then there is a transversal $\tau$ so that the
  induced measures $\mu_\tau$ and $\mu'_\tau$ are different.
  By uniqueness in the Carath\'{e}odory Extension Theorem,
  after replacing $\tau$ with a sub-transversal, we may assume that the total measures
  $\mu(\tau)$ and $\mu'(\tau)$ are different.
  Since every leaf of $\Lambda$ intersects some $\tau_i$, for each point $p\in \tau \cap \Lambda$ we may apply Lemma \ref{lem:flowboxes} to find a sub-transversal $\sigma\subset \tau$ containing $p$ which is homotopic into some $\tau_i$.
  By compactness of $\tau \cap \Lambda$, we can sub-divide $\tau$
  into finitely many sub-transversals each of which is homotopic to a
  sub-transversal of some $\tau_i$. It follows that for some $i$ the
  measures on $\Lambda\cap\tau_i$ induced by $\mu$ and $\mu'$ are
  distinct, showing that $\Phi$ is injective.

  Now, suppose we are given a point in the sub-cone $C_\Lambda$. This yields
  Borel measures on $\Lambda\cap\tau_i$ for every $i$, satisfying the
  homotopy invariance. If $\tau$ is an arbitrary transversal, we can
  sub-divide it as above into sub-transversals so that each is homotopic
  into some $\tau_i$, and we can pull back the measures on $\tau_i$ to get a
  measure on $\Lambda\cap\tau$.
  If $\tau=\sigma_1 \cup \ldots\cup \sigma_r$ and $\tau=\sigma'_1\cup \ldots \cup \sigma'_s$ are two different such partitions of $\tau$ into sub-transversals, then we may consider their common refinement $\tau=\bigcup_{i,j} (\sigma_i\cap \sigma'_j)$. Using the equations defining $C_\Lambda$, we see that the measure on $\sigma_i\cap \sigma'_j$, and thus on $\tau$, is independent of the partition. Independence of the partition yields invariance of the constructed measure under homotopies and passing to sub-transversals.
  This shows that the image of $\Phi$
  is the entire sub-cone $C_\Lambda$.

  Finally, we argue that $\Phi^{-1}:C_\Lambda\to \mathcal M(\Lambda)$
  is continuous. By the definition of the topology on $\mathcal
  M(\Lambda)$, it suffices to argue that the composition of
  $\Phi^{-1}$ with the restriction to $\mathcal M(\Lambda\cap\tau)$ is
  continuous, for every transversal $\tau$. When $\tau=\tau_i$ for some $i$
  this is just a coordinate projection, so it is continuous. For an
  arbitrary $\tau$, sub-divide and reduce to sub-transversals of
  $\tau_i$'s as above.
\end{proof}

\begin{cor}
\label{cor:baseexistence}
  Let $\Lambda$ be a geodesic lamination on a complete hyperbolic
  surface $X$ of the first kind. If there is a compact subsurface of $X$ that
  intersects every leaf of $\Lambda$ then $\mathcal M(\Lambda)$ admits
  a base.
\end{cor}

\begin{proof}
  In this case we can choose a finite collection of transversals that
  intersect every leaf.
  Thus, the product $\prod_j \mathcal M(\Lambda \cap \tau_j)$ is finite and each factor has its compact base of probability measures.
  For convex cones the property of having a
  base passes to finite products and closed sub-cones.
\end{proof}

\section{$\mathcal M(\Lambda)$ as an inverse limit}
\label{sec:inverselim}

As before, $X$ is hyperbolic of the
first kind and $\Lambda\subset X$ is a geodesic lamination. Since $X$ is of the first kind, we may fix an
exhaustion $$X_1\subset X_2\subset\ldots$$ of $X$ where each $X_i$ is
a finite area complete subsurface with totally geodesic boundary (see \cite{complete}).
Thus, $X_i$ is a compact surface with boundary minus finitely many points. We will sometimes refer to such surfaces as \emph{punctured compact subsurfaces}. In fact the proof of \cite[Proposition 3.1]{complete} shows that any exhaustion of $X$ by finite type subsurfaces straightens to an exhaustion by complete finite area subsurfaces with geodesic boundary. So we may blur the distinction between topological exhaustions and exhaustions by complete finite area subsurfaces with geodesic boundary.

We will assume for convenience that the boundary components of $X_i$ are transverse to $\Lambda$. This may be achieved as follows. Suppose that we have constructed a sequence $Y_1\subset Y_2\subset \ldots \subset Y_n$ of punctured compact subsurfaces such that $Y_i$ contains $X_i$ and $\partial Y_i$ is transverse to $\Lambda$ for each $i\leq n$. Choose $m$ large enough that $X_m$ contains both $Y_n$ and $X_{n+1}$. Choose $p>m$ large enough that $X_p$ contains $\partial X_m$ in its interior and $q>p$ large enough that $X_q$ contains $\partial X_p$ in its interior. The components of $\partial X_p$ are contained in the interior of $X_q\setminus X_m$. We may apply a mapping class $f$ supported on the components of $X_q\setminus X_m$, so that for any component $c$ of $\partial X_p$, $f(c)$ intersects $\Lambda$ transversely (if at all). Then setting $Y_{n+1}=f(X_p)$ yields $Y_n\subset Y_{n+1}$, $X_{n+1}\subset Y_{n+1}$, and $\partial Y_{n+1}$ is transverse to $\Lambda$.

We will need the following structure theorem for $\Lambda\cap
X_i$.
Recall that the \emph{support} of a transverse measure consists of the points $p$ such that every transversal $\tau$ containing $p$ has positive measure.

\begin{prop}
  Consider the lamination $\Lambda_i=\Lambda\cap X_i$. It has a sub-lamination consisting of finitely many compact minimal sub-laminations contained in the interior of $X_i$, plus finitely many parallel families $A\times I$
  of proper arcs, with $A$ compact and totally disconnected and $I$ a closed sub-interval of $\R$. Any transverse measure on $\Lambda_i$ is supported on this sub-lamination. Any leaf of
  $\Lambda_i$ that does not belong to this sub-lamination accumulates
  on one or more of the compact minimal sub-laminations of $\Lambda_i$.
\end{prop}

\begin{proof}
The lamination $\Lambda_i$ contains leaves of three possible types:
\begin{enumerate}[(1)]
\item arcs which on each end either (a) intersect a boundary component of $X_i$ or (b) are asymptotic to a puncture of $X_i$;
\item simple closed geodesics and bi-infinite geodesics which are contained in a compact minimal sub-lamination in the interior of $X_i$;
\item rays and bi-infinite geodesics which accumulate onto minimal sub-laminations on at least one end but are not contained in these minimal sub-laminations (we will also say these geodesics \emph{spiral} onto the minimal sub-laminations).
\end{enumerate}

\begin{figure}[h]

\centering
\def\svgwidth{0.7\textwidth}
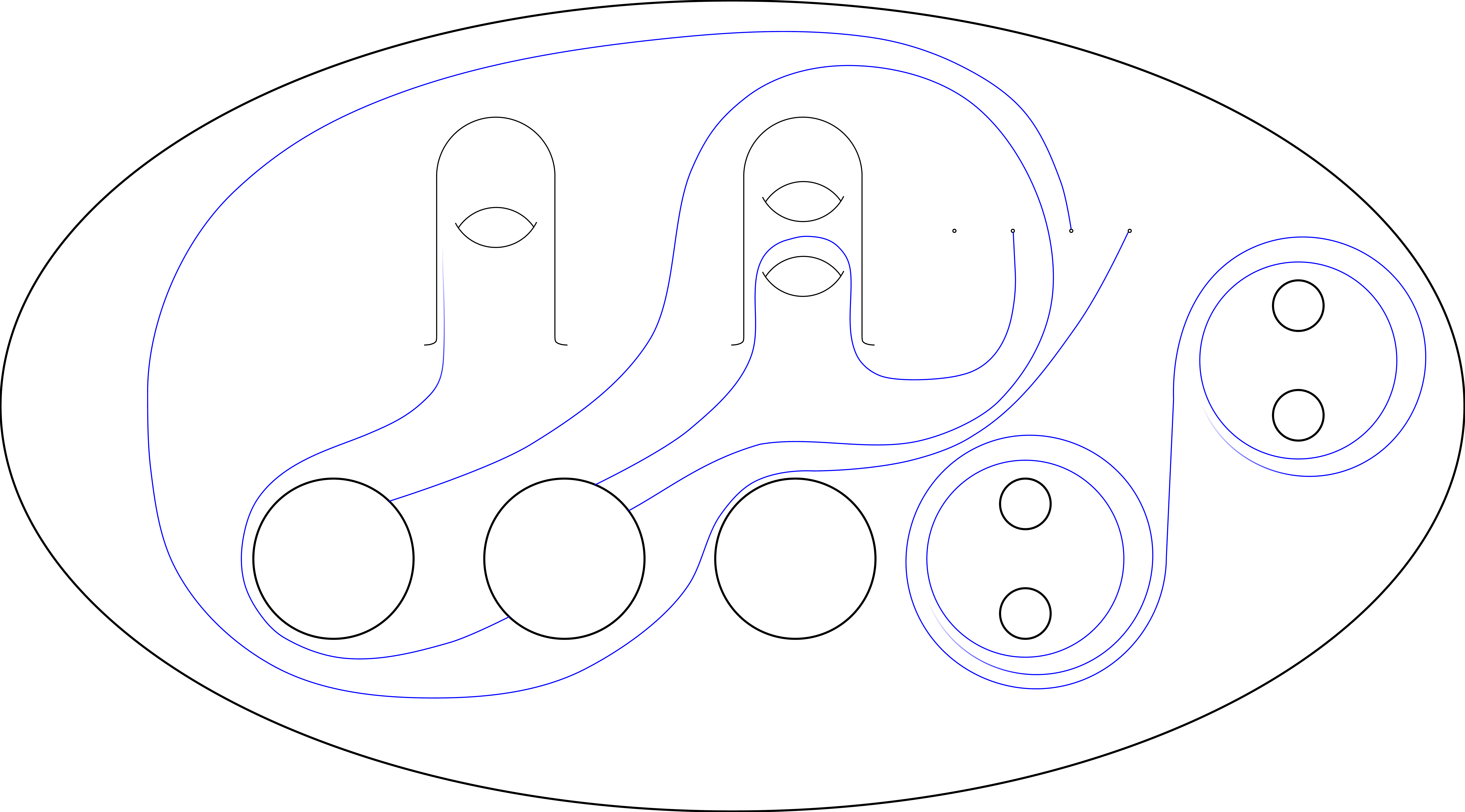

\caption{The various geodesics of $\Lambda_i$. Boundary components of $X_i$ are denoted by thick black lines while punctures are denoted by small circles. Geodesics of $\Lambda_i$ are indicated by blue lines. In this case there are three minimal sub-laminations and three homotopy classes of arcs. Homotopy classes $\ell_j^i$ are drawn as though they consist of a single arc for ease of presentation. Other geodesics of $\Lambda_i$ accumulate onto compact minimal sub-laminations $\Gamma_j^i$ on at least one end.}
\label{fig:geodesics}
\end{figure}

See Figure \ref{fig:geodesics}. There are finitely many arcs of type (1) up to homotopies preserving the boundary components of $X_i$ setwise. Moreover, removing the leaves of types (1) and (3) and applying the classification theorem for laminations on finite type surfaces \cite[Theorem I.4.2.9]{fundamentals} yields that there are finitely minimal sub-laminations in the interior. Considering all of the arcs of type (1) in a single homotopy class yields a clopen subset of leaves of $\Lambda_i$ which is homeomorphic to $A\times I$ for some closed sub-interval $I$ of $\R$ (possibly all of $\R$ or a ray). The fact that $A$ is totally disconnected follows from Theorem \ref{thm:nowheredense}. This proves the first claim.

Now we show that any transverse measure to $\Lambda_i$ is supported on the union of the leaves of types (1) and (2). Let $\mu$ be a transverse measure to $\Lambda_i$ and consider a leaf $L$ which accumulates onto a compact minimal sub-lamination $\Gamma$ of $\Lambda_i$ but is not contained in $\Gamma$. There is a transversal $\tau$ through $L$ and a direction such that all the rays of $\Lambda_i$ through $\tau$ in this direction are asymptotic, accumulate onto $\Gamma$, and never return to $\tau$.
To see the existence of such a $\tau$, we take $S$ to be one of the following surfaces: (i) if $\Gamma$ is a closed geodesic then $S$ is a collar neighborhood of $\Gamma$ small enough that any geodesic that intersects $S$ and is disjoint from $\Gamma$ spirals onto $\Gamma$; (ii) if $\Gamma$ is not a closed geodesic then $S$ is the surface filled by $\Gamma$. We may take $\tau$ to lie inside of $S\setminus \Gamma$ and then all of the leaves of $\Lambda_i$ through $\tau$ spiral onto $\Gamma$. By taking $\tau$ even smaller if necessary, such leaves never return to $\tau$. Taking $\tau$ smaller again, such leaves all exit the same cusp of $S\setminus \Gamma$ in case (ii) and they are all asymptotic, in either case.
Consider a transversal $\sigma$ which intersects $\Gamma$. Considering a point $p\in \sigma\cap \Gamma$, we see that $L$ intersects $\sigma$ in infinitely many points limiting to $p$. We see that we may homotope $\tau$ to infinitely many disjoint sub-intervals of $\sigma$. Since $\mu(\sigma)<\infty$, we must have $\mu(\tau)=0$. This completes the proof.
\end{proof}

We denote by $\mathcal M(\Lambda_i)$ the cone of transverse measures
to $\Lambda_i$. By Example \ref{ex:product}, we can write $\mathcal M(\Lambda_i)$ as a finite
product $\prod_{\Gamma}\mathcal M(\Gamma)\times \prod_{A\times I}\mathcal M(A)$
where $\Gamma$ ranges over the compact minimal sub-laminations of $\Lambda_i$ and $A\times I$ over the
parallel families of proper arcs. There is an associated cone
$C_i\vcentcolon = \prod_{\Gamma}\mathcal M(\Gamma)\times \prod_{A\times I}\R_+$, which is the
quotient of $\mathcal M(\Lambda_i)$ obtained by identifying all
measures on $A\times I$ with the same total mass. Each cone $\mathcal M(\Gamma)$ is finite-dimensional (see e.g. \cite[Section 1.9.1]{calegari}) and thus $C_i$ is a finite-dimensional simplicial cone. We sometimes denote $C_i$ by $C(X_i)$ to make the dependence on the surface $X_i$ clear.

The situation is summarized in the following commutative diagram, where $\mathcal W(\Lambda)$ is the
inverse limit of the bottom row.
The maps $\Psi_i:\mathcal M(\Lambda_i)\to C_i$ are the quotient maps just defined.
The horizontal arrows $\rho_i$ on the top are
restriction maps and on the bottom $\pi_i$ are the induced maps on the
quotient cones. One may check that the maps $\pi_i$ are linear, since $\Psi_{i+1},\rho_i,$ and $\Psi_i$ are linear. The map $\Psi$ is $(\Psi_1,\Psi_2,\ldots)$.

\begin{center}
\begin{tikzcd}
\mathcal M(\Lambda_1) \arrow[d,"\Psi_1"] & \mathcal M(\Lambda_2)
\arrow[l,"\rho_1" above] \arrow[d,"\Psi_2"] & \mathcal M(\Lambda_3)
\arrow[l,"\rho_2" above]  \arrow[d,"\Psi_3"] & \arrow[l,"\rho_3" above]
\ldots & \mathcal M(\Lambda)\arrow[d,"\Psi"] \\ C_1 & C_2
\arrow[l,"\pi_1"]  & C_3 \arrow[l,"\pi_2"] & \ldots  \arrow[l,"\pi_3"]
& \mathcal W(\Lambda)
\end{tikzcd}
\end{center}

\noindent We now state the main theorem of this section. Theorem \ref{mainthm:inverselim} from the introduction will follow immediately from it.

\begin{thm}
\label{thm:limithomeo}
The map $\Psi:\mathcal M(\Lambda) \to \mathcal W(\Lambda)$ is a linear homeomorphism.
\end{thm}

The proof has the following outline: (1) $\mathcal M(\Lambda)$ is the inverse limit of $\mathcal M(\Lambda_i)$. (2) All vertical maps $\Psi_i$ are proper and surjective. (3) Consequently, $\Psi$ is proper and surjective. (4) Since $X$ is of the first kind, $\Psi$ is injective. (5) Consequently, $\Psi$ is a homeomorphism.

Fact (1) follows from the definitions and (2) is a consequence of the
Banach-Alaoglu Theorem. Then (3) follows by a diagram chase. The main
thing to be proved is (4).
Before giving the full proof we pause to consider the cones $C_i$, the transition maps $\pi_i$, and some examples of cones of measures that can be characterized using Theorem \ref{mainthm:inverselim}.

\subsection{Cones of weights and transition maps}
\label{sec:transmaps}

We pause to give a more complete and intuitive description of the maps $\pi_n$. Denote by $\ell_1^n,\ldots,\ell_{r(n)}^n$ the homotopy classes of proper arcs in $\Lambda_n$ and by $\Gamma_1^n,\ldots,\Gamma_{s(n)}^n$ the compact minimal sub-laminations contained in the interior of $X_n$. Thus, the arcs in the homotopy class $\ell_i^n$ form some parallel family $A^n_i\times I_i^n$ where $A^n_i$ is compact, totally disconnected, and metrizable, and $I_i^n$ is a closed (possibly infinite) interval in $\R$. There exists $s_0\geq 0$ such that $\Gamma_1^{n+1},\ldots,\Gamma_{s_0}^{n+1}$ all intersect $X_n$ in some (possibly empty) collection of arcs, while $\Gamma_{s_0+1}^{n+1},\ldots,\Gamma_{s(n)}^{n+1}$ are all contained in $X_n$. We have $\mathcal M(\Lambda_n)=\prod_{i=1}^{r(n)} \mathcal M(A^n_i) \times \prod_{i=1}^{s(n)} \mathcal M(\Gamma_i^n)$ and $C_n=\prod_{i=1}^{r(n)} \mathcal \R_+ \times \prod_{i=1}^{s(n)} \mathcal M(\Gamma_i^n)$. Let $e_j^n$ be the basis element 1 in the $j$-th factor $\R_+$ in $\prod_{i=1}^{r(n)} \R_+$. Then we may write an element of $C_n$ as \[w=\sum_{i=1}^{r(n)} b_i^n e_i^n + \sum_{i=1}^{s(n)} \nu_i^n\] where $b_i^n\geq 0$ and $\nu_i^n \in \mathcal M(\Gamma_i^n)$ for each $i$. If $\nu \in \mathcal M(\Lambda_n)$ then $\Psi_n(\nu)=\sum_i b_i^n e_i^n + \sum_i \nu_i^n$ where (1) $b_i^n$ is the measure $\nu(\tau_i^n)$ of a transversal $\tau_i^n$ which intersects each arc of $\ell_i^n$ exactly once and is disjoint from $\Lambda_n \setminus \ell_i^n$; and (2) $\nu_i^n$ is the restriction of $\nu$ to $\Gamma_i^n$: $\nu_i^n\vcentcolon= \nu|\Gamma_i^n$. We think of an element of $C_n$ as a \emph{weight}, assigning a number to each homotopy class of arcs $\ell_i^n$ and a transverse measure to each minimal lamination $\Gamma_i^n$. We will refer to $C_n$ as the \emph{cone of weights} for $\Lambda_n$. Finally, define $\tau_i^{n+1}$ for $1\leq i\leq r(n+1)$ to be a transversal intersecting each arc of $\ell_i^{n+1}$ exactly once and disjoint from $\Lambda_{n+1}\setminus \ell_i^{n+1}$.

For $1\leq j\leq r(n+1)$, choose $L$ to be any arc in $\ell_j^{n+1}$. For $1\leq i\leq r(n)$ we denote by $a_{ij}$ the number of arcs of $L\cap X_n$ which are homotopic to $\ell_i^n$. Thus, $A_j^{n+1} \times I_j^{n+1}$ passes through $A_i^n\times I_i^n$ exactly $a_{ij}$ times. Informally, we will say that $\ell_j^{n+1}$ \emph{traverses} $\ell_i^n$ $a_{ij}$ times.
From this, we see that we may partition $\tau_i^n$ into sub-transversals, $a_{ij}$ of which are homotopic to $\tau_j^{n+1}$ for each $1\leq j\leq r(n+1)$, and the remaining of which are disjoint from $\ell_1^{n+1}\cup \ldots \cup \ell_{r(n+1)}^{n+1}$.
The sub-transversals of $\tau_i^n$ which are disjoint from $\ell_1^{n+1}\cup \ldots\cup \ell_{r(n+1)}^{n+1}$ intersect the various minimal laminations $\Gamma_j^{n+1}$ for $1\leq j\leq s_0$ and leaves which spiral onto such $\Gamma_j^{n+1}$, but are otherwise disjoint from $\Lambda_{n+1}$. Therefore if $\nu\in \mathcal M(\Lambda_{n+1})$ then \[\nu(\tau_i^n)= \sum_{j=1}^{r(n+1)} a_{ij} \nu(\tau_j^{n+1})+\sum_{j=1}^{s_0} \left(\nu|\Gamma_j^{n+1}\right)(\tau_i^n).\] Putting this together yields: if $w=\sum_{j=1}^{r(n+1)}b_j^{n+1}e_j^{n+1}+\sum_{j=1}^{s(n+1)} \nu_j^{n+1}$ then \[\pi_n(w)= \sum_{i=1}^{r(n)} \sum_{j=1}^{r(n+1)} a_{ij}b_j^{n+1} e_i^n  +\sum_{i=1}^{r(n)} \sum_{j=1}^{s_0} \nu_j^{n+1}(\tau_i^n)e_i^n+ \sum_{j=s_0+1}^{s(n+1)} \nu_j^{n+1}\] noting that for $s_0+1\leq j\leq s(n+1)$, $\nu_j^{n+1}$ lies in $C_n$ since $\Gamma_j^{n+1}$ is contained in $X_n$.

The easiest case to understand is when $\Lambda_n$ and $\Lambda_{n+1}$ contain no compact minimal sub-laminations $\Gamma_*^*$. Then $C_n=\R_+^{r(n)}$, $C_{n+1}=\R_+^{r(n+1)}$, and $\pi_n(w)=\sum_{i=1}^{r(n)} \sum_{j=1}^{r(n+1)} a_{ij} b_j^{n+1} e_i^n$. Thus, $\pi_n$ is represented by the $r(n) \times r(n+1)$ matrix $(a_{ij})_{i=1,j=1}^{r(n),r(n+1)}$. 

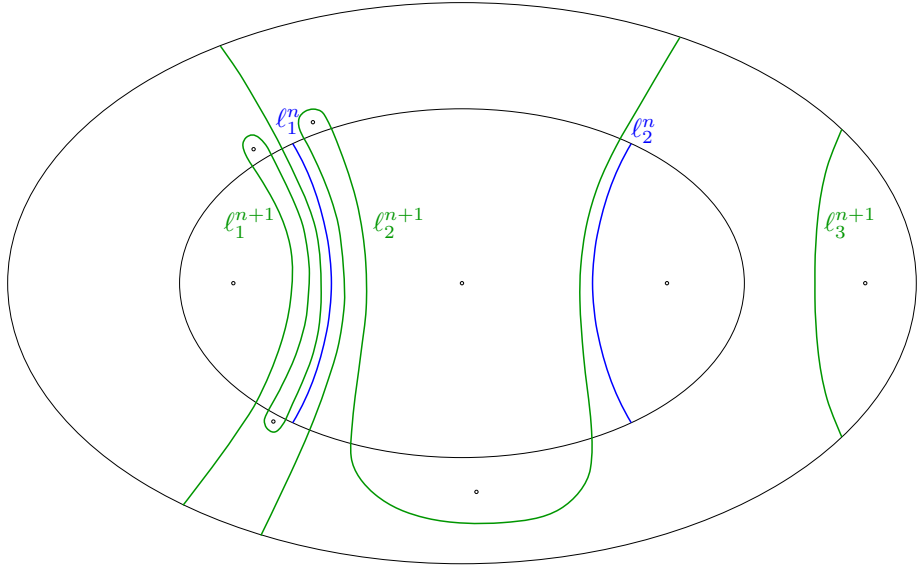
\begin{figure}[h]

\centering
\def\svgwidth{0.8\textwidth}
\begingroup%
  \makeatletter%
  \providecommand\color[2][]{%
    \errmessage{(Inkscape) Color is used for the text in Inkscape, but the package 'color.sty' is not loaded}%
    \renewcommand\color[2][]{}%
  }%
  \providecommand\transparent[1]{%
    \errmessage{(Inkscape) Transparency is used (non-zero) for the text in Inkscape, but the package 'transparent.sty' is not loaded}%
    \renewcommand\transparent[1]{}%
  }%
  \providecommand\rotatebox[2]{#2}%
  \newcommand*\fsize{\dimexpr\f@size pt\relax}%
  \newcommand*\lineheight[1]{\fontsize{\fsize}{#1\fsize}\selectfont}%
  \ifx\svgwidth\undefined%
    \setlength{\unitlength}{1708.23595903bp}%
    \ifx\svgscale\undefined%
      \relax%
    \else%
      \setlength{\unitlength}{\unitlength * \real{\svgscale}}%
    \fi%
  \else%
    \setlength{\unitlength}{\svgwidth}%
  \fi%
  \global\let\svgwidth\undefined%
  \global\let\svgscale\undefined%
  \makeatother%
  \begin{picture}(1,0.61835091)%
    \lineheight{1}%
    \setlength\tabcolsep{0pt}%
    \put(0,0){\includegraphics[width=\unitlength,page=1]{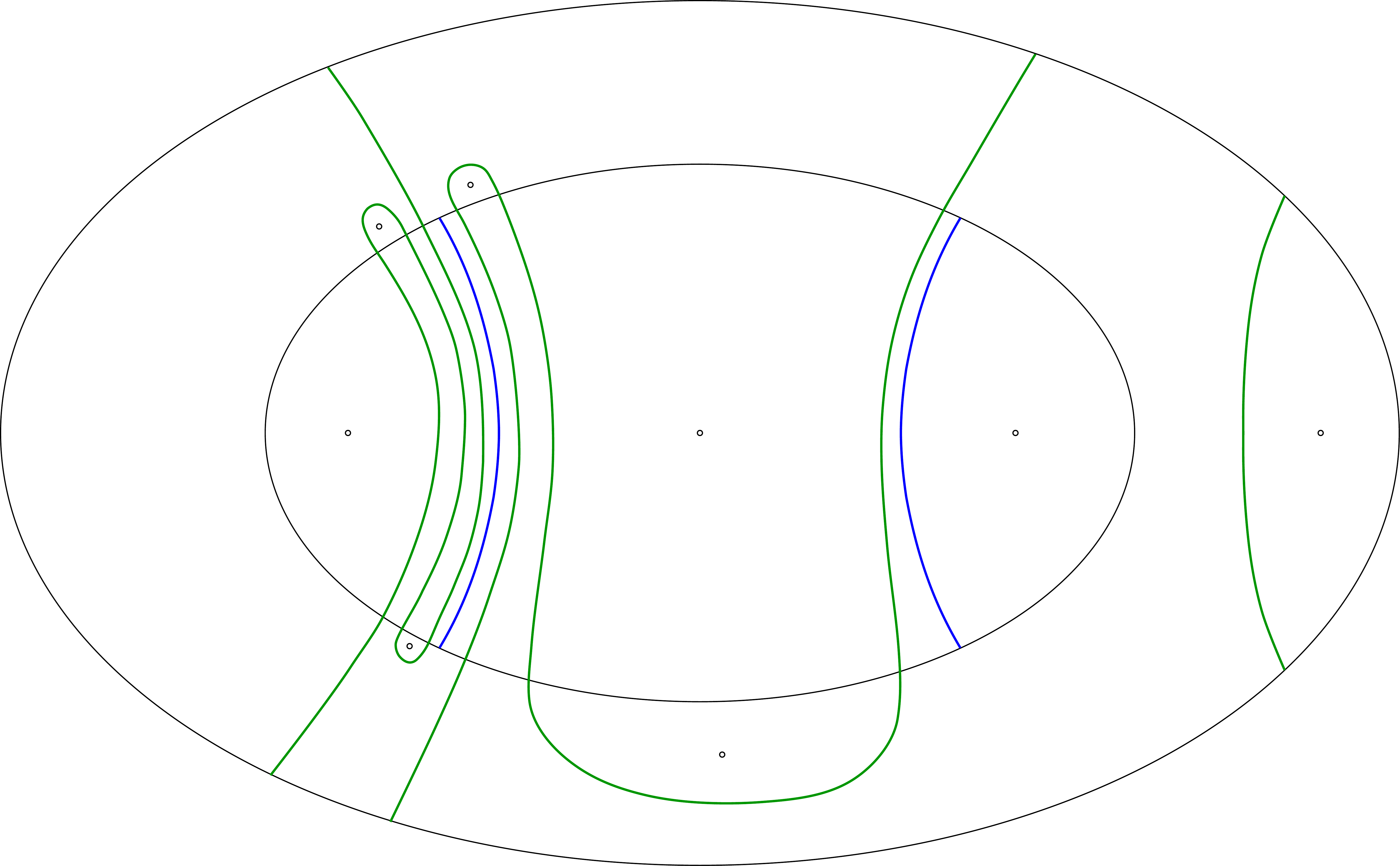}}%
    \put(0.29533628,0.48101557){\color[rgb]{0,0,1}\makebox(0,0)[lt]{\lineheight{1.25}\smash{\begin{tabular}[t]{l}$\ell_1^n$\end{tabular}}}}%
    \put(0.68609867,0.47040697){\color[rgb]{0,0,1}\makebox(0,0)[lt]{\lineheight{1.25}\smash{\begin{tabular}[t]{l}$\ell_2^n$\end{tabular}}}}%
    \put(0.23828804,0.36613759){\color[rgb]{0,0.58823529,0}\makebox(0,0)[lt]{\lineheight{1.25}\smash{\begin{tabular}[t]{l}$\ell_1^{n+1}$\end{tabular}}}}%
    \put(0.40241777,0.36613759){\color[rgb]{0,0.58823529,0}\makebox(0,0)[lt]{\lineheight{1.25}\smash{\begin{tabular}[t]{l}$\ell_2^{n+1}$\end{tabular}}}}%
    \put(0.89825994,0.36613759){\color[rgb]{0,0.58823529,0}\makebox(0,0)[lt]{\lineheight{1.25}\smash{\begin{tabular}[t]{l}$\ell_3^{n+1}$\end{tabular}}}}%
  \end{picture}%
\endgroup%

\caption{The figure illustrates two punctured disks $X_n$ (the smaller punctured disk) contained inside $X_{n+1}$ (the larger punctured disk). There are two homotopy classes of arcs on $X_n$, $\ell_i^n$, and three homotopy classes of arcs on $X_{n+1}$, $\ell_j^{n+1}$.}

\label{fig:transitionmap}
\end{figure}

\begin{ex}
Consider the punctured disks $X_n$ and $X_{n+1}$ pictured in Figure \ref{fig:transitionmap}. There are two homotopy classes of arcs $\ell_1^n,\ell_2^n$ on $X_n$ and three homotopy classes of arcs $\ell_1^{n+1},\ell_2^{n+1},\ell_3^{n+1}$ on $X_{n+1}$. The class $\ell_1^{n+1}$ traverses $\ell_1^n$ three times, the class $\ell_2^{n+1}$ traverses $\ell_1^n$ twice and $\ell_2^n$ once, while $\ell_3^{n+1}$ doesn't traverse $\ell_1^n$ or $\ell_2^n$. Thus, $\pi_n$ is represented by the $2\times 3$ matrix $\begin{pmatrix} 3 & 2 & 0 \\ 0 & 1 & 0 \end{pmatrix}$.
\end{ex}

\subsection{Examples of cones of measures}
\label{sec:examplecones}

In this section we use Theorem \ref{mainthm:inverselim} to give explicit descriptions of cones of transverse measures for certain examples of geodesic laminations.

\begin{figure}[h]
\centering

\begin{subfigure}[t]{0.8\textwidth}
\def\svgwidth{\textwidth}
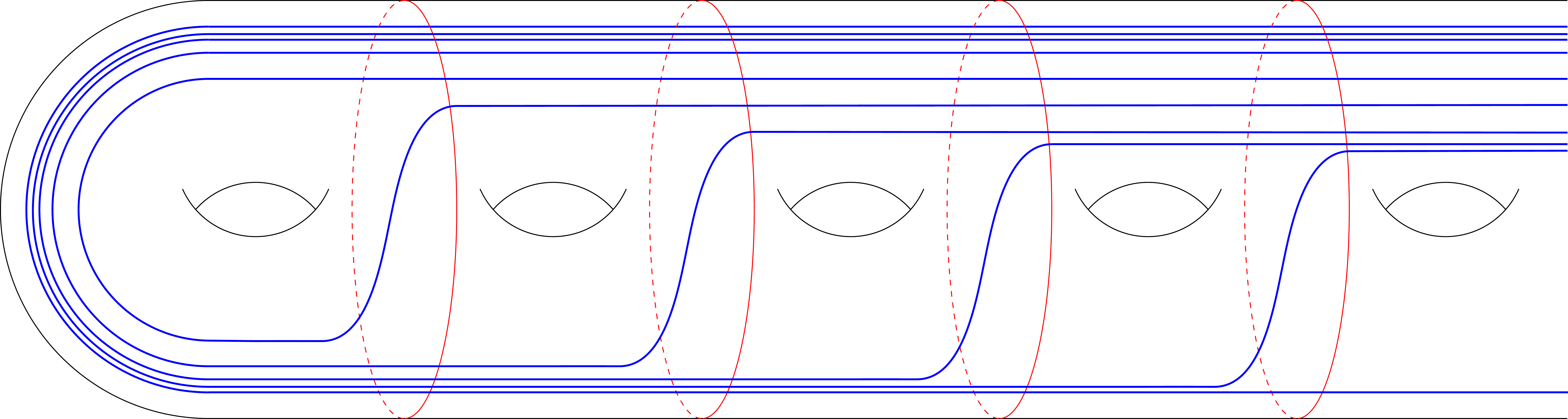
\caption{The lamination from Example \ref{ex:choquetexample1} is a union of countably many isolated proper leaves which limit to a single non-isolated proper leaf.}
\label{fig:choquetexample1}
\end{subfigure} 

\begin{subfigure}[t]{0.8\textwidth}
\def\svgwidth{\textwidth}
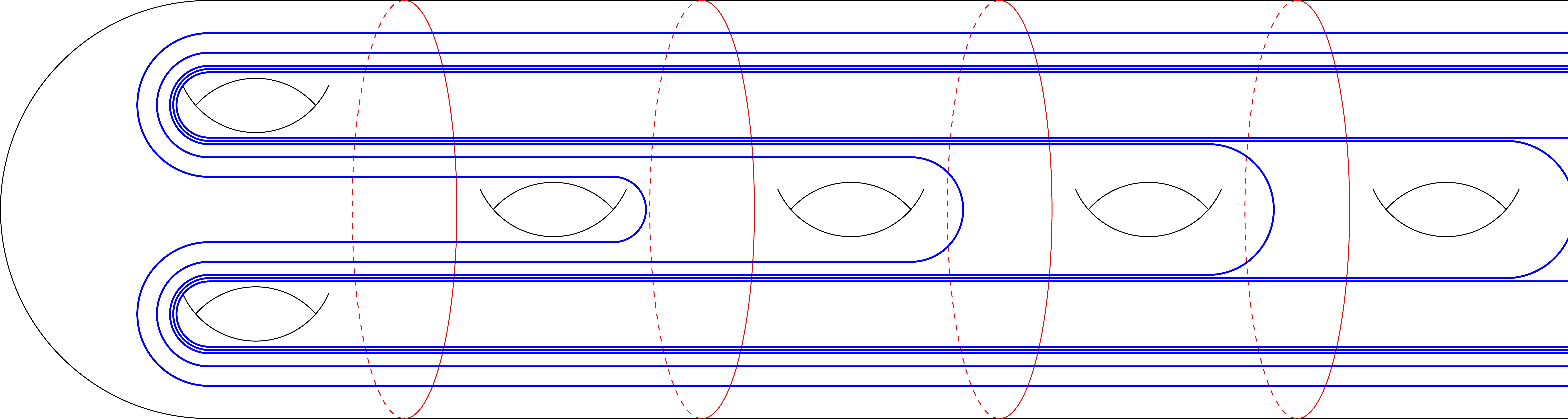
\caption{The lamination from Example \ref{ex:choquetexample2} is a union of countably many isolated proper leaves which limit to two non-isolated proper leaves.}
\label{fig:choquetexample2}
\end{subfigure} 

\begin{subfigure}[t]{0.8\textwidth}
\def\svgwidth{\textwidth}
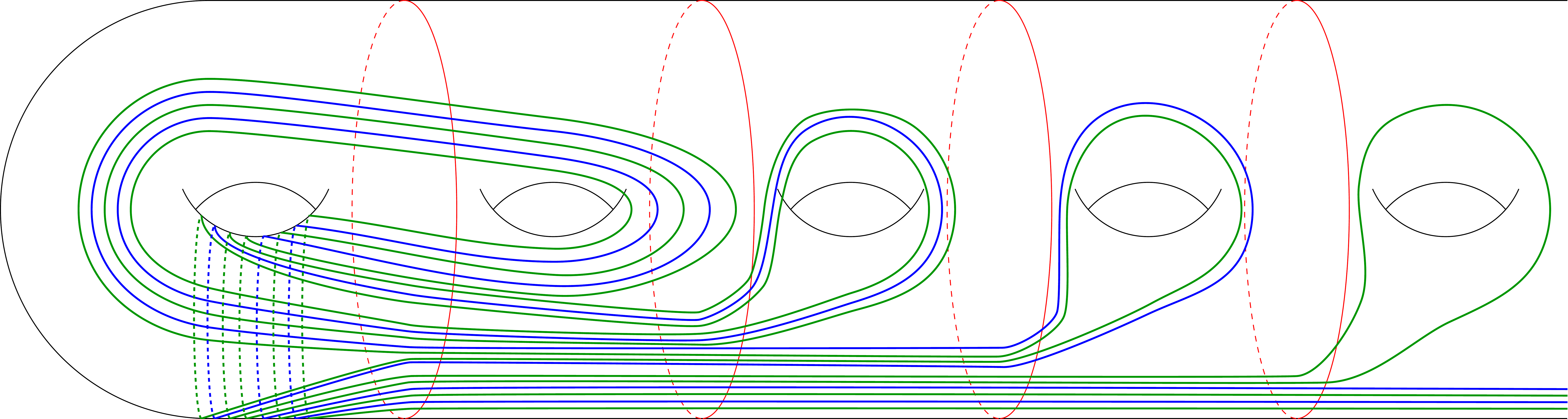
\caption{The lamination from Example \ref{ex:choquetexample3} is the closure of the two pictured non-proper leaves.}
\label{fig:choquetexample3}
\end{subfigure}

\caption{The laminations from Examples \ref{ex:choquetexample1}-\ref{ex:choquetexample3}.}
\end{figure}

\begin{ex}
\label{ex:choquetexample1}
Consider the lamination $\Lambda$ in Figure \ref{fig:choquetexample1}. Thus $\Lambda$ consists of a countable collection of isolated proper leaves $L_i$ that converge to a single proper leaf $L$ (which is not isolated). Recall that a leaf is \emph{isolated} when it has an open neighborhood disjoint from the rest of the lamination. There is a transverse arc $\tau$ intersecting each leaf exactly once and one may check that $\mathcal M(\Lambda)$ is linearly homeomorphic to $\mathcal M(\Lambda \cap \tau)$. We verify this using inverse limits.

An exhaustion $\{X_n\}$ is given by the surfaces bounded by the red curves. Thus $X_n$ has genus $n$ and one boundary component. There are $n$ homotopy classes of arcs $\ell_1^n,\ldots,\ell_n^n$ on $X_n$. Moreover, choosing the numbering correctly, $\ell_i^{n+1} \cap X_n$ is homotopic to $\ell_i^n$ for $1\leq i\leq n$ whereas $\ell_{n+1}^{n+1}\cap X_n$ is homotopic to $\ell_n^n$. Thus $\mathcal W(\Lambda)$ is the inverse limit of $\R_+\xleftarrow{\pi_1} \R_+^2 \xleftarrow{\pi_2} \R_+^3 \xleftarrow{\pi_3} \ldots$ where \[\pi_1=\begin{pmatrix} 1 & 1 \end{pmatrix}, \ \ \ \pi_2= \begin{pmatrix} 1 & 0 & 0 \\ 0 & 1 & 1 \end{pmatrix}, \ \ \ \pi_3= \begin{pmatrix} 1 & 0 & 0 & 0 \\ 0 & 1 & 0 & 0 \\ 0 & 0 & 1 & 1 \end{pmatrix}, \ \ \ \pi_4= \begin{pmatrix} 1 & 0 & 0 & 0 & 0 \\ 0 & 1 & 0 & 0 & 0 \\ 0 & 0 & 1 & 0 & 0 \\ 0 & 0 & 0 & 1 & 1 \end{pmatrix}, \ \ \ \ldots.\] We claim that $\mathcal M(\Lambda)\cong \mathcal W(\Lambda)$ is linearly homeomorphic to the cone \[C\subset \ell^1 \text{ defined by } C=\left\{(x,y_1,y_2,\ldots) : y_i\geq 0 \text{ for all } i \text{ and } x\geq \sum_{i=1}^\infty y_i\right\}\] where the space $\ell^1$ of summable sequences is endowed with its weak${}^*$ topology as the dual of the space $c_0$ of sequences convergent to 0, and $C$ is endowed with the subspace topology. We outline the proof. An element of the inverse limit $\mathcal W(\Lambda)$ has the form \[\overleftarrow{x}=\left( \begin{pmatrix} x_0^0\end{pmatrix}, \begin{pmatrix} x_1^1 \\ x_2^1 \end{pmatrix}, \begin{pmatrix} x_1^2 \\ x_2^2 \\ x_3^2\end{pmatrix}, \ldots \right)\] where \[x_0^0=x_1^1+x_2^1=x_1^1+x_2^2+x_3^2=x_1^1+x_2^2+x_3^3+x_4^3=\ldots \text{ and } x_n^n=x_n^{n+1}=x_n^{n+2}=\ldots \text{ for }n\geq 1.\] Thus, the element is determined by the sequence of non-negative numbers $(x_0^0,x_1^1,x_2^2,\ldots)$. Moreover, since $x_0^0\geq \sum_{i=1}^k x_i^i$ for each $k\geq 0$ we have that $\sum_{i=1}^\infty x_i^i\leq x_0^0<\infty$. Define a function $\mathcal W(\Lambda)\to C$ by sending $\overleftarrow{x}$ to $(x_0^0,x_1^1,\ldots)$, which lies in $C$. The inverse is given by sending an element $(x,y_1,y_2,\ldots)\in C$ to the element \[\left( \begin{pmatrix} x \end{pmatrix}, \begin{pmatrix} y_1 \\ x-y_1 \end{pmatrix}, \begin{pmatrix} y_1 \\ y_2 \\ x-y_1-y_2\end{pmatrix}, \ldots\right).\] One may check that these functions are linear and that the function $C\to \mathcal W(\Lambda)$ is continuous since the coordinate functions $x$ and $y_n$ are continuous. One may check that $\mathcal W(\Lambda)\to C$ is continuous by using the fact that $\ell^1$ is equipped with its weak${}^*$ topology from its predual $c_0$.
\end{ex}

\begin{ex}
\label{ex:choquetexample2}

Consider the lamination $\Lambda$ pictured in Figure \ref{fig:choquetexample2}. Thus $\Lambda$ consists of a countable collection of isolated proper leaves $L_i$ that converge to the union of two disjoint proper leaves $L$ and $L'$ (neither of which is isolated). An exhaustion is given by the surfaces bounded by the red curves again. Thus $X_n$ has genus $n+1$ and there are $n+1$ homotopy classes of arcs $\ell_1^n,\ldots,\ell_{n+1}^n$ on $X_n$. The numbering can be chosen so that $\ell_i^{n+1}\cap X_n$ is homotopic to $\ell_i^n$ for $1\leq i\leq n+1$ and so that $\ell_{n+2}^{n+1}\cap X_n$ is homotopic to the union of $\ell_1^n$ and $\ell_2^n$. Thus, $\mathcal W(\Lambda)$ is the inverse limit of $\R_+^2\xleftarrow{\pi_1} \R_+^3 \xleftarrow{\pi_2} \R_+^4 \xleftarrow{ \pi_3} \ldots$ where \[\pi_1=\begin{pmatrix} 1 & 0 & 1 \\ 0 & 1 & 1\end{pmatrix}, \ \ \ \pi_2=\begin{pmatrix} 1 & 0 & 0 & 1 \\ 0 & 1 & 0 & 1 \\ 0 & 0 & 1 & 0\end{pmatrix}, \ \ \ \pi_3=\begin{pmatrix} 1 & 0 & 0 & 0 & 1 \\ 0 & 1 & 0 & 0 & 1 \\ 0 & 0 & 1 & 0 & 0 \\ 0 & 0 & 0 & 1 & 0\end{pmatrix}, \ \ \ \ldots. \] An element of the inverse limit has the form \[\left(\begin{pmatrix} x_1^0 \\ x_2^0 \end{pmatrix}, \begin{pmatrix} x_1^1 \\ x_2^1 \\ x_3^1\end{pmatrix}, \begin{pmatrix} x_1^2 \\ x_2^2 \\ x_3^2 \\ x_4^2 \end{pmatrix},\ldots\right).\] where \[x_1^0=x_1^1 +x_3^1=x_1^2+x_3^1+x_4^2=x_1^3+x_3^1+x_4^2+x_5^3=\ldots\] and \[x_2^0=x_2^1+x_3^1=x_2^2+x_3^1+x_4^2=x_2^3+x_3^1+x_4^2+x_5^3=\ldots\] and $x_{n+2}^n=x_{n+2}^{n+1}=x_{n+2}^{n+2}=\ldots$ for $n\geq 1$. Using similar techniques as in Example \ref{ex:choquetexample1} one may show that $\mathcal M(\Lambda)\cong \mathcal W(\Lambda)$ is isomorphic to the cone \[C\subset \ell^1 \text{ defined by } C=\left\{ (x_1,x_2,y_1,y_2,y_3,y_4,\ldots) : y_i \geq 0 \text{ for all } i \text{ and } x_1\geq \sum y_i \text{ and } x_2 \geq \sum y_i \right\}\] where again $\ell^1$ is endowed with its weak${}^*$ topology as the dual of $c_0$.
\end{ex}

\begin{ex}
\label{ex:choquetexample3}
Consider the lamination $\Lambda$ pictured in Figure \ref{fig:choquetexample3}. The figure shows (the beginnings of) two leaves of the lamination. The lamination $\Lambda$ is the closure of these two leaves. An exhaustion is given by the surfaces $X_n$ bounded by the red curves. Observe that any leaf of the closure $\Lambda$ intersected with $X_n$ is homotopic to a component of the intersection of either the blue leaf or the green leaf with $X_n$. Thus, to describe the cone $\mathcal W(\Lambda)$ it suffices to study the intersections of the green leaf and the blue leaf with the compact subsurfaces $X_n$. We refer the reader also to Figure \ref{fig:choquetconstruction}. This illustrates the construction of the blue and green leaves from Figure \ref{fig:choquetexample3}. The reader may continue the construction inductively. Moreover, $\Lambda \cap X_n$ consists of two homotopy classes $\ell_1^n$ (the green arcs pictured in Figure \ref{fig:choquetconstruction}) and $\ell_2^n$ (the blue arcs pictured in Figure \ref{fig:choquetconstruction}). At each step, there is a value of $i$ for which $\ell_i^{n+1}\cap X_n$ is homotopic to $\ell_i^n$ while for the other arc $\ell_{3-i}^{n+1}$, $\ell_{3-i}^{n+1}\cap X_n$ consists of one arc homotopic to $\ell_1^n$ and another arc homotopic to $\ell_2^n$. Moreover, the value of $i$ with this property alternates between 1 and 2 at each step. Thus, $\mathcal W(\Lambda)$ is the inverse limit of $\R_+^2\xleftarrow{\pi_1}\R_+^2 \xleftarrow{\pi_2} \R_+^2 \xleftarrow{ \pi_3}\ldots$ where \[\begin{pmatrix} 1 & 1 \\ 0 & 1 \end{pmatrix}=\pi_1=\pi_3=\pi_5=\ldots \text{ and} \begin{pmatrix} 1 & 0 \\ 1 & 1 \end{pmatrix}=\pi_2=\pi_4=\pi_6=\ldots.\] Observe that \[ M_o \vcentcolon = \pi_i \circ \pi_{i+1}=\begin{pmatrix} 2 & 1 \\ 1 & 1 \end{pmatrix} \text{ for } i \text{ odd and }  M_e \vcentcolon = \pi_i \circ \pi_{i+1}=\begin{pmatrix} 1 & 1 \\ 1 & 2 \end{pmatrix}  \text{ for } i \text{ even}.\]

Consider the cones $C_n=C(X_n)=\R_+^2$ and the intersection of the images of $C_m$ in $C_n$. That is, consider $\bigcap_{m=n}^\infty \pi_{nm}(C_m)$ where $\pi_{nm}=\pi_n \circ \pi_{n+1}\circ \cdots\circ \pi_{m-1}$.
This is contained in the intersection of cones $\bigcap_{j=0}^\infty M_o^j (\R_+^2)$ or the intersection $\bigcap_{j=0}^\infty M_e^j (\R_+^2)$ depending on whether $n$ is odd or even. In the odd case, this intersection is equal to the ray spanned by $ v_o$ where $v_o$ is the (positive) attracting eigenvector $v_o = (\phi,1)$ of $M_o$, where $\phi$ is the golden ratio $(1+\sqrt{5})/2$. In the even case the intersection is equal to the span $\R_+v_e$ where $v_e$ is the attracting eigenvector $v_e=(\phi-1,1)$ of $M_e$. Thus, we see that $\mathcal W(\Lambda)$ is in fact equal to the limit of an inverse system of rays \[ \R_+v_o\ \xleftarrow{\pi_1} \  \R_+ v_e \  \xleftarrow{\pi_2} \  \R_+ v_o  \ \xleftarrow{ \pi_3} \  \R_+ v_e \ \xleftarrow{ \pi_4 } \ldots\] and one may check that each map $\pi_n$ in this inverse system is surjective. Thus $\mathcal W(\Lambda)$ is linearly homeomorphic to a ray $\R_+$. That is, $\Lambda$ has a single non-zero transverse measure up to scaling.
\end{ex}

\begin{figure}[h]
\centering

\begin{tabular}{c c}
\def\svgwidth{0.5\textwidth}
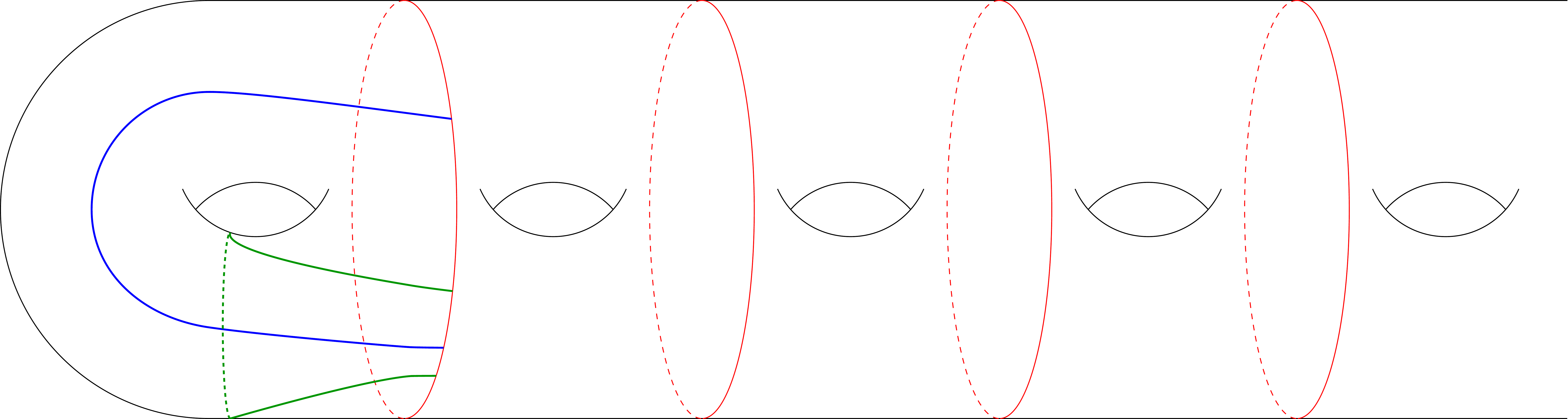 &

\def\svgwidth{0.5\textwidth}
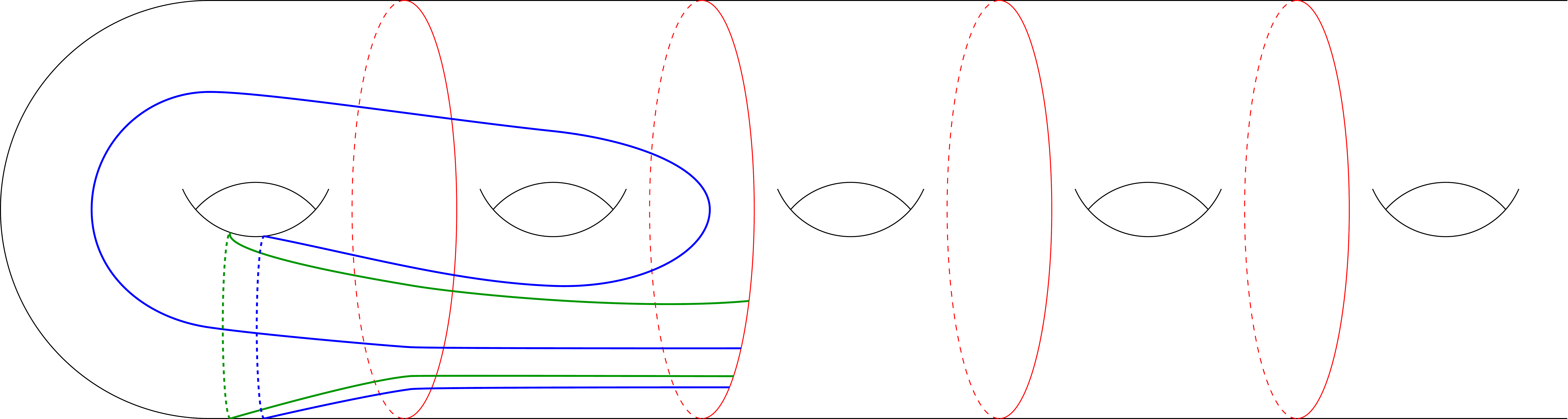 \\

\def\svgwidth{0.5\textwidth}
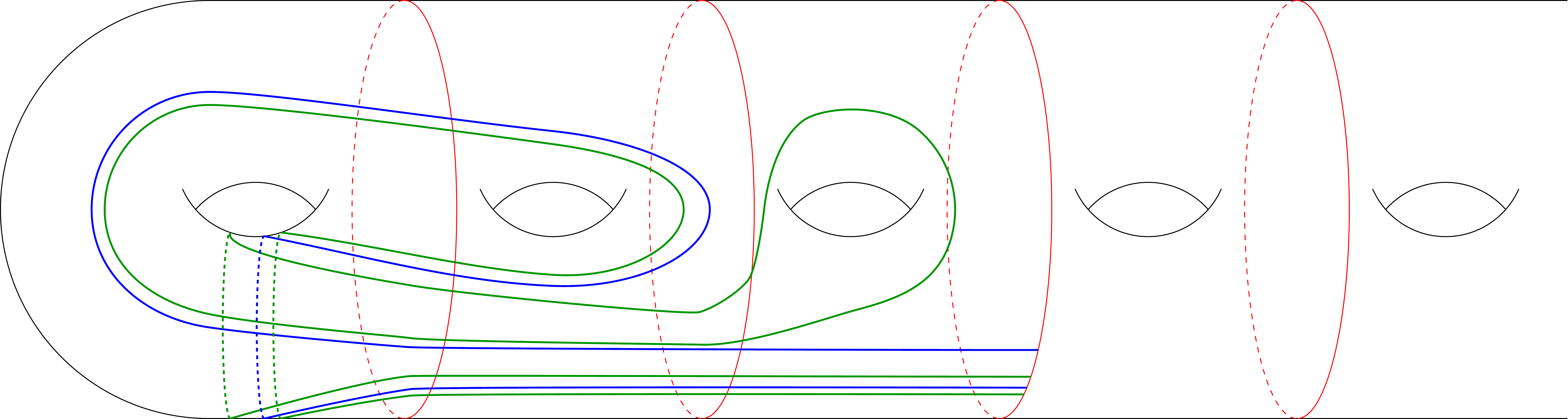 &

\def\svgwidth{0.5\textwidth}
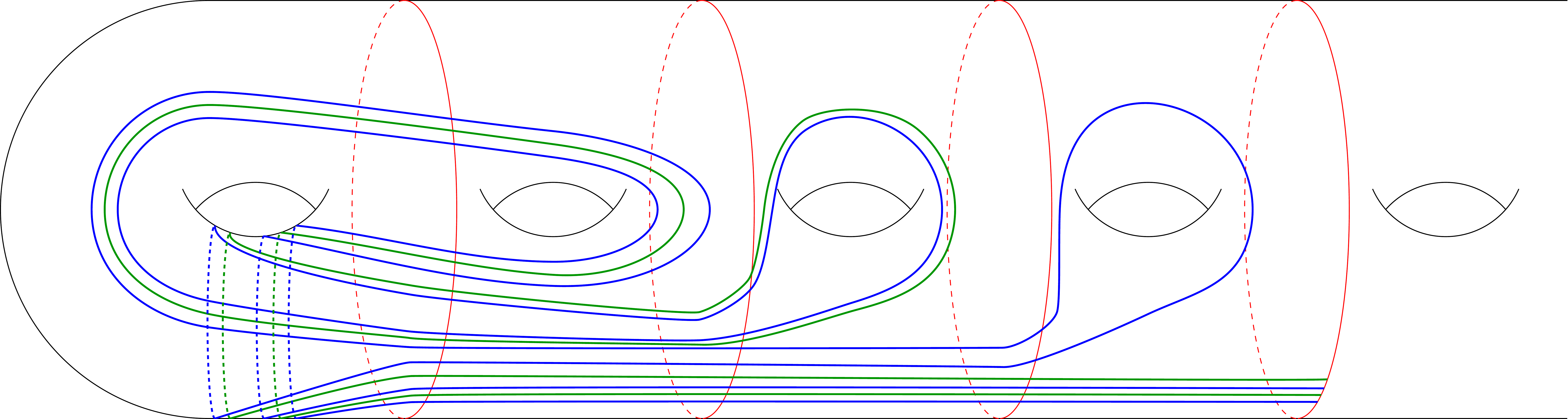 
\end{tabular}

\caption{The construction of the two leaves from Example \ref{ex:choquetexample3}. At each new step one arc is extended to traverse both arcs from the previous step.}
\label{fig:choquetconstruction}
\end{figure}

\begin{figure}[h]
\centering
\def\svgwidth{0.5\textwidth}
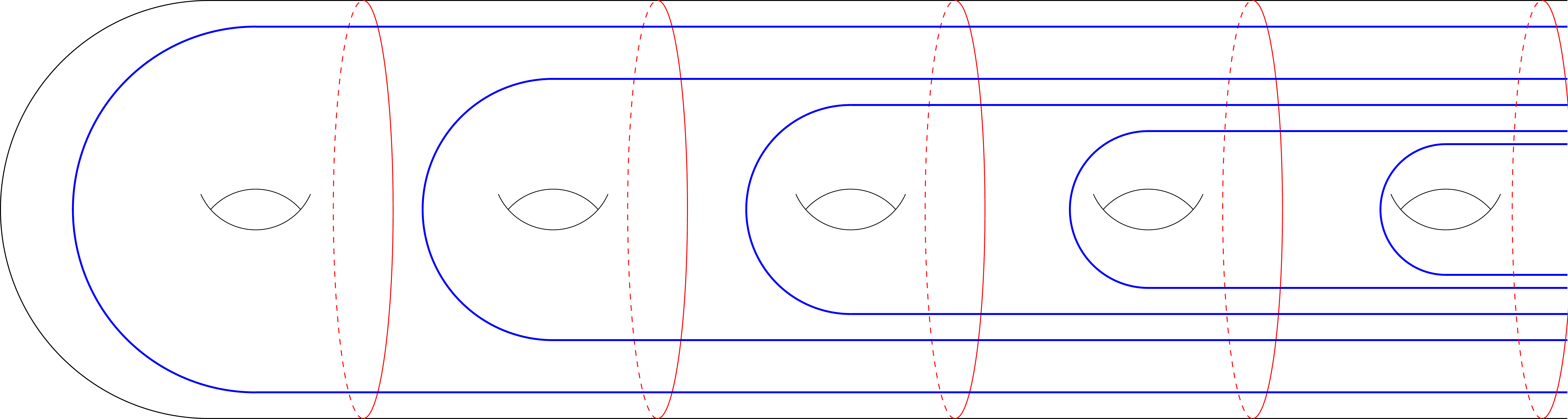

\caption{A lamination consisting of countably many isolated proper leaves.}
\label{fig:nobase1}
\end{figure}

\begin{ex}
\label{ex:nobase1}
Consider the lamination $\Lambda$ in Figure \ref{fig:nobase1}, consisting of countably many isolated proper leaves which exit out the single end of the surface. Using the pictured exhaustion, we see that $\mathcal M(\Lambda)$ is the limit of $\R_+ \xleftarrow{\pi_1} \R_+^2 \xleftarrow{\pi_2} \R_+^3 \xleftarrow{\pi_3} \ldots$ where \[\pi_1=\begin{pmatrix} 1 & 0 \end{pmatrix}, \ \ \ \pi_2=\begin{pmatrix} 1 & 0 & 0 \\ 0 & 1 & 0 \end{pmatrix},  \ \ \ \pi_3=\begin{pmatrix} 1 & 0 & 0 & 0 \\ 0 & 1 & 0 & 0 \\ 0 & 0 & 1 & 0 \end{pmatrix}, \ \ \ \ldots.\] From this we see that $\mathcal M(\Lambda)$ is linearly homeomorphic to $\R_+^{\mathbb{N}}$, a countable product of rays $\R_+$, with the product topology. 
\end{ex}

\subsection{Completing the proof of Theorem \ref{thm:limithomeo}}
\label{sec:completingproof}

In this section we complete the proof of Theorem \ref{thm:limithomeo}. Our main new ingredient will be used to show injectivity of the map $\Psi$.

\begin{lem}
\label{lem:sameintnum}
Let $\Lambda$ be a geodesic lamination on the hyperbolic surface $X$ of the first kind. Let $\tau$ be an arc transverse to $\Lambda$. Consider the exhaustion $X=\bigcup_{n=1}^\infty X_n$ by punctured compact subsurfaces and the homotopy classes of arcs $\{\ell_i^n\}_{i=1}^{r(n)}$ contained in $X_n$ for each $n$. Then for any $m$ large enough, and any $i$ between 1 and $r(m)$, all of the arcs in the homotopy class $\ell_i^m$ intersect $\tau$ the same number of times. 
\end{lem}

Before giving the proof we give an informal description. Each homotopy class $\ell_i^n$ forms a strip $A_i^n\times I_i^n$ where $A_i^n$ is compact totally disconnected and $I_i^n$ is a closed interval in $\R$. The transversal $\tau$ intersects these strips but may not pass all the way through each time. Thus, $\tau$ may turn around between two arcs $L$ and $M$ of $A_i^n\times I_i^n$ or have an endpoint between them. For $m\geq n$ the strips $A_j^m\times I_j^m$ traverse the strip $A_i^n \times I_i^n$ and partition it. The partition is eventually fine enough to separate $L$ and $M$ and this fixes the issue.

We also introduce some notation. If $Y\subset \widetilde{X}$ is a closed convex subset, then $\partial_0 Y$ denotes the boundary of $Y$ as a subset of $\widetilde{X}$. The notation $\partial_\infty Y$ denotes the limit set of $Y$ in $\partial_\infty \widetilde X$, i.e. the closure of $Y$ in $\widetilde X\cup \partial_\infty \widetilde X$ intersected with $\partial_\infty \widetilde X$. Before proving Lemma \ref{lem:sameintnum}, we prove the following general fact, which will be used several times in the sequel. Note that for a punctured compact subsurface $Y\subset X$, the pre-image of $Y$ in the universal cover $\widetilde{X}$ consists of a family of disjoint closed convex subsets of $\widetilde{X}$.

\begin{lem}
\label{lem:firstkindgeodesics}
Let $X$ be a hyperbolic surface of the first kind. Let $X=\bigcup_{n=1}^\infty X_n$ be an exhaustion by punctured compact subsurfaces. Let $\widetilde{X}$ be the universal cover of $X$, choose a basepoint $\ast \in \widetilde{X}$ in the pre-image of $X_1$, and let $\widetilde{X_n}$ be the unique component of the pre-image of $X_n$ in $\widetilde{X}$ containing $\ast$. Then $\bigcup_{n=1}^\infty \widetilde{X_n} = \widetilde{X}$. Moreover, for two distinct geodesics $L,M\subset \widetilde{X}$ and any $n$ sufficiently large, the arcs of intersection $L\cap \widetilde{X_n}$ and $M\cap \widetilde{X_n}$ are not homotopic in $\widetilde{X_n}$ through homotopies preserving $\partial_0 \widetilde{X_n}$.
\end{lem}

\begin{proof}
The component $\widetilde{X_n}$ is invariant under the fundamental group $\pi_1(X_n)$. Consequently the union $\bigcup_{n=1}^\infty \widetilde{X_n}$ is a convex subset of $\widetilde{X}$ which is invariant under $\pi_1(X)$, and therefore we have $\bigcup_{n=1}^\infty \widetilde{X_n} = \widetilde{X}$. 

For the last sentence of the lemma, note that at least one endpoint of $L$ in $\partial_\infty \widetilde{X}$ is not shared by $M$. Hence, for any $n$ sufficiently large, either $L$ has an endpoint in $\partial_\infty \widetilde{X_n}$ which is not contained in $M \cap \partial_\infty \widetilde{X_n}$, or $L\cap \widetilde{X_n}$ has an endpoint in $\partial_0 \widetilde{X_n}$ which is not contained in $M \cap \partial_0 \widetilde{X_n}$. In either case, $L\cap \widetilde{X_n}$ is not homotopic to $M\cap \widetilde{X_n}$.
\end{proof}

\begin{proof}[Proof of Lemma \ref{lem:sameintnum}]

Choose $n$ large enough that $\tau$ is contained in $X_n$. We consider the lamination $\Lambda_n$ and the homotopy classes $\{\ell_i^n\}_{i=1}^{r(n)}$. From $\Lambda_n$, remove all of the compact minimal sub-laminations $\Gamma_i^n$ and all of the geodesics accumulating onto them. Denote by $\Lambda_n'$ the sub-lamination which remains after this operation, consisting exactly of the arcs in all of the classes $\ell_i^n$.

We make one simplifying assumption on $\tau$, which we will remove at the end of the proof: for each $\ell_i^n$, each arc in the homotopy class intersects $\tau$ exactly 0 or 1 times. This has the following consequence. If $\tau$ intersects an arc in the homotopy class $\ell_i^n$ then either (1) $\tau$ crosses every arc in $\ell_i^n$ exactly once or (2) it intersects some of them once, and has an endpoint between two arcs in $\ell_i^n$. In particular, there are some values of $i$ such that $\tau$ crosses all arcs in $\ell_i^n$ once and at most two values of $i$ such that $\tau$ crosses some of the arcs in $\ell_i^n$ once and doesn't cross the others.

Consider the universal cover $\widetilde{X}$. Lift $\tau$ to an arc $\widetilde{\tau}$. The arc $\widetilde{\tau}$ is contained in a unique component $\widetilde{X_n}$ of the pre-image of $X_n$ in $\widetilde{X}$. Thus, $\widetilde{X_n}$ is a universal cover for $X_n$. Define $\widetilde{\Lambda_n'}$ to be the pre-image of $\Lambda_n'$ in $\widetilde{X_n}$. For each $i$, the arcs in the pre-image of $\ell_i^n$ are partitioned into families of arcs which are homotopic in $\widetilde{X_n}$ through homotopies preserving the boundary $\partial_0 \widetilde{X_n}$ setwise. Namely, if $\ell_i^n$ joins $p$ to $q$ where $p,q$ can each be either punctures or boundary components of $X_n$, then $p$ and $q$ each either lift to geodesics of $\partial_0 \widetilde{X}_n$ (lifts of boundary components) or to ends of $\partial_\infty \widetilde{X_n}$. A homotopic family of arcs in the pre-image of $\ell_i^n$ joins a lift of $p$ to a lift of $q$. See Figure \ref{fig:transversalstrips}.

Consider the families of homotopic arcs in $\widetilde{\Lambda_n'}$ that intersect $\widetilde{\tau}$. This yields two finite multi-sets $\mathcal A$ and $\mathcal B$ with $|\mathcal A|=|\mathcal B|$, such that:
\begin{itemize}
\item each element of $\mathcal A$ (or $\mathcal B$) is either a boundary component of $\widetilde{X_n}$ in $\partial_0 \widetilde{X_n}$ or an end of $\widetilde{X_n}$ in $\partial_\infty \widetilde X$;
\item enumerating the elements of $\mathcal A$ as $\{a_1,\ldots,a_r\}$ and the elements of $\mathcal B$ as $\{b_1,\ldots,b_r\}$, $\widetilde{\tau}$ intersects $\widetilde{\Lambda_n'}$ \emph{only} in the arcs joining $a_i$ to $b_i$ for $i=1,\ldots,r$.
\end{itemize}
We allow $\mathcal A$ and $\mathcal B$ to be multi-sets, since for instance, $\widetilde{\tau}$ may be intersected by two families of arcs which join a common boundary component of $\widetilde{X_n}$ to two different boundary components of $\widetilde{X_n}$. Moreover, we choose the numbering so that $\widetilde{\tau}$ intersects each arc of $\widetilde{\Lambda_n'}$ joining $a_i$ to $b_i$ \emph{exactly once} unless $i=1$ or $r$. For $i=1$ or $r$, $\widetilde{\Lambda_n'}$ may intersect \emph{only some} of the arcs of $\widetilde{\Lambda_n'}$ joining $a_i$ to $b_i$ (again intersecting each arc at most once).

\begin{figure}[h]
\centering

\begin{tabular}{c c}
\begin{subfigure}[t]{0.5\textwidth}
\centering
\def\svgwidth{0.85\textwidth}
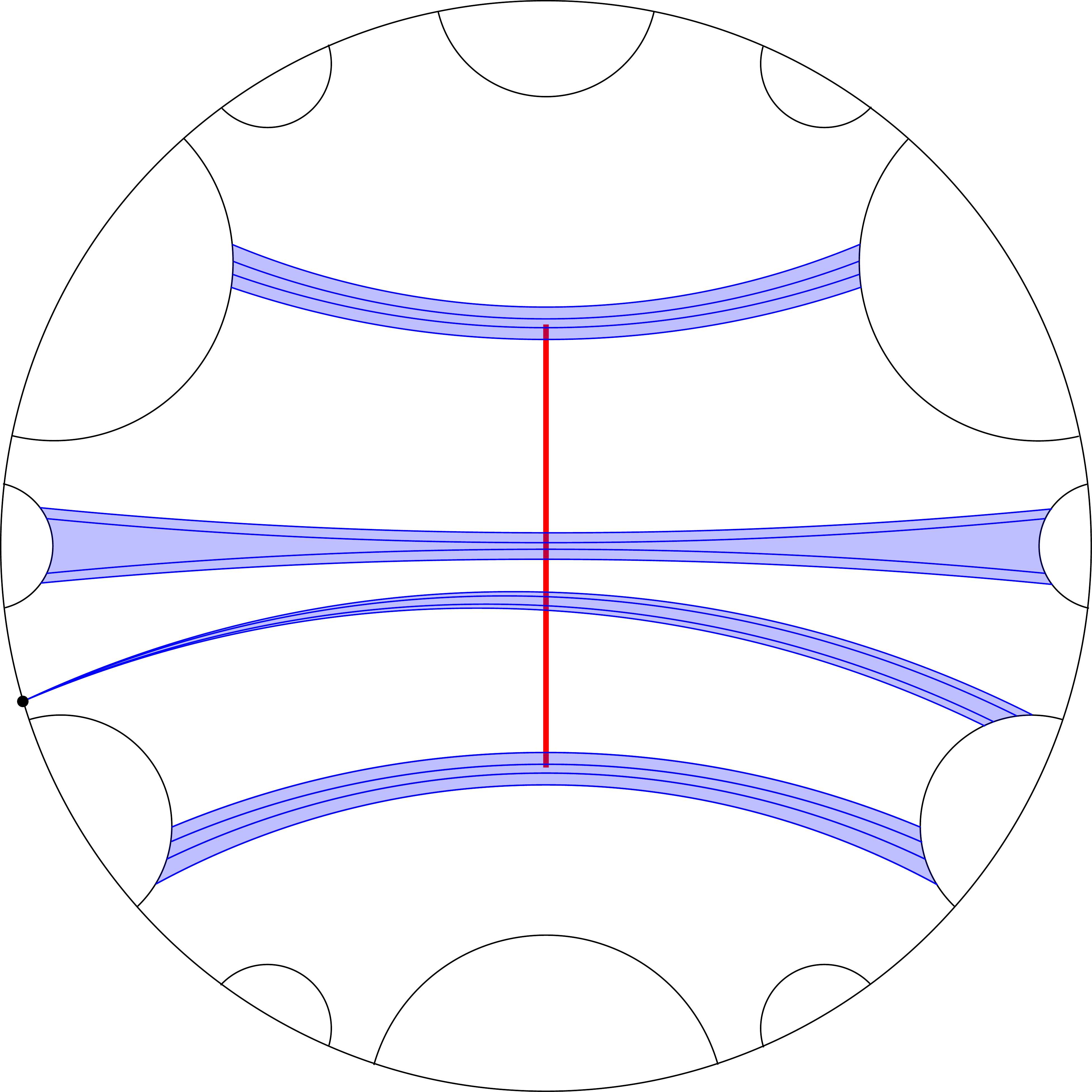
\caption{The cover $\widetilde{X_n}$ is bounded by the solid black geodesics. The lift $\widetilde{\tau}$ is drawn in red along with the strips of $\widetilde{\Lambda_n}$ that intersect it. The elements of $\mathcal A$ are the $a_i$ and the elements of $\mathcal B$ are the $b_i$. Note that all $a_i$ and $b_i$ are geodesics of $\partial_0 \widetilde{X_n}$ except for $a_2$ which is an end of $\widetilde{X_n}$. Note also that $b_1=b_2$.}
\label{fig:transversalstrips}
\end{subfigure} &

\begin{subfigure}[t]{0.5\textwidth}
\centering
\def\svgwidth{0.85\textwidth}
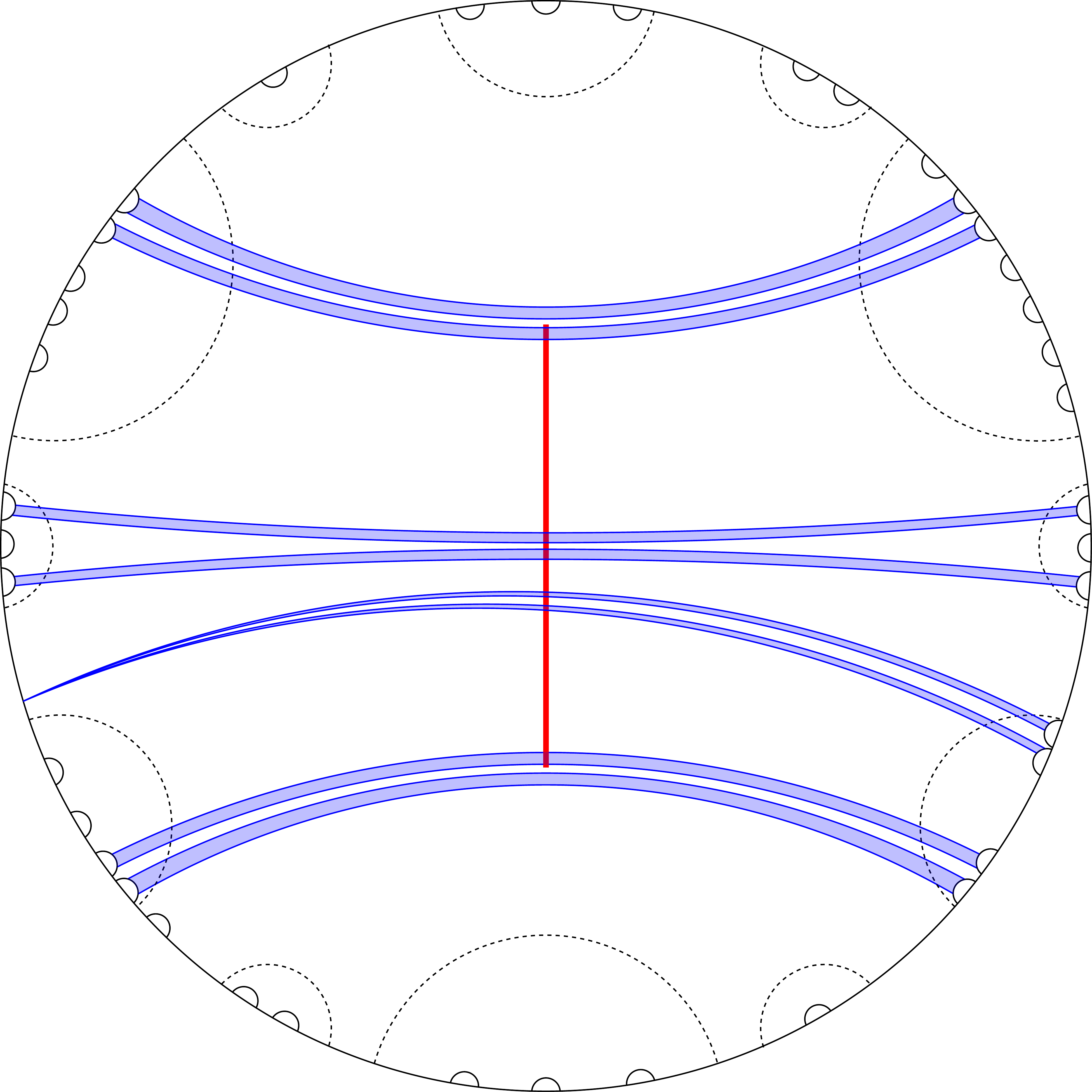
\caption{The cover $\widetilde{X_m}$ is bounded by the solid black geodesics. Boundary components of $\widetilde{X_n}$ are indicated by dotted lines. For each strip of geodesics of $\widetilde{\Lambda}$ in $\widetilde{X_m}$, its geodesics all either intersect $\widetilde{\tau}$ exactly once or all intersect it zero times.}
\label{fig:transversalextended}
\end{subfigure}

\end{tabular}

\caption{Moving to a larger surface $\widetilde{X_m}$ to separate geodesics in $\widetilde{X_n}$.}
\end{figure}

Now, if the arcs joining $a_1$ to $b_1$ do not all intersect $\widetilde{\tau}$ then there is a \emph{last} such arc $L$ which intersects $\widetilde{\tau}$ and a \emph{first} such arc $M$ which does not intersect $\widetilde{\tau}$. That is, $L$ and $M$ separate all the arcs from $a_1$ to $b_1$ which do not intersect $\widetilde \tau$ from all the arcs from $a_1$ to $b_1$ which do. Similarly, among the arcs joining $a_r$ to $b_r$, there is (possibly) a \emph{last} such arc $L'$ which intersects $\widetilde{\tau}$ and a \emph{first} such arc $M'$ which does not intersect $\widetilde{\tau}$. Now $L,M,L',M'$ may be extended to bi-infinite geodesics in $\widetilde{X}$. By Lemma \ref{lem:firstkindgeodesics} for any $m$ large enough, we have that in $\widetilde{X_m}$, the unique component of the pre-image of $X_m$ containing $\widetilde{\tau}$, the intersections $L\cap \widetilde{X_m}$ and $M\cap \widetilde{X_m}$ are \emph{not} homotopic in $\widetilde{X_m}$ through homotopies preserving $\partial_0 \widetilde{X_m}$ and similarly for $L'$ and $M'$. 
See Figure \ref{fig:transversalextended}.

We claim that if $m$ is large enough to satisfy this condition, then all arcs in $\ell_i^m$ for any $i=1,\ldots,r(m)$ intersect $\tau$ the same number of times. To see this, we consider $\Lambda_m'$, the sub-lamination of $\Lambda_m$ consisting of all the arcs in all of the homotopy classes $\ell_i^m$. We again consider the pre-image $\widetilde{\Lambda_m'}$ in $\widetilde{X_m}$, the pre-image $\widetilde{\ell_i^m}$ in $\widetilde{X_m}$ for each $i$, and our fixed lift $\widetilde{\tau}$. Any pre-image $\widetilde{\ell_i^m}$ is partitioned into families of arcs in $\widetilde{X_m}$ which are homotopic through homotopies preserving $\partial_0 \widetilde{X_m}$. Moreover, if $c$ and $d$ are components of $\partial_0 \widetilde{X_m} \cup \partial_\infty \widetilde{X_m}$, then the arcs of $\widetilde{\Lambda_m'}$ joining $c$ to $d$ either \emph{all} intersect $\widetilde{\tau}$ or \emph{all} miss $\widetilde{\tau}$ by our choice of $m$. See Figure \ref{fig:intersectonce}.

\begin{figure}[h]
\centering

\begin{tabular}{c c}
\begin{subfigure}[t]{0.5\textwidth}
\centering
\def\svgwidth{0.85\textwidth}
\begingroup%
  \makeatletter%
  \providecommand\color[2][]{%
    \errmessage{(Inkscape) Color is used for the text in Inkscape, but the package 'color.sty' is not loaded}%
    \renewcommand\color[2][]{}%
  }%
  \providecommand\transparent[1]{%
    \errmessage{(Inkscape) Transparency is used (non-zero) for the text in Inkscape, but the package 'transparent.sty' is not loaded}%
    \renewcommand\transparent[1]{}%
  }%
  \providecommand\rotatebox[2]{#2}%
  \newcommand*\fsize{\dimexpr\f@size pt\relax}%
  \newcommand*\lineheight[1]{\fontsize{\fsize}{#1\fsize}\selectfont}%
  \ifx\svgwidth\undefined%
    \setlength{\unitlength}{1248.56020469bp}%
    \ifx\svgscale\undefined%
      \relax%
    \else%
      \setlength{\unitlength}{\unitlength * \real{\svgscale}}%
    \fi%
  \else%
    \setlength{\unitlength}{\svgwidth}%
  \fi%
  \global\let\svgwidth\undefined%
  \global\let\svgscale\undefined%
  \makeatother%
  \begin{picture}(1,1)%
    \lineheight{1}%
    \setlength\tabcolsep{0pt}%
    \put(0,0){\includegraphics[width=\unitlength,page=1]{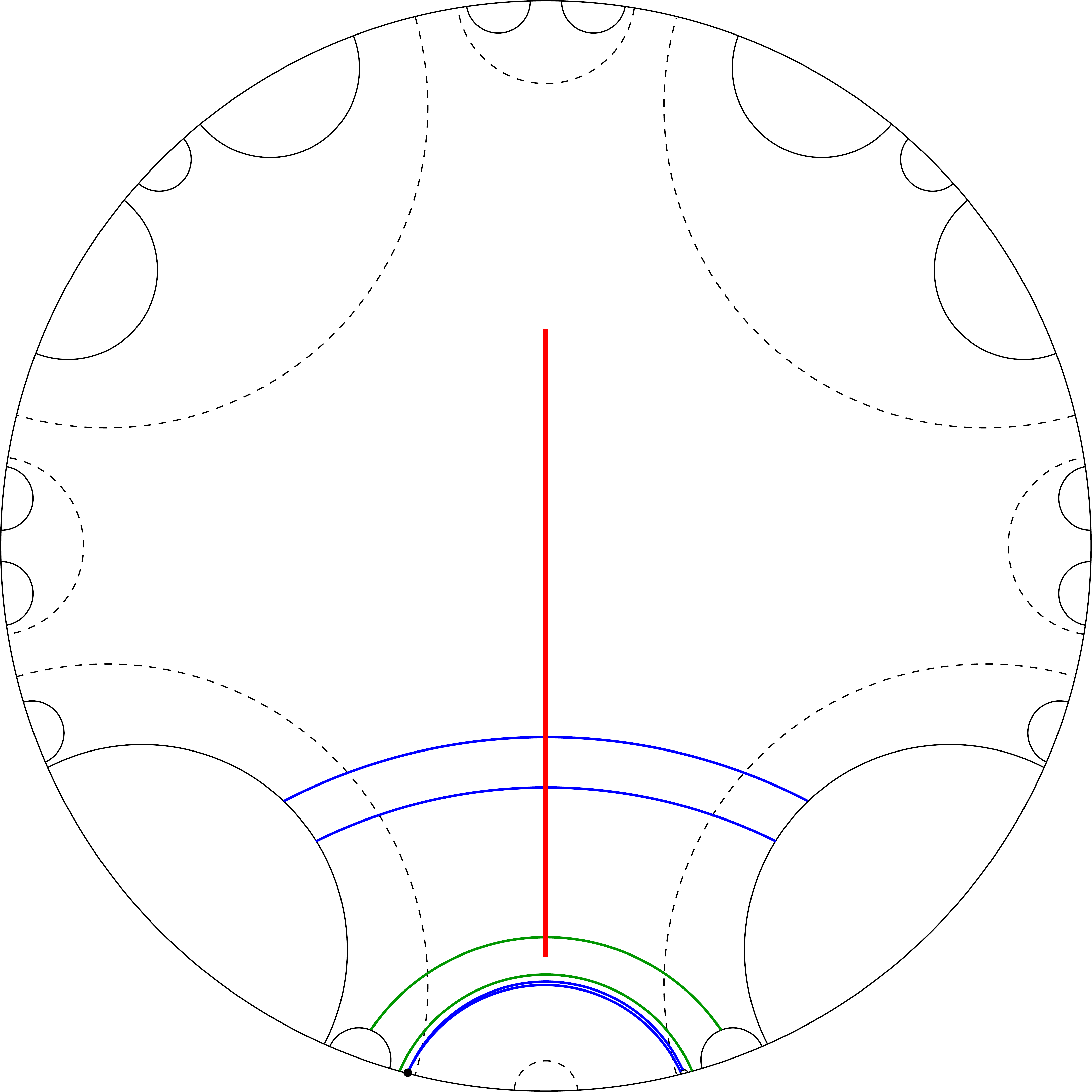}}%
    \put(0.26995633,0.35678929){\color[rgb]{0,0,0}\makebox(0,0)[lt]{\lineheight{1.25}\smash{\begin{tabular}[t]{l}$a_1$\end{tabular}}}}%
    \put(0.68468985,0.35678929){\color[rgb]{0,0,0}\makebox(0,0)[lt]{\lineheight{1.25}\smash{\begin{tabular}[t]{l}$b_1$\end{tabular}}}}%
    \put(0.3182678,0.1987959){\color[rgb]{0,0,0}\makebox(0,0)[lt]{\lineheight{1.25}\smash{\begin{tabular}[t]{l}$c$\end{tabular}}}}%
    \put(0.65225249,0.1987959){\color[rgb]{0,0,0}\makebox(0,0)[lt]{\lineheight{1.25}\smash{\begin{tabular}[t]{l}$d$\end{tabular}}}}%
    \put(0.42972454,0.14347154){\color[rgb]{0.66666667,0.53333333,0}\makebox(0,0)[lt]{\lineheight{1.25}\smash{\begin{tabular}[t]{l}$L$\end{tabular}}}}%
    \put(0.54635576,0.10301903){\color[rgb]{0.66666667,0.53333333,0}\makebox(0,0)[lt]{\lineheight{1.25}\smash{\begin{tabular}[t]{l}$M$\end{tabular}}}}%
    \put(0,0){\includegraphics[width=\unitlength,page=2]{intersect_once.pdf}}%
  \end{picture}%
\endgroup%

\caption{Consider $c$ and $d$ components of $\partial_0\widetilde{X_m} \cup \partial_\infty \widetilde{X_m}$. If the arcs from $c$ to $d$ intersect $\widetilde{\tau}$ and intersect $\partial \widetilde{X_n}$ in arcs from $a_1$ to $b_1$ then they all do so exactly once since $L$ separates them from $M$. Here dotted black lines indicate $\partial_0 \widetilde{X_n}$, solid black lines indicate $\partial_0 \widetilde{X_m}$, and $\widetilde{\tau}$ is indicated by a solid red line. There is a second (small) strip of geodesics which do not intersect $\widetilde{\tau}$ at all since they are separated from $L$ by $M$.}
\label{fig:intersectonce}
\end{subfigure} &

\begin{subfigure}[t]{0.5\textwidth}
\centering
\def\svgwidth{0.85\textwidth}
\begingroup%
  \makeatletter%
  \providecommand\color[2][]{%
    \errmessage{(Inkscape) Color is used for the text in Inkscape, but the package 'color.sty' is not loaded}%
    \renewcommand\color[2][]{}%
  }%
  \providecommand\transparent[1]{%
    \errmessage{(Inkscape) Transparency is used (non-zero) for the text in Inkscape, but the package 'transparent.sty' is not loaded}%
    \renewcommand\transparent[1]{}%
  }%
  \providecommand\rotatebox[2]{#2}%
  \newcommand*\fsize{\dimexpr\f@size pt\relax}%
  \newcommand*\lineheight[1]{\fontsize{\fsize}{#1\fsize}\selectfont}%
  \ifx\svgwidth\undefined%
    \setlength{\unitlength}{1248.56020469bp}%
    \ifx\svgscale\undefined%
      \relax%
    \else%
      \setlength{\unitlength}{\unitlength * \real{\svgscale}}%
    \fi%
  \else%
    \setlength{\unitlength}{\svgwidth}%
  \fi%
  \global\let\svgwidth\undefined%
  \global\let\svgscale\undefined%
  \makeatother%
  \begin{picture}(1,1)%
    \lineheight{1}%
    \setlength\tabcolsep{0pt}%
    \put(0,0){\includegraphics[width=\unitlength,page=1]{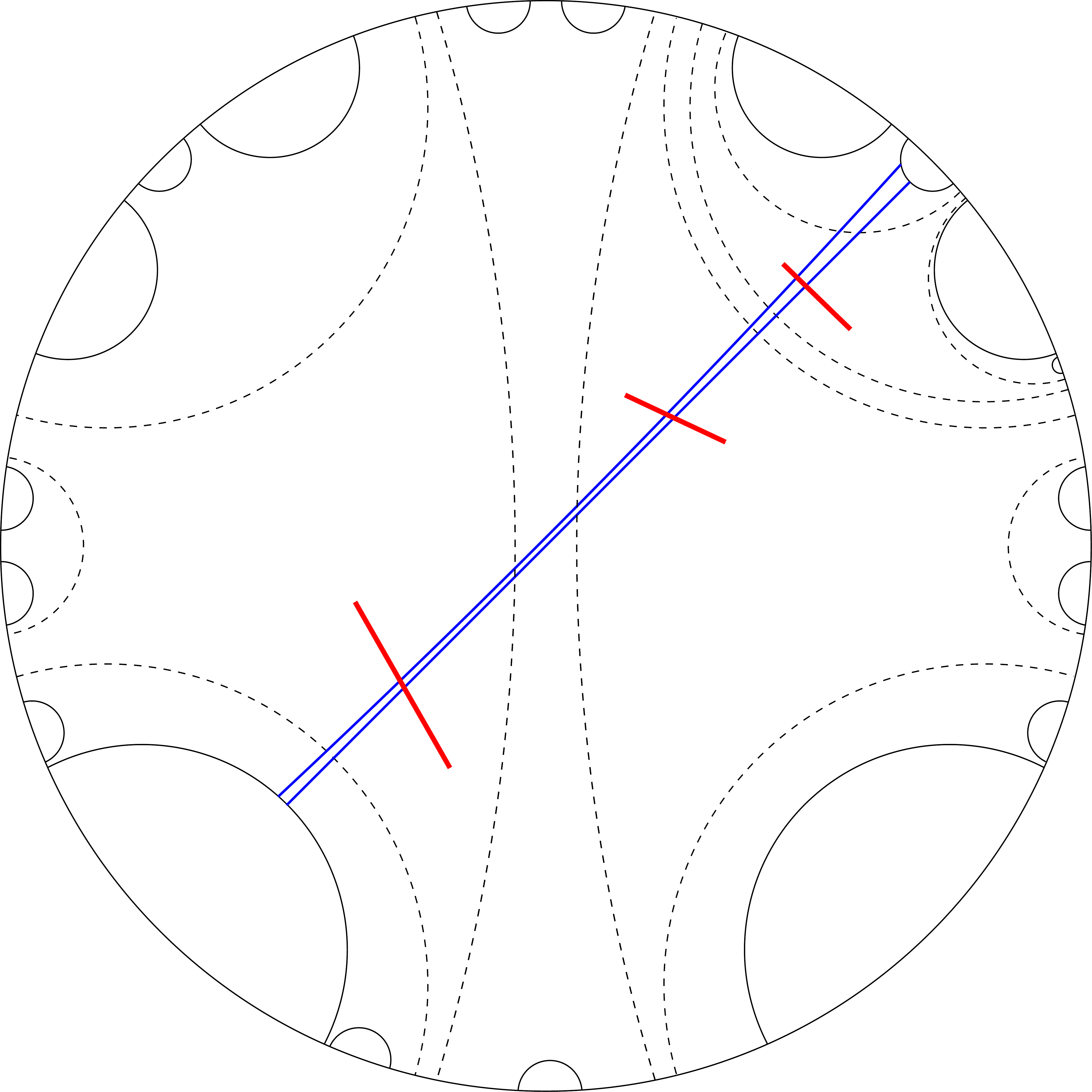}}%
    \put(0.30034595,0.22338228){\color[rgb]{0,0,0}\makebox(0,0)[lt]{\lineheight{1.25}\smash{\begin{tabular}[t]{l}$a$\end{tabular}}}}%
    \put(0.79028528,0.85129303){\color[rgb]{0,0,0}\makebox(0,0)[lt]{\lineheight{1.25}\smash{\begin{tabular}[t]{l}$b$\end{tabular}}}}%
    \put(0,0){\includegraphics[width=\unitlength,page=2]{transversal_translates.pdf}}%
  \end{picture}%
\endgroup%

\caption{Consider $a$ and $b$ which are either boundary components in $\partial_0\widetilde{X_m}$ or ends in $\partial_\infty \widetilde{X_m}$, and the arcs from $a$ to $b$. They may intersect multiple translates of $\widetilde{X_n}$ and for each translate of $\widetilde{X_n}$ they intersect a translate of $\widetilde{\tau}$ at most once. Here translates of $\widetilde{\tau}$ are indicated by solid red lines. Translates of $\widetilde{X_n}$ are shaded in green.}
\label{fig:transversaltranslates}
\end{subfigure}

\end{tabular}
\caption{Intersections of arcs of $\widetilde{\Lambda_m'}$ with lifts of the transversal $\tau$.}
\end{figure}

Thus, there are finite multi-sets $\mathcal A'=\{a'_1,\ldots,a'_s\}$ and $\mathcal B'=\{b'_1,\ldots,b'_s\}$, each consisting of boundary components of $\partial_0\widetilde{X_m}$ and/or ends of $\partial_\infty \widetilde{X_m}$, and such that the arcs of $\widetilde{\Lambda_m'}$ intersecting $\widetilde{\tau}$ are exactly those joining $a_i'$ to $b_i'$ for some $i$. Moreover, the arcs of $\widetilde{\Lambda_m'}$ joining $a_i'$ to $b_i'$ all intersect $\widetilde{\tau}$ exactly once. For each $i=1,\ldots,s$ there is a sub-interval of $\widetilde{\tau}$, call it $J_i$, with endpoints in $\widetilde{\Lambda_m'}$, containing all the intersections with the arcs from $a'_i$ to $b'_i$ and no intersections with arcs from $a'_j$ to $b'_j$ for $j\neq i$. Consider the arcs in a homotopic family in $\widetilde{\Lambda_m'}$ joining a boundary component or end $a$ to a boundary component or end $b$. Whenever the family intersects a lift of $\tau$, the lift has the form $g\widetilde{\tau}$ for some $g\in \pi_1(X_m)$. Translating by $g^{-1}$, we see that all of the arcs from $a$ to $b$ intersect the lift $g\widetilde{\tau}$ exactly once, in the interval $gJ_i$ for some $i$. See Figure \ref{fig:transversaltranslates}.

Finally, we see from this that every arc in a class $\ell_i^m$ intersects $\tau$ a number of time equal to the number of lifts of $\tau$ that $\ell_i^m$ intersects when we lift it to a family of arcs homotopic through homotopies preserving $\partial_0\widetilde{X_m}$. This completes the proof in the special case that every arc in every class $\ell_i^n$ intersects $\tau$ at most once. For the case of a general transversal $\tau$, split $\tau$ into transversals $\tau_1, \tau_2, \ldots, \tau_t$ such that each $\tau_i$ intersects every arc in every class $\ell_i^n$ at most once. Apply the previous arguments to the arcs $\tau_i$ separately to find numbers $m_i\geq n$ as in the statement of the lemma for each $\tau_i$. Taking $m=\max\{m_1,\ldots,m_t\}$ completes the proof.
\end{proof}

\begin{rem}
\label{rem:sameintmeas}
Suppose that $\tau$ is a transversal and $n$ is large enough that every arc in each equivalence class $\ell_i^n$ intersects $\tau$ the same number of times, for $1\leq i\leq r(n)$. Let $E_i$ be the number of times that $\tau$ intersects any arc in $\ell_i^n$. Recall that $\tau_i^n$ denotes a transversal intersecting $\Lambda_n$ only in $\ell_i^n$ and intersecting each arc of $\ell_i^n$ exactly once. We may partition $\tau$ into (1) some arcs which intersect $\Lambda_n$ only in $\ell_i^n$, for one value of $1\leq i\leq r(n)$, and intersect each arc in $\ell_i^n$ at most once, plus (2) some arcs which are disjoint from $\ell_1^n\cup \ldots \cup \ell_{r(n)}^n$. The arcs of type (1) can be homotoped into $\tau_i^n$ and then their union covers the points of $\Lambda \cap \tau_i^n$ uniformly $E_i$ times. The arcs of type (2) intersect $\Lambda_n$ only in the minimal sub-laminations $\Gamma_i^n$ for $1\leq i\leq s(n)$ and leaves spiraling onto them. Thus, for $\mu$ any transverse measure, \[\mu(\tau)=\sum_{i=1}^{r(n)} E_i \mu(\tau_i^n) + \sum_{i=1}^{s(n)} (\mu|\Gamma_i^n)(\tau).\]
\end{rem}

Our second preliminary result identifies $\mathcal M(\Lambda)$ with the inverse limit of the $\mathcal M(\Lambda_n)$'s. Denote by $\rho_{n\infty}$ the linear map $\mathcal M(\Lambda)\to \mathcal M(\Lambda_n)$ which restricts transverse measures to $\Lambda_n$: $\rho_{n\infty}(\mu)=\mu|\Lambda_n$.

\begin{lem}
The restriction maps $\rho_{n\infty}$ identify $\mathcal M(\Lambda)$ linearly homeomorphically with the limit of the inverse system $\mathcal M(\Lambda_1) \xleftarrow{\rho_1} \mathcal M(\Lambda_2) \xleftarrow{\rho_2} \ldots$.
\end{lem}

\begin{proof}
Since $\rho_{n-1}\circ\rho_{n\infty}=\rho_{(n-1) \infty}$, there is an induced continuous linear map $\mathcal M(\Lambda)\to \varprojlim \mathcal M(\Lambda_n)$. On the other hand, there is a map $\varprojlim \mathcal M(\Lambda_n) \to \mathcal M(\Lambda)$ defined as follows. If $(\mu_n)_{n=1}^\infty\in \varprojlim \mathcal M(\Lambda_n)$ then its image $\mu$ in $\mathcal M(\Lambda)$ is the transverse measure defined by $\mu_\tau=(\mu_n)_\tau$ for any $n$ large enough that $\tau$ lies in $X_n$. The map $\varprojlim \mathcal M(\Lambda_n) \to \mathcal M(\Lambda)$ is linear and continuous since these properties hold for each map $\mu_n \mapsto (\mu_n)_\tau$. One may check that these functions between $\mathcal M(\Lambda)$ and $\varprojlim \mathcal M(\Lambda_n)$ are mutually inverse.
\end{proof}

Thus, we may define a map $\Psi:\mathcal M(\Lambda)=\varprojlim \mathcal M(\Lambda_n)\to \varprojlim C_n= \mathcal W(\Lambda)$ by $\Psi=(\Psi_1,\Psi_2,\Psi_3,\ldots)$. Our last preliminary result will be used to show that $\Psi$ is proper and surjective.

\begin{lem}
\label{lem:mapsproper}
Each map $\Psi_n:\mathcal M(\Lambda_n)\to C_n$ is proper and surjective.
\end{lem}

\begin{proof}
Since $\mathcal M(\Lambda_n)=\prod_{i=1}^{r(n)} \mathcal M(A_i^n) \times \prod_{i=1}^{s(n)} \mathcal M(\Gamma_i^n)$, $C_n=\prod_{i=1}^{r(n)} \R_+ \times \prod_{i=1}^{s(n)} \mathcal M(\Gamma_i^n)$, $\Psi_n$ is defined component-wise, and the maps on the $\mathcal M(\Gamma_i^n)$ are identities, it suffices to show that the maps $\mathcal M(A_i^n)\to \R_+$ defined by taking total mass are proper and surjective. Given any $c\in \R_+$ we may consider $c\delta$ where $\delta$ is a point mass at some point of $A_i^n$. This shows that $\mathcal M(A_i^n)\to \R_+$ is surjective. On the other hand, it is proper, since by the Banach-Alaoglu Theorem the space of measures on $A_i^n$ with total mass bounded by some number $E$ is compact in the weak${}^*$ topology.
\end{proof}

Finally we prove Theorem \ref{thm:limithomeo}.

\begin{proof}[Proof of Theorem \ref{thm:limithomeo}]
The map $\Psi=(\Psi_1,\Psi_2,\ldots)$ is continuous and linear since each $\Psi_i$ is. We now check that $\Psi$ is proper and surjective. If $K\subset \mathcal W(\Lambda)$ is compact and non-empty then its images $K_i$ in $C_i$ are each compact and non-empty. Each $\Psi_i^{-1}(K_i)$ is compact and non-empty by Lemma \ref{lem:mapsproper}. Finally, $\Psi^{-1}(K)$ is equal to the inverse limit of the sets $\Psi_i^{-1}(K_i)$ with the transition maps $\rho_i$ (\cite[Sec. I.4.4 Corollary to Proposition 9]{bourbaki}). An inverse limit of non-empty compact Hausdorff spaces is non-empty and compact (\cite[Section I.9.6, Proposition 8]{bourbaki}). Thus $\Psi^{-1}(K)$ is compact and non-empty so that $\Psi$ is proper and surjective. A proper map between metrizable spaces is closed, so $\Psi$ is closed since $\mathcal W(\Lambda)$ and $\mathcal M(\Lambda)$ are metrizable (as subsets of countable products of metrizable spaces). To complete the proof, it suffices to show that $\Psi$ is injective. Suppose that $\mu,\mu'\in \mathcal M(\Lambda)$ with $\mu\neq \mu'$. Then choosing a transversal for which $\mu_\tau \neq \mu'_\tau$ and possibly passing to a sub-transversal, we may suppose that $\mu(\tau)\neq \mu'(\tau)$. By Lemma \ref{lem:sameintnum}, we may choose a surface $X_n$ large enough that for each $1\leq i\leq r(n)$, each arc in the homotopy class $\ell_i^n$ intersects $\tau$ the same number of times $E_i$. By Remark \ref{rem:sameintmeas}, \[\mu(\tau)=\sum_{i=1}^{r(n)} E_i \mu(\tau_i^n) +\sum_{i=1}^{s(n)} (\mu|\Gamma_i^n)(\tau)\] and similarly for $\mu'$. Consequently we must have $\mu(\tau_i^n)\neq \mu'(\tau_i^n)$ or $\mu|\Gamma_i^n\neq \mu'|\Gamma_i^n$ for some $i$. These are the components of the image of $\mu$ in $C_n$, so $\Psi(\mu)\neq \Psi(\mu')$. This completes the proof.
\end{proof}

\noindent As noted earlier, Theorem \ref{mainthm:inverselim} follows immediately.

\subsection{Effectivizing the linear homeomorphism}

We showed that the map $\Psi:\mathcal M(\Lambda)\to \mathcal W(\Lambda)$ is a linear homeomorphism. However, the inverse $\Psi^{-1}$ remains mysterious from this point of view. Defining $\Psi^{-1}$ would yield a more effective result, in that one could explicitly construct a transverse measure from any element of the inverse limit $\mathcal W(\Lambda)$. We outline how to do this, leaving the details to the interested reader.

Consider an element $(w_n)_{n=1}^\infty \in \varprojlim C_n$. We wish to construct a transverse measure $\mu$ to $\Lambda$ from $(w_n)_n$. To do this, we construct \emph{approximate measures}. Consider a transversal $\tau$ to $\Lambda$. It is contained in $X_n$ for all $n$ sufficiently large. Set $E_i^n$ to be the \emph{maximum} number of times that any arc in the homotopy class $\ell_i^n$ intersects $\tau$ (the number of intersection points may vary by arc). If \[w_n = \sum_{i=1}^{r(n)} b_i^n e_i^n + \sum_{i=1}^{s(n)} \nu_i^n \text{ then we define } w_n(\tau)=\sum_{i=1}^{r(n)} b_i^nE_i^n + \sum_{i=1}^{s(n)} \nu_i^n(\tau)\] which we think of as an approximate measure of $\tau$. We emphasize that this does not define an actual transverse measure to $\Lambda$ but only an approximation. We take limits to find honest measures:

\begin{prop}
Let $\tau$ be a transversal to $\Lambda$ and $(w_n)_{n=1}^\infty \in \mathcal W(\Lambda)$. Then the approximate measures $w_n(\tau)$ are decreasing with $n$ and therefore $\lim_{n\to \infty} w_n(\tau)$ exists.
\end{prop}

One now defines a pre-measure $\mu_\tau$ by setting $\mu(\sigma)=\lim_{n\to \infty} w_n(\sigma)$ for any sub-transversal $\sigma$ of $\tau$ and extending over disjoint unions of such sub-transversals. An application of the Carath\'{e}odory Extension Theorem yields an honest measure $\mu_\tau$ on $\tau$.

\begin{prop}
Let $(w_n)_{n=1}^\infty\in \mathcal W(\Lambda)$ and define the limits $\mu_\tau$ as above for any transversal $\tau$ to $\Lambda$. Then the Borel measures $\mu_\tau$ define a transverse measure to $\Lambda$. Moreover, setting $\Psi^{-1}((w_n)_{n=1}^\infty)=\mu$ defines the inverse homeomorphism to the homeomorphism $\Psi:\mathcal M(\Lambda)\to \mathcal W(\Lambda)$.
\end{prop}

\section{Bases for cones of measures}
\label{sec:bases}

In Corollary \ref{cor:baseexistence} we showed that $\mathcal M(\Lambda)$ admits a base whenever there is a compact subsurface of $X$ intersecting every leaf of $\Lambda$. This criterion is sufficient but not necessary for the existence of a compact base. Thus, the question of which cones of transverse measures admit bases is not completely straightforward. An example of a cone of transverse measures which has no compact base is the infinite product of rays, $\R_+^{\mathbb{N}}$ (see Example \ref{ex:nobase1}). This example recurs repeatedly. In Section \ref{sec:nocompactbase} we give other examples of cones without bases.

Even when a base does exist, its structure is not transparent from Corollary \ref{cor:baseexistence}. In this section we make the structure more transparent by proving Theorem \ref{mainthm:choquetbase} from the introduction. First consider the case of an inverse system of finite-dimensional simplicial cones \[C_1\xleftarrow{f_1} C_2 \xleftarrow{ f_2 }C_3 \xleftarrow{ f_3} \ldots.\] Here the maps $f_n:C_{n+1}\to C_n$ are linear.
For $n\leq m$ we denote by $f_{nm}:C_m\to C_n$ the composition $f_{nm}\vcentcolon=f_n \circ f_{n+1} \circ \cdots \circ f_{m-1}$.

\begin{lem}
\label{lem:coneinverselim}
Let $\mathcal C$ be the limit of an inverse system \[C_1\xleftarrow{f_1 } C_2 \xleftarrow{ f_2 }C_3 \xleftarrow{ f_3 } \ldots\] of finite-dimensional simplicial cones with linear maps $f_n$. Suppose that the maps $f_n$ satisfy the property that $f_n(v)=0$ only if $v=0$. Let $B_1$ be a base for $C_1$ and define $B_n$ to be the inverse image $f_{1n}^{-1}(B_1)$. Then the inverse limit of the bases $B_n$ is a base for $\mathcal C$ and it is a compact metrizable Choquet simplex.
\end{lem}

\noindent The condition $f_n(v)=0$ only if $v=0$ is \emph{not} equivalent to injectivity of $f_n$. Rather, $f_n$ may be extended to a linear map on some $\R^m$ and the condition says that the kernel of the extension intersects $C_n$ only at 0. To prove this lemma we use the following important theorem of Davies-Vincent-Smith:

\begin{thm}[{\cite[Theorem 13]{tensor_prod}}]
\label{thm:choquetinverselim}
Consider an inverse system $B_1\xleftarrow{f_1 } B_2 \xleftarrow{ f_2 } B_3 \xleftarrow{ f_3} \ldots$ of Choquet simplices with affine maps $f_n$. Then the limit $\mathcal B$ of this inverse system is a Choquet simplex. 
\end{thm}

\begin{proof}[Proof of Lemma \ref{lem:coneinverselim}]
As in the statement, choose a base $B_1$ for $C_1$ and define $B_n$ to be the inverse image $f_{1n}^{-1}(B_1)$ in $C_n$. One may check that $B_n$ is convex using that $B_1$ is convex. For $v\in C_n\setminus \{0\}$, there is a unique $r>0$ with $rf_{1n}(v)\in B_1$ and therefore $r>0$ is the unique number with $rv \in B_n$; i.e. $B_n$ is a base.

We obtain by restriction an inverse system \[B_1\xleftarrow{f_1 } B_2 \xleftarrow{ f_2 } B_3 \xleftarrow{ f_3 } \ldots\] of finite-dimensional simplices. Define $\mathcal B$ to be the inverse limit of this system. It is a subspace of $\mathcal C$. We claim that in fact $\mathcal B$ is a base for $\mathcal C$. Consider an element $(v_n)_{n=1}^\infty \in \mathcal C \setminus \{0\}$. We have $v_1\neq 0$. Thus there is a unique $r>0$ with $rv_1\in B_1$. Then for each $n$ we have $rv_n\in f_{1n}^{-1}(B_1)=B_n$. Thus, $r(v_n)_{n=1}^\infty\in \mathcal B$ and $r$ is the unique number with this property. We may verify that $\mathcal B$ is convex by using the convexity of each $B_n$. This proves that $\mathcal B$ is a base, as desired. By Theorem \ref{thm:choquetinverselim}, $\mathcal B$ is a Choquet simplex. As an inverse limit of countably many compact metrizable spaces, $\mathcal B$ is compact and metrizable.
\end{proof}

A problem with applying Lemma \ref{lem:coneinverselim} to our cones of measures is that the transition maps $\pi_n$ do not generally satisfy the condition $\pi_n(w_n)\neq 0$ if $w_n\neq 0$. To utilize Lemma \ref{lem:coneinverselim} it will thus be necessary to modify our inverse system.

\subsection{Modifying exhaustions and inverse systems}

In this section we wish to prove Theorem \ref{mainthm:choquetbase} from the introduction. Consider the hyperbolic surface $X$ endowed with an exhaustion $X_1\subset X_2 \subset \ldots$ as considered earlier and a lamination $\Lambda$. Thus $X_n$ is a punctured compact subsurface with geodesic boundary. As before we consider the laminations $\Lambda_n=\Lambda \cap X_n$. Each $\Lambda_n$ contains finitely many homotopy classes of arcs $\{\ell_i^n\}_{i=1}^{r(n)}$ and finitely many compact minimal sub-laminations $\{\Gamma_i^n\}_{i=1}^{s(n)}$ in the interior of $X_n$. We let $C_n$ be the cone for $\Lambda_n$ defined in Section \ref{sec:inverselim} and $\pi_n:C_{n+1}\to C_n$ the resulting transition maps. Then $\mathcal M(\Lambda)$ is linearly homeomorphic to the inverse limit $\mathcal W(\Lambda)$ of the cones $C_n$ with the maps $\pi_n$.

\begin{rem}
\label{rem:choquetcriterion}
Note that if each homotopy class of arcs $\ell_i^{n+1}$ on $X_{n+1}$ and each compact sub-lamination $\Gamma_i^{n+1}$ on $X_{n+1}$ intersects $X_n$, then $\pi_n$ satisfies the property that $\pi_n(w)=0$ only if $w=0$. If this property is satisfied for each $n$, then Lemma \ref{lem:coneinverselim} will show that $\mathcal M(\Lambda)$ has a base which is a compact metrizable Choquet simplex. The property may \emph{not} be satisfied for every lamination though, since there may be arcs or minimal sub-laminations of $\Lambda_{n+1}$ which do not intersect $X_n$ (see e.g. Example \ref{ex:nobase1}).
\end{rem}
Thus, we will attempt to modify our exhaustion $X_1\subset X_2 \subset \ldots$ to have this property. The key lemma to prove is the following:

\begin{lem}
\label{lem:exhaustionmodification}
Let $U\subset V$ be punctured compact subsurfaces of $X$ with geodesic boundary. Let $\Lambda$ be a geodesic lamination on $X$ such that every leaf of $\Lambda$ intersects $U$. Then there is a larger punctured compact subsurface $W\supset V$ with geodesic boundary such that every geodesic of $\Lambda \cap W$ intersects $U$.
\end{lem}

\begin{proof}
The idea of the proof is that we will construct $W$ by gluing on \emph{strips} $[0,1]\times [0,1]$ to the boundary of $V$. We will do this as follows: if a geodesic of $\Lambda\cap V$ doesn't hit $U$ then we may extend the geodesic in one direction until it \emph{does} hit $U$. The extended geodesic leaves $V$ finitely many times and then eventually enters $U$. We will add on a strip containing each arc where the extended geodesic leaves $V$. The resulting subsurface may not be essential so we finish the proof by adding on disks and punctured disks and homotoping the boundary components to geodesics.

Every leaf of $\Lambda$ which is contained entirely in $V$ must intersect $U$ by hypothesis. So we focus on geodesics of $\Lambda\cap V$ which have at least one endpoint on $\partial V$. Consider $p\in \partial V\cap \Lambda$. It is contained in a leaf $L$ of $\Lambda$. Choose an orientation for the geodesic $L$ and denote by $L|[p,\infty)$ and $L|(-\infty,p]$ the rays of $L$ based at $p$ which are oriented \emph{away from} $p$ and \emph{towards} $p$, respectively. At least one of these two rays intersects $U$; say $L|[p,\infty)$, without loss of generality. Consider the first intersection point $q$ of $L|[p,\infty)$ with $U$. This gives rise to a sub-arc $L|[p,q]$ of $L$ from $p$ to $q$. There is a small open arc $I_p$ of $\partial V$ containing $p$ for which all rays of $\Lambda$ through $I_p$ in the direction of $L|[p,\infty)$ contain a sub-arc with endpoints in $I_p$ and $\partial U$ which is homotopic to $L|[p,q]$ (through homotopies preserving the boundary components). 

The arc $L|[p,q]$ leaves $V$ at most finitely many times and then hits $\partial U$ at $q$. We may take a small neighborhood of $L|[p,q]$ containing all the homotopic arcs through $I_p$ and $\partial U$. In the complement of $V$, this consists of a finite disjoint union $S_p$ of \emph{strips} $[0,1]\times [0,1]$ such that: (1) the \emph{horizontal boundary components} $[0,1]\times \{0\}$ and $[0,1]\times \{1\}$ are contained in $\partial V$; (2) the \emph{vertical boundary components} are disjoint from $\Lambda$; and (3) any leaf of $\Lambda$ through a point in $I_p$ contains a sub-arc in $V\cup S_p$ homotopic to $L|[p,q]$. In particular, any leaf of $\Lambda$ which passes through $I_p$ contains a sub-arc in $V\cup S_p$ which intersects $U$.

\begin{figure}[h]
\centering
\def\svgwidth{0.5\textwidth}
\begingroup%
  \makeatletter%
  \providecommand\color[2][]{%
    \errmessage{(Inkscape) Color is used for the text in Inkscape, but the package 'color.sty' is not loaded}%
    \renewcommand\color[2][]{}%
  }%
  \providecommand\transparent[1]{%
    \errmessage{(Inkscape) Transparency is used (non-zero) for the text in Inkscape, but the package 'transparent.sty' is not loaded}%
    \renewcommand\transparent[1]{}%
  }%
  \providecommand\rotatebox[2]{#2}%
  \newcommand*\fsize{\dimexpr\f@size pt\relax}%
  \newcommand*\lineheight[1]{\fontsize{\fsize}{#1\fsize}\selectfont}%
  \ifx\svgwidth\undefined%
    \setlength{\unitlength}{1278.73956155bp}%
    \ifx\svgscale\undefined%
      \relax%
    \else%
      \setlength{\unitlength}{\unitlength * \real{\svgscale}}%
    \fi%
  \else%
    \setlength{\unitlength}{\svgwidth}%
  \fi%
  \global\let\svgwidth\undefined%
  \global\let\svgscale\undefined%
  \makeatother%
  \begin{picture}(1,0.80197371)%
    \lineheight{1}%
    \setlength\tabcolsep{0pt}%
    \put(0,0){\includegraphics[width=\unitlength,page=1]{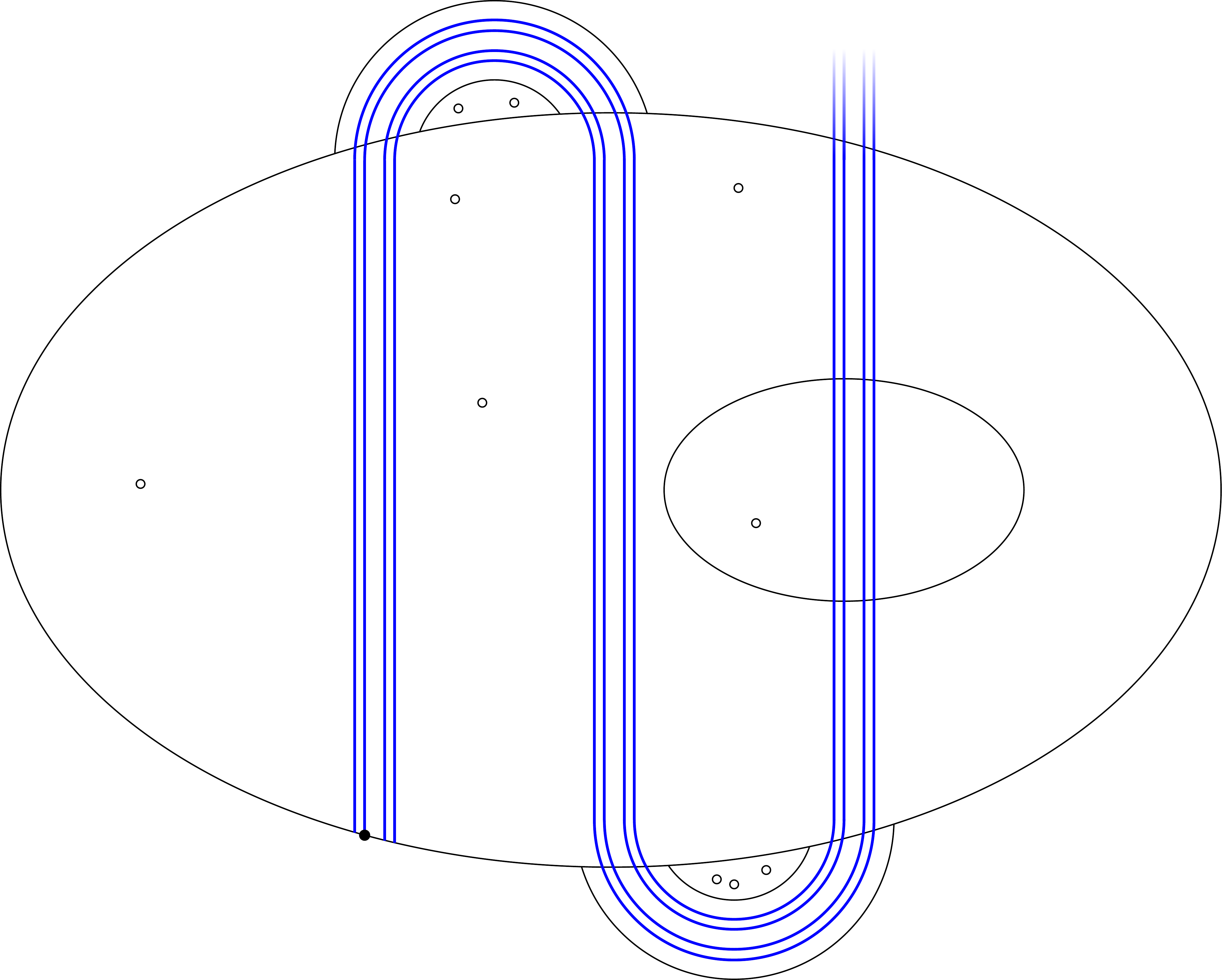}}%
    \put(0.29685716,0.07530394){\color[rgb]{0,0,0}\makebox(0,0)[lt]{\lineheight{1.25}\smash{\begin{tabular}[t]{l}\textit{$p$}\end{tabular}}}}%
    \put(0,0){\includegraphics[width=\unitlength,page=2]{strips.pdf}}%
    \put(0.6377643,0.28640089){\color[rgb]{0,0,0}\makebox(0,0)[lt]{\lineheight{1.25}\smash{\begin{tabular}[t]{l}\textit{$q$}\end{tabular}}}}%
    \put(0,0){\includegraphics[width=\unitlength,page=3]{strips.pdf}}%
    \put(0.5355293,0.76211306){\color[rgb]{0,0,0}\makebox(0,0)[lt]{\lineheight{1.25}\smash{\begin{tabular}[t]{l}\textit{$S_p$}\end{tabular}}}}%
    \put(0.74621159,0.02887261){\color[rgb]{0,0,0}\makebox(0,0)[lt]{\lineheight{1.25}\smash{\begin{tabular}[t]{l}\textit{$S_p$}\end{tabular}}}}%
    \put(0.62545041,0.40098736){\color[rgb]{0,0,0}\makebox(0,0)[lt]{\lineheight{1.25}\smash{\begin{tabular}[t]{l}\textit{$U$}\end{tabular}}}}%
    \put(0.90742642,0.40098736){\color[rgb]{0,0,0}\makebox(0,0)[lt]{\lineheight{1.25}\smash{\begin{tabular}[t]{l}\textit{$V$}\end{tabular}}}}%
  \end{picture}%
\endgroup%

\caption{Adding on strips $[0,1]\times [0,1]$ to form the surface $V_0$. In a small neighborhood of $p\in \partial V\cap \Lambda$, all geodesics fellow travel the ray $L|[p,\infty)$ long enough to have an arc homotopic to $L|[p,q]$. All such homotopic arcs are contained in the union of $V$ with $S_p$, which consists of two strips.}
\label{fig:strips}
\end{figure}

Now, the arcs $I_p$ form an open cover of the compact set $\Lambda \cap \partial V$. Consider a finite sub-cover $I_{p_1} \cup \ldots \cup I_{p_k}$. We form the subsurface \[V_0\vcentcolon=V\cup \bigcup_{i=1}^k S_{p_i}.\] Note that the boundary of $V_0$ consists of some subset of $\partial V$ along with subsets of the \emph{vertical} boundary components of the strips in the finite unions $S_{p_i}$. Since the vertical boundary components of all the strips are disjoint from $\Lambda$, $\Lambda \cap \partial V_0$ is contained in $\Lambda\cap \partial V$. Thus, every point of $\Lambda \cap \partial V_0$ lies in $I_{p_i}$ for some $i$ and therefore the geodesic of $\Lambda \cap V_0$ through such a point intersects $U$.

Now, the subsurface $V_0$ may not be essential: some of its boundary components may bound disks or once-punctured disks. Form $V_1$ by taking the union of $V_0$ with all disks or once-punctured disks bounded by any of the components of $\partial V_0$. Any geodesic in $V_1 \cap \Lambda$ contains at least one geodesic (and possibly multiple geodesics) of $V_0\cap \Lambda$. Hence, any geodesic in $V_1\cap \Lambda$ intersects $U$.

Finally, we form the subsurface $W$ by homotoping the boundary components of $V_1$ to their geodesic representatives. Since $V_1$ contains $V$, so does $W$. Finally, we claim that every geodesic of $\Lambda \cap W$ intersects $U$. Consider a connected component $\widetilde W$ of the pre-image of $W$ in the universal cover $\widetilde X$ and let $\pi_1(W)$ act on $\widetilde X$ stabilizing this component. Then there is a unique component $\widetilde{V_1}$ of the pre-image of $V_1$ which is also stabilized by $\pi_1(W)$. Consider a geodesic of $W\cap \Lambda$ with at least one endpoint on $\partial W$. It is contained in a leaf $L$ of $\Lambda$. This geodesic of $W\cap \Lambda$ lifts to a geodesic contained in $\widetilde W$ with one endpoint on a geodesic $p$ of $\partial_0\widetilde W$ and the other endpoint in a component of $\partial_0 \widetilde W \cup \partial_\infty \widetilde{W}$, which we call $q$. This lifted geodesic is also contained in a lift $\widetilde L$ of $L$. There is a component of $\partial_0 \widetilde{V_1}$ with the same endpoints as $p$ and similarly a component of $\partial_0 \widetilde{V_1}\cup \partial_\infty \widetilde{V_1}$ corresponding to $q$. Since any arc of $L\cap V_1$ intersects $U$, any arc of $\widetilde{L}\cap \widetilde{V_1}$ intersects some lift of $U$. There is thus a component $g$ of the pre-image of $\partial U$ separating the components of $\partial_0 \widetilde{V_1} \cup \partial_\infty \widetilde{V_1}$ corresponding to $p$ and $q$. The component $g$ therefore also separates $p$ and $q$. Thus our lifted geodesic intersects a lift of $U$.
\end{proof}

Now we may prove Theorem \ref{mainthm:choquetbase}.

\begin{proof}[Proof of Theorem \ref{mainthm:choquetbase}]
Begin with an exhaustion $X_1\subset X_2\subset \ldots $ of $X$ by punctured compact subsurfaces of $X$ with geodesic boundary. We will modify the $X_i$ to an exhaustion $Y_1\subset Y_2 \subset Y_3\subset \ldots$ such that every geodesic of $\Lambda\cap Y_i$ intersects $Y_1$ for each $i$ (and in particular every geodesic of $\Lambda \cap Y_i$ intersects $Y_{i-1}$). First choose $X_n$ large enough that every leaf of $\Lambda$ intersects $X_n$. Set $Y_1=X_n$. Now $X_{n+1}$ is a subsurface containing $Y_1$ and by Lemma \ref{lem:exhaustionmodification} there is a punctured compact subsurface $Y_2$ containing $X_{n+1}$ with the property that every geodesic in $\Lambda \cap Y_2$ intersects $Y_1$. Choose $m>n+1$ large enough that $X_m\supset Y_2$. Again by Lemma \ref{lem:exhaustionmodification}, there is a punctured compact subsurface $Y_3$ containing $X_m$ such that every geodesic in $\Lambda\cap Y_3$ intersects $Y_1$. Repeat this process inductively to form the desired exhaustion $Y_1\subset Y_2 \subset Y_3 \subset \ldots$.

Set $C_i$ to be the cone $C(Y_i)$ of weights on $\Lambda\cap Y_i$ for each $i$. There is an inverse system \[C_1 \xleftarrow{\pi_1} C_2 \xleftarrow{\pi_2} C_3  \xleftarrow{\pi_3 }\ldots.\] By Remark \ref{rem:choquetcriterion} and Lemma \ref{lem:coneinverselim}, the inverse limit $\mathcal M(\Lambda)\cong \varprojlim C_i$ has a base $\varprojlim B_i$ where $B_i$ is a base of $C_i$ and $\varprojlim B_i$ is a compact metrizable Choquet simplex.
\end{proof}

\subsection{Examples of bases}

In this section we re-visit the laminations of Examples \ref{ex:choquetexample1} and \ref{ex:choquetexample2} from Section \ref{sec:examplecones} and describe bases for them as Choquet simplices.

\begin{ex}
Consider the lamination $\Lambda$ in Example \ref{ex:choquetexample1}. The cone of transverse measures is  the cone $C\subset \ell^1$ defined by $C=\{(x,y_1,y_2,\ldots) : y_i \geq 0 \text{ for all } i \text{ and } x\geq \sum y_i\}$ with the weak${}^*$ topology obtained from the pre-dual $c_0$. A base for $C$ is given by the convex set $B$ defined by $x=1$. Thus $B=\{(1,y_1,y_2,\ldots) : y_i \geq 0 \text{ for all } i \text{ and } 1\geq \sum y_i\}$. The base $B$ is a Choquet simplex. Its extreme points are \[e=(1,0,0,\ldots) \text{ and } e_i=(1,0,0,\ldots,0,1,0,\ldots)\] where $e_i$ has a 1 in the $i^{\text{th}}$ position. The points $e_i$ are isolated in the space $\operatorname{Ext}(B)$ of extreme points while $e_i\to e$ as $i\to \infty$. Thus $\operatorname{Ext}(B)$ is homeomorphic to the ordinal $\omega+1$. We may also identify the extreme points $e_i$ and $e$ with explicit measures on $\Lambda$. Namely, denote by $L_i$ the isolated leaves of $\Lambda$ and by $L$ the non-isolated leaf, so that $L_i\to L$ as $i\to\infty$. Then $e_i$ is identified with the $\delta$-mass on $L_i$ for each $i$ (which assigns to a transversal its number of intersections with $L_i$), while $e$ is identified with the $\delta$-mass on $L$.

By choosing bases $B_n$ for the cones $C_n=C(X_n)$ as in Lemma \ref{lem:coneinverselim} (where $X_1\subset X_2\subset \ldots$ is the exhaustion chosen in Example \ref{ex:choquetexample1}) we may consider $B$ to be the inverse limit $\varprojlim B_n$. The base $B_n$ has $n$ vertices $v_1^n,\ldots,v_n^n$. The map $B_{n+1}\to B_n$ is defined by $v_i^{n+1}\mapsto v_i^n$ for $1\leq i\leq n$ and $v_{n+1}^{n+1}\mapsto v_n^n$. It is instructive to consider what the extreme points of the inverse limit are. They are: \[e_i=(v_1^1,v_2^2,\ldots,v_i^i,v_i^{i+1},v_i^{i+2},\ldots) \text{ and } e=(v_1^1,v_2^2,\ldots,v_n^n,\ldots).\]
\end{ex}

\begin{ex}
Consider the lamination $\Lambda$ in Example \ref{ex:choquetexample2}. The cone of transverse measures is now the set $C\subset \ell^1$ consisting of sequences $(x_1,x_2,y_1,y_2,y_3,\ldots)$ with $y_i\geq 0$ for all $i$ and $x_j\geq \sum_i y_i$ for $j=1,2$. A Choquet simplex base $B$ is defined by $x_1+x_2=1$ (and $y_i\geq 0$, $x_1,x_2\geq \sum_i y_i$). Its extreme points are \[f_1=(1,0,0,0,\ldots), \ \ f_2=(0,1,0,0,\ldots), \text{ and } e_i=\left(\frac{1}{2},\frac{1}{2},0,0,\ldots,0,\frac{1}{2},0\right)\] where $e_i$ has a $1/2$ in the $i^{\text{th}}$ position. The set of extreme points $\operatorname{Ext}(B)$ is not closed: $f_1,f_2,$ and $e_i$ are all isolated in $\operatorname{Ext}(B)$ while $e_i\to \frac{1}{2}f_1+\frac{1}{2}f_2$, which does not lie in $\operatorname{Ext}(B)$. Here $f_1$ and $f_2$ are identified with $\delta$-masses on the proper non-isolated leaves $L$ and $L'$, respectively. The points $2e_i$ are identified with $\delta$-masses on the proper isolated leaves $L_i$, which converge to $L\cup L'$ as $i\to \infty$.

Again consider the cones $C_n=C(X_n)$ and suitable bases $B_n$ for $C_n$. Then $B_n$ has $n+1$ vertices $v_i^n$. The map $B_{n+1}\to B_n$ is defined by $v_i^{n+1}\mapsto v_i^n$ for $1\leq i\leq n+1$ and $v_{n+2}^{n+1}\mapsto \frac{1}{2}v_1^n+\frac{1}{2}v_2^n$. The extreme points are $f_1=(v_1^1,v_1^2,v_1^3,\ldots)$,  $f_2=(v_2^1,v_2^2,v_2^3,\ldots)$, and \[e_i=\left(\frac{1}{2}v_1^1+\frac{1}{2}v_2^1,\frac{1}{2}v_1^2+\frac{1}{2}v_2^2,\ldots,\frac{1}{2}v_1^i+\frac{1}{2}v_2^i,v_{i+2}^{i+1}, v_{i+2}^{i+2},v_{i+2}^{i+3},\ldots \right).\]
\end{ex}

\section{Inverse limit laminations}
\label{sec:inverselimlams}

In this section we consider a construction of laminations as ``inverse limits'' of systems of arcs. We consider the hyperbolic surface $X$ of the first kind and an exhaustion $X_1\subset X_2 \subset \ldots$ by punctured compact subsurfaces with geodesic boundary. We consider the universal cover $\widetilde X$, which is isometric to the hyperbolic plane and fix a basepoint $\ast \in \widetilde X$ in the pre-image of $X_1$. By Lemma \ref{lem:firstkindgeodesics}, if $\widetilde{X_n}$ is the unique component of the pre-image of $X_n$ containing $\ast$, we have that $\widetilde{X} = \bigcup_{n=1}^\infty \widetilde{X_n}$.

Consider the compactification $\widetilde X \cup \partial_\infty \widetilde X$ where $\partial_\infty \widetilde X\cong S^1$ is the Gromov boundary. Recall that $\partial_\infty \widetilde{X_n}$ denotes the intersection of the closure of $\widetilde{X_n}$ in $\widetilde X \cup \partial_\infty \widetilde X$ with $\partial_\infty \widetilde X$. The complement $\partial_\infty \widetilde X\setminus \partial_\infty \widetilde{X_n}$ is a countable collection of open intervals. If $n\leq m$ then the (interval) components of $\partial_\infty \widetilde X \setminus \partial_\infty \widetilde {X_m}$ are nested in the components of $\partial_\infty \widetilde X \setminus \partial_\infty \widetilde {X_n}$. Since $\bigcup_{n=1}^\infty \widetilde{X_n}=\widetilde X$, we have the following fact: if $I_n$ are components of $\partial_\infty \widetilde{X} \setminus \partial_\infty \widetilde{X_n}$ with $I_1\supset I_2\supset \ldots$, then the intersection $\bigcap_{n=1}^\infty I_n$ consists of a single point of $\partial_\infty \widetilde X$. For each component $I$ of $\partial_\infty \widetilde X \setminus \partial_\infty \widetilde{X_n}$, there is a unique component of $\partial_0 \widetilde{X_n}$ (the topological boundary of $\widetilde{X_n}$ as a subset of $\widetilde{X}$) joining its endpoints.

For each $n$, fix a (finite) collection $A_n$ of pairwise disjoint, pairwise non-homotopic, homotopically non-trivial arcs in $X_n$ with both endpoints on $\partial X_n$. We suppose without loss of generality that each $\ell \in A_n$ has been chosen to be a  geodesic, so that it intersects $X_m$ minimally for each $m\leq n$. We say that the system (set) $\{A_n\}_{n=1}^\infty$ is \emph{directed} if it satisfies the following conditions: (1) for each $\ell\in A_{n+1}$, the arcs of the intersection $\ell\cap X_n$ are homotopic to arcs in $A_n$; and (2) for each $\ell\in A_n$, there is an arc $\ell'\in A_{n+1}$ such that $\ell'\cap X_n$ contains an arc homotopic to $\ell$.
Fix a directed system of collections of arcs $\{A_n\}_{n=1}^\infty$. We will now construct a lamination $\Lambda$ on $X$. The lamination $\Lambda$ will consist of all geodesics on $X$ which intersect each $X_n$ in a family of arcs homotopic to the arcs in $A_n$. To verify that this is a lamination will take a bit of work. 

Each $\ell \in A_n$ lifts to an infinite set of arcs in $\widetilde{X_n}$. Consider two (interval) components  $I$ and $J$ of $\partial_\infty \widetilde X\setminus \partial_\infty \widetilde{X_n}$. There are geodesics $B$ and $C$ of $\partial_0\widetilde{X_n}$ with the same endpoints as $I$ and $J$, respectively. We say that $I$ and $J$ are \emph{joined} by $A_n$ if there is an arc $\ell\in A_n$ and a lift $\widetilde{\ell}$ with one endpoint on $B$ and the other endpoint on $C$. We say that $\ell$ \emph{joins} $B$ and $C$. We define a set of geodesics $\widetilde{\Lambda}$ as follows: if $p,q\in \partial_\infty \widetilde{X}$ and $p\neq q$ then the geodesic $[p,q]$ from $p$ to $q$ in $\widetilde X$ lies in $\widetilde{\Lambda}$ if for some $ n_0\geq 1$ we have \[p=\bigcap_{n=n_0}^\infty I_n \text{ and } q=\bigcap_{n=n_0}^\infty J_n \text{ where } I_{n_0} \supset I_{n_0+1} \supset I_{n_0+2} \supset \ldots \text{ and } J_{n_0}\supset J_{n_0+1} \supset J_{n_0+2}\supset \ldots\] are nested sequences of arcs of $\partial_\infty \widetilde X \setminus \partial_\infty \widetilde{X_n}$ with the property that $I_n$ is joined to $J_n$ by $A_n$ for each $n\geq n_0$. See Figure \ref{fig:inverselimlam}.

\begin{figure}[h]
\centering
\def\svgwidth{0.6\textwidth}
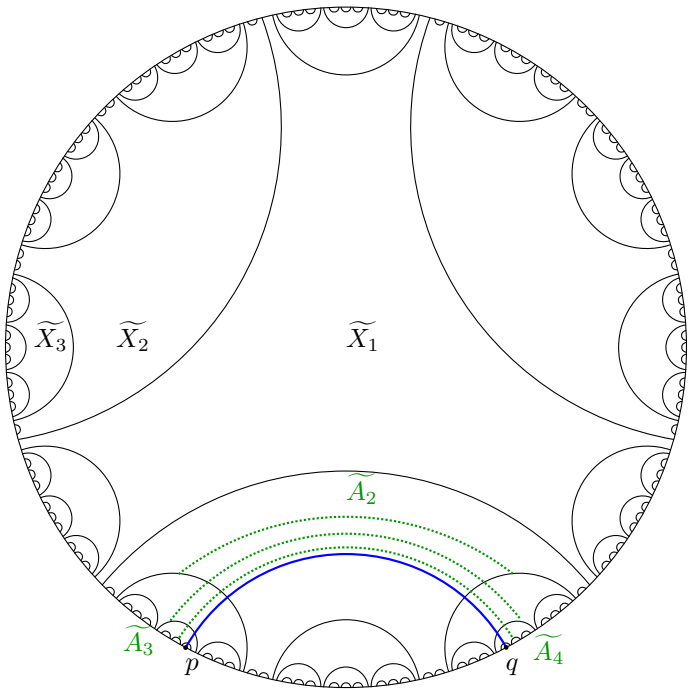

\caption{A geodesic $[p,q]$ in $\widetilde{X}$. Dotted green lines denote lifts of arcs in $A_n$, joining components of $\partial_0\widetilde{X_n}$. The geodesic $[p,q]$ lies in $\widetilde{\Lambda}$ since there are elements of $\widetilde{A_n}$ joining the pairs of intervals containing $p$ and $q$, respectively. Here we may take $n_0=2$.}
\label{fig:inverselimlam}
\end{figure}

It is notable that every lamination on $X$ without compact sub-laminations or leaves asymptotic to isolated punctures arises in this way, as may be seen by considering the arcs of intersection with each subsurface $X_n$. Thus, the construction is reversible.

We will now investigate the properties of $\widetilde \Lambda$ and show that it descends to a lamination on $X$. By the invariance under $\pi_1(X_n)$ of the property of intervals of $\partial_\infty \widetilde{X} \setminus \partial_\infty \widetilde{X_n}$ being joined by $A_n$, one sees that $\widetilde \Lambda$ is invariant under $\pi_1(X)$. Moreover:

\begin{lem}
\label{lem:inverselimnocross}
No two geodesics of $\widetilde \Lambda$ cross.
\end{lem}

\begin{proof}
Consider $[p,q]$ and $[p',q']$ geodesics of $\widetilde \Lambda$. We must show that the pairs $\{p,q\}$ and $\{p',q'\}$ do not separate each other as pairs of points in the circle. For $n_0$ sufficiently large, we may write \[p=\bigcap_{n=n_0}^\infty I_n, \ \ q=\bigcap_{n=n_0}^\infty J_n, \ \ p'=\bigcap_{n=n_0}^\infty I'_n, \ \ q'=\bigcap_{n=n_0}^\infty J'_n\] where $I_n$ and $J_n$ are joined by $A_n$ for each $n$ and similarly for $I'_n$ and $J'_n$. If $\{p,q\}$ separates $\{p',q'\}$ then for all $n$ sufficiently large, $I_n$ and $J_n$ separate $I'_n$ and $J'_n$. Denote by $B_n,C_n,B'_n,$ and $C'_n$ the components of $\partial_0 \widetilde{X_n}$ with the same endpoints as $I_n,J_n,I'_n,$ and $J'_n$, respectively. Then there are $\ell,\ell' \in A_n$ with lifts $\widetilde{\ell},\widetilde{\ell'}$ with endpoints on $B_n$ and $C_n$ and on $B_n'$ and $C_n'$, respectively. But then for $n\gg 0$, $B_n$ and $C_n$ separate $B_n'$ and $C_n'$ so that $\widetilde{\ell}$ and $\widetilde{\ell'}$ cross, and the same is true for $\ell$ and $\ell'$. This contradicts that the arcs in $A_n$ are pairwise disjoint.
\end{proof}

Hence the image of $\widetilde \Lambda$ in $X$ is a family of pairwise non-crossing simple geodesics. Denote this image by $\Lambda$. In order to show that $\Lambda$ is a geodesic lamination, it suffices to show that $\widetilde \Lambda$ is closed.

\begin{lem}
\label{lem:inverselimclosed}
The set $\widetilde \Lambda$ is closed in $\widetilde X$. Hence $\Lambda$ is closed in $X$.
\end{lem}

\begin{proof}
We must show the following: if $p,q\in \partial_\infty \widetilde X$, $p\neq q$, and $[p_i,q_i]$ are geodesics of $\widetilde \Lambda$ with $p_i\to p$ and $q_i\to q$, then $[p,q]\subset \widetilde \Lambda$. First we show that no ray of $[p,q]$ is contained in $\widetilde{X_n}$ for any $n$.

So suppose that a ray of $[p,q]$ is contained in $\widetilde{X_n}$ for some $n$. Thus, one of the endpoints, say $p$, is contained in the closure of $\widetilde{X_n}$ in $\widetilde X\cup \partial_\infty \widetilde X$. Then $p$ is either an endpoint of a geodesic of $\partial_0\widetilde{X_n}$ or it is not. We deal with the latter case first. In this case, for all sufficiently large $i$, $[p_i,q_i]$ intersects $\widetilde{X_n}$. Denote by $I_i$ and $J_i$ the arcs of $\partial_\infty \widetilde X \setminus \partial_\infty \widetilde X_n$ containing $p_i$ and $q_i$, respectively. Then the arcs $I_i$ converge to $p$. The arcs $J_i$ lie in some common neighborhood of $q$. Let $\ell_i$ be an arc of $A_n$ joining $I_i$ and $J_i$. Then we see that the length of $\ell_i$ goes to infinity as $i\to \infty$. This is a contradiction, since $A_n$ is finite.

Now we consider the case that $p$ is an endpoint of a geodesic of $\partial_0 \widetilde{X_n}$. In this case, for any $m>n$ sufficiently large, $p$ is contained in $\partial_\infty \widetilde{X_m}$ but is not the endpoint of a geodesic of $\partial_0 \widetilde{X_m}$; hence this case reduces to the previous after replacing $n$ by $m$. Thus we have shown that no ray of $[p,q]$ is contained in $\widetilde{X_n}$ for any $n$.

Choosing $n_0$ sufficiently large, $[p,q]$ intersects $\widetilde{X_n}$ for each $n\geq n_0$. Moreover, since neither endpoint lies in $\partial_\infty \widetilde{X_n}$ by what we showed above, we have $p\in I_n$ and $q\in J_n$ for two components $I_n$ and $J_n$ of $\partial_\infty \widetilde X \setminus \partial_\infty \widetilde{X_n}$. Consider the sequence $[p_i,q_i]$ of geodesics in $\widetilde \Lambda$ for each $i$. Then for all $i$ sufficiently large, we also have $p_i\in I_n$ and $q_i\in J_n$. Thus, $A_n$ joins $I_n$ to $J_n$ for each $n\geq n_0$. We have $p=\bigcap_{n=n_0}^\infty I_n$ and $q=\bigcap_{n=n_0}^\infty J_n$ so that $[p,q]\subset \widetilde \Lambda$.
\end{proof}

Combining Lemmas \ref{lem:inverselimnocross} and \ref{lem:inverselimclosed} we have the following:

\begin{thm}
\label{thm:inverselim}
Let $A_n$ be a finite collection of homotopically non-trivial, pairwise disjoint, pairwise non-homotopic arcs on $X_n$. Assume that the system $\{A_n\}_{n=1}^\infty$ is directed. Then the set $\Lambda$ obtained from $\{A_n\}_{n=1}^\infty$ is a geodesic lamination on $X$.
\end{thm}

We call $\Lambda$ the \emph{inverse limit} of the system $\{A_n\}_{n=1}^\infty$. One nice application is that inverse limits allow us to easily construct examples of \emph{minimal} laminations.

\begin{defn}
Let $\Lambda$ be a geodesic lamination. We say that $\Lambda$ is \emph{minimal} if it has no proper sub-laminations.
\end{defn}

Equivalently, a lamination is minimal exactly when all of its leaves are dense in the lamination.

\begin{prop}
\label{prop:inverselimminimal}
Let $A_n$ be a finite set of pairwise non-homotopic, pairwise disjoint, homotopically non-trivial arcs on $X_n$ and suppose that $\{A_n\}_{n=1}^\infty$ is directed. Suppose that $\{A_n\}_{n=1}^\infty$ has the following property: for each $\ell\in A_n$, there is $m_0\geq n$, such that if $m\geq m_0$ then for each $\ell'\in A_m$, $\ell'\cap A_n$ contains an arc homotopic to $\ell$. Then the inverse limit lamination $\Lambda$ of $\{A_n\}$ is minimal.
\end{prop}

\begin{proof}
Let $L$ be a leaf of $\Lambda$. Let $\widetilde L=[p,q]$ be a lift of $L$ to $\widetilde X$. It suffices to show that for any other leaf $M$ of $\Lambda$, there is a lift of $M$ to $\widetilde X$ with endpoints on $\partial_\infty \widetilde X$ which are arbitrarily close to $p$ and $q$. By definition of $\Lambda$, there is $n_0\geq 0$ such that $p=\bigcap_{n=n_0}^\infty I_n$ and $q=\bigcap_{n=n_0}^\infty J_n$ where $I_n$ and $J_n$ are arcs of $\partial_\infty \widetilde X\setminus \partial_\infty \widetilde{X_n}$ which are joined by $A_n$ and we have $I_{n_0}\supset I_{n_0+1}\supset \ldots$ and similarly for the $J_n$. Since the $I_n$ nest down to $p$ and the $J_n$ nest down to $q$, it suffices to show that there is a lift of $M$ with endpoints in $I_n$ and $J_n$ for any $n\geq n_0$. Fix an $n\geq n_0$. Let $B_n$ and $C_n$ be the geodesics of $\partial_0 \widetilde{X_n}$ with the same endpoints as $I_n$ and $J_n$, respectively. Then there is an arc $\ell\in A_n$ and a lift $\widetilde \ell$ to $\widetilde X$ with endpoints on $B_n$ and $C_n$.

For any $m$ sufficiently large compared to $n$, each arc of $A_m$ traverses each arc of $A_n$. Choose any $m$ with this property and with the property that $M$ intersects $X_m$. Then $M\cap X_m$ contains an arc which is homotopic to some $\ell' \in A_m$. Therefore $M\cap X_n$ contains $\ell' \cap X_n$, which contains an arc homotopic to $\ell$. Lifting this arc of $M\cap X_n$ and extending it to a lift of $M$, we see that there is a lift $\widetilde M$ of $M$ which intersects $B_n$ and $C_n$ transversely, and therefore $\widetilde M$ has endpoints in $I_n$ and $J_n$. This completes the proof.
\end{proof}

\subsection{Cones of transverse measures for inverse limit laminations}

\label{sec:inverselimcones}

We note that if $\{A_n\}_{n=1}^\infty$ is a directed system of arcs and $\Lambda$ is the inverse limit lamination, then $\Lambda$ consists of \emph{exactly} the simple bi-infinite geodesics $L$ on $X$ for which each intersection $L\cap X_n$ consists of a (possibly empty) set of arcs all homotopic to arcs in $A_n$. As already used implicitly earlier, if $L$ is a leaf of $\Lambda$ then it satisfies this property. On the other hand if, for each $n$, $L\cap X_n$ consists of arcs homotopic to arcs in $A_n$, then we may lift $L$ to a geodesics $\widetilde L=[p,q]$. Then $\widetilde L$ intersects $\widetilde {X_{n_0}}$ for $n_0$ large enough and for each $n\geq n_0$, $\widetilde L\cap \widetilde{X_n}$ is an arc from a component $B_n$ of $\partial_\infty \widetilde{X} \setminus \partial_\infty \widetilde{X_n}$ to a component $C_n$ and $B_n$ is joined to $C_n$ by $A_n$ for each such $n$. Thus $L$ lies in $\Lambda$.

Finally, for each $\ell\in A_n$ there is a leaf $L$ of $\Lambda$ such that $L\cap X_n$ contains an arc homotopic to $\ell$. To see this, set $\ell_n=\ell$, and inductively for $i> n$, set $\ell_i \in A_i$ to be an arc such that $\ell_i\cap X_{i-1}$ contains $\ell_{i-1}$. Choose $I_n,J_n\subset \partial_\infty \widetilde{X} \setminus \partial_\infty \widetilde{X_n}$ to be intervals joined by $\ell_n$. Inductively, we may choose arcs $I_n\supset I_{n+1} \supset\ldots$ and $J_n\supset J_{n+1} \supset \ldots$ which are joined by $\ell_i$ for $i\geq n$. Setting $p=\bigcap_{i=n}^\infty I_n$ and $q=\bigcap_{i=n}^\infty J_n$ we have that $\widetilde L=[p,q]$ is a geodesic of $\widetilde \Lambda$ and its image $L$ in $X$ is a leaf of $\Lambda$ satisfying that $L\cap X_n$ contains $\ell$.

Using this discussion, we may read off the cone of transverse measures $\mathcal M(\Lambda)$. Denote the elements of $A_n$ by $\{\ell_i^n\}_{i=1}^{|A_n|}$. We've shown that $\Lambda\cap X_n$ consists exactly of the homotopy classes of arcs in $A_n$ for each $n\geq 1$. By the discussion in Section \ref{sec:transmaps}, we have:

\begin{lem}
\label{lem:inverselimcone}
Let $\{A_n\}_{n=1}^\infty$ be a directed system of arcs on $X_n$. Then the cone $\mathcal M(\Lambda)$ is linearly homeomorphic to the limit of the inverse system  \[ \R_+^{|A_1|} \xleftarrow{ \pi_1 } \R_+^{|A_2|} \xleftarrow{\pi_2 } \R_+^{|A_3|} \xleftarrow{\pi_3 } \ldots\] where $\pi_n$ is the $|A_n|\times |A_{n+1}|$ matrix whose $(i,j)$-entry counts the number of arcs of $\ell_j^{n+1}\cap X_n$ which are homotopic to $\ell_i^n$.
\end{lem}

\section{Realizing Choquet simplices}
\label{sec:realization}

In this section we prove Theorem \ref{mainthm:choquetrealization} from the introduction. To do this, we use several tools. First we have the following realization theorem of Lazar-Lindenstrauss:

\begin{thm}[{\cite[Corollary to Theorem 5.2]{lazar_lindenstrauss}}]
\label{thm:llinverselims}
Let $\Delta$ be a compact metrizable Choquet simplex. Then there exists a sequence of finite-dimensional simplices $\Delta_n$ together with surjective affine maps $f_n:\Delta_{n+1}\to \Delta_n$ such that $\Delta$ is affinely homeomorphic to the limit of the inverse system \[\Delta_1 \xleftarrow{f_1 } \Delta_2 \xleftarrow{f_2 } \Delta_3 \xleftarrow{f_3 }  \ldots.\]
\end{thm}

Our other main tool is an approximation theorem of Brown (\cite{brown}). First we set the notation. If $f:X\to Y$ and $g:X\to Y$ are maps between compact metric spaces then $D(f,g)$ denotes the supremum distance $D(f,g)=\sup\{d(f(x),g(x)) : x\in X\}$. Suppose that \[X_1 \xleftarrow{f_1 } X_2 \xleftarrow{f_2 } X_3 \xleftarrow{f_3 }  \ldots\] is an inverse system of topological spaces and $X$ is the inverse limit. Recall that for $i\leq j$, $f_{ij}:X_j\to X_i$ denotes the composition $f_i\circ f_{i+1} \circ \cdots \circ f_{j-1}$. We denote by $f_{i\infty}:X\to X_i$ the natural projections, which satisfy $f_{ij}\circ f_{j\infty}=f_{i\infty}$ for each $i\leq j$.

\begin{thm}[{\cite[Theorem 2]{brown}}]
\label{thm:brownapprox}
Let $\{X_i\}_{i=1}^\infty$ be a sequence of compact metric spaces and let $f_i:X_{i+1}\to X_i$ and $g_i:X_{i+1}\to X_i$ be maps. Let $X,Y$ be the inverse limits of $(X_i,f_i)$ and $(X_i,g_i)$, respectively. Then for each $i$ there is a constant $L(g_1,\ldots,g_{i-1})>0$ depending only on $g_1,\ldots,g_{i-1}$ such that if \[D(f_i,g_i)<  L(g_1,\ldots,g_{i-1})\tag{\textasteriskcentered} \label{eq:brown}\] for each $i$ then the following properties are satisfied. For each $i$, the function $F_i:X\to X_i$ defined by $F_i=\lim_{j\to \infty} g_{ij}\circ f_{j \infty}$ is well-defined and continuous. Moreover, the function $F:X\to Y$ defined by $F(s)=(F_1(s),F_2(s),\ldots)$ is a homeomorphism.
\end{thm}

Our strategy to prove Theorem \ref{mainthm:choquetrealization} will be to choose a Choquet simplex and an inverse system of finite-dimensional simplices $\Delta_n$ with that Choquet simplex as an inverse limit, as given by Theorem \ref{thm:llinverselims}. We interpret a simplex $\Delta_n$ as the base of a cone of weights of a system of $\dim(\Delta_n)+1$ arcs on a punctured disk. We may do the same for $\Delta_{n+1}$. The map $f_n:\Delta_{n+1}\to \Delta_n$ may not be realized by including the surface with arcs realizing $\Delta_n$ into the surface with arcs realizing $\Delta_{n+1}$. So we perturb $f_n$ by a small amount to be realized by an inclusion of punctured disks. We then take advantage of Theorem \ref{thm:brownapprox}.

The following lemma is the essential part of the inductive step of the proof. For the statement, we define a \emph{punctured disk} to be a closed disk minus at most finitely many interior points. If $U$ is a punctured disk and $A$ is a collection of pairwise disjoint arcs on $U$ with endpoints on $\partial U$, then we may consider the dual graph $T$ to $A$. Thus, $T$ has one vertex for each component of $U\setminus A$ and two vertices are joined by an edge if the corresponding components are separated by a single arc of $A$. One may check that $T$ is a tree.

\begin{lem}
\label{lem:oddmatrixrealization}
Let $U$ be a punctured disk. Let $A$ be a finite collection of $r$ disjoint, homotopically distinct, homotopically non-trivial arcs on $U$ such that the dual tree to $A$ is homeomorphic to the interval $[0,1]$. Let $\R_+^s$ be a cone with $s>0$ and $\pi:\R_+^s\to \R_+^r$ a linear map whose matrix representative (with respect to the standard bases) has entries which are all positive odd integers. Then $U$ is contained in a punctured disk $V$, together with a finite collection $B$ of $s$ disjoint, homotopically distinct, homotopically non-trivial arcs such that
\begin{itemize}
\item the arcs of $B\cap U$ are homotopic to the arcs in $A$;
\item the dual tree to $B$ is homeomorphic to an interval; and
\item the induced map on cones of weights $C(V)\to C(U)$ for $A$ and $B$ is given by $\pi$.
\end{itemize}
\end{lem}

\begin{proof}
Set $\pi=(a_{ij})_{i=1,j=1}^{r,s}$. Let $T$ be the dual tree to $A$. Embed $T$ in $U$ in such a way that a vertex of $T$ lies in the region of $U\setminus A$ that it represents and an edge of $T$ intersects $A$ in exactly one point, which lies on the arc that it represents. Choose an ordering of the vertices $v_0,v_1,\ldots,v_r$ of $T$ such that the edges $e_i$ of $T$ join $v_{i-1}$ to $v_i$ for each $i$. Denote by $\ell_1^U,\ldots,\ell_r^U$ the arcs in $A$ corresponding to $e_1,\ldots,e_r$. Now, sub-divide each edge $e_i$ into $\sum_{j=1}^s a_{ij}$ edges for each $1\leq i\leq r$.

Having sub-divided the edges of $T$, we label the new vertices of $T$ as follows. We start from the first vertex $v_0$ of $e_1$ and traverse the new vertices in order until we reach $v_1$. We label $v_0$ with a 1. We label the next vertex of $e_1$ with a 2. We then continue alternating between labels of 1 and 2 until we have crossed $a_{11}$ (sub-)edges of $e_1$. Since $a_{11}$ is odd, the final label of a vertex of $e_1$ that we put down will be a 2. We then consider the next $a_{12}$ edges of $e_1$, starting with the last vertex labeled 2. We label them alternately 2 and 3 until we have crossed $a_{12}$ edges. Since we started with a 2 and $a_{12}$ is odd, our last label will be a 3. Continue in this way, alternating between 3 and 4, with the next $a_{13}$ edges of $e_1$ and so on, until we have traversed all $\sum_{j=1}^s a_{1j}$ edges of $e_1$. The last $a_{1s}$ edges of $e_1$ will have vertices labeled by $s$ and $s+1$ and the final vertex $v_1$ of $e_1$ will be labeled by $s+1$. Now we will label the vertices of the first $a_{2s}$ edges of $e_2$. The first vertex $v_1$ of $e_2$ has already been labeled $s+1$. We label the next $a_{2s}$ vertices alternately by $s$ and $s+1$ and therefore the final vertex labeled in this way will be labeled by $s$. We then label the vertices of the next $a_{2(s-1)}$ edges of $e_2$ alternately $s$ and $s-1$, the vertices of the next $a_{2(s-2)}$ edges of $e_2$ alternately $s-1$ and $s-2$, and so on. The final $a_{21}$ edges of $e_2$ will be labeled alternately 2 and 1 and the final vertex $v_2$ of $e_2$ will be labeled 1. We then repeat the process, labeling the first $a_{31}$ edges of $e_3$ alternately 1 and 2 and so on, and continue until we have labeled all the vertices of $T$.  See Figure \ref{fig:choquetinterval} for an example.

\begin{figure}[h]

\centering

\begin{subfigure}[t]{\textwidth}

\centering
\def\svgwidth{0.6\textwidth}
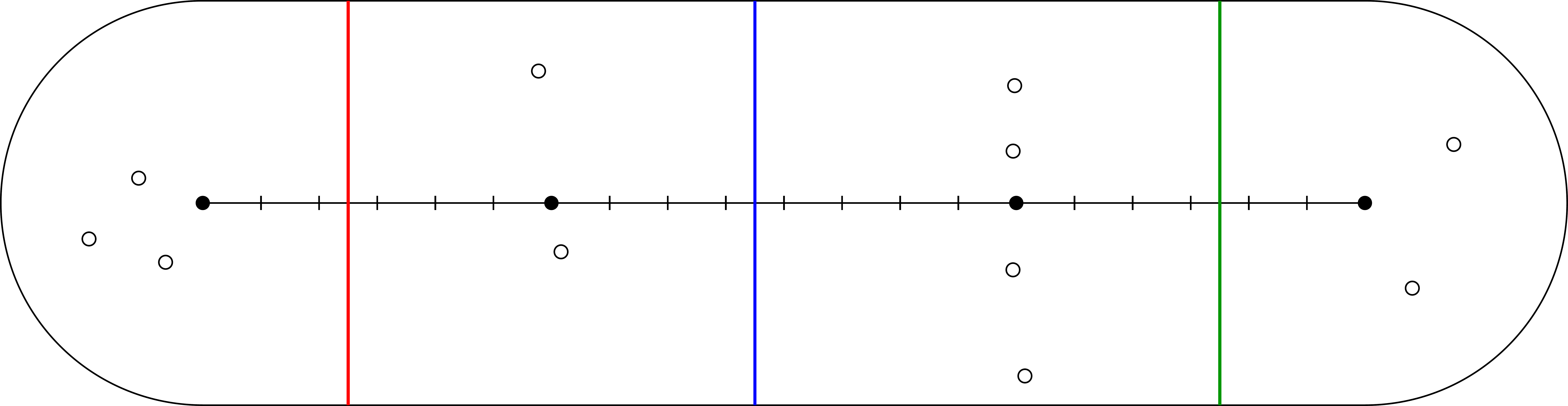

\caption{Step 1 of constructing the arcs $B$ from the arcs $A$ and the matrix $\pi$. The punctured disk $U$ contains three homotopy classes of arcs drawn from left to right. The dual tree is embedded in $U$ and the edges of $T$ are sub-divided into 6, 8, and 6 edges, respectively.}
\label{fig:choquetinterval}
\end{subfigure} 

\begin{subfigure}[t]{\textwidth}

\centering
\def\svgwidth{0.6\textwidth}
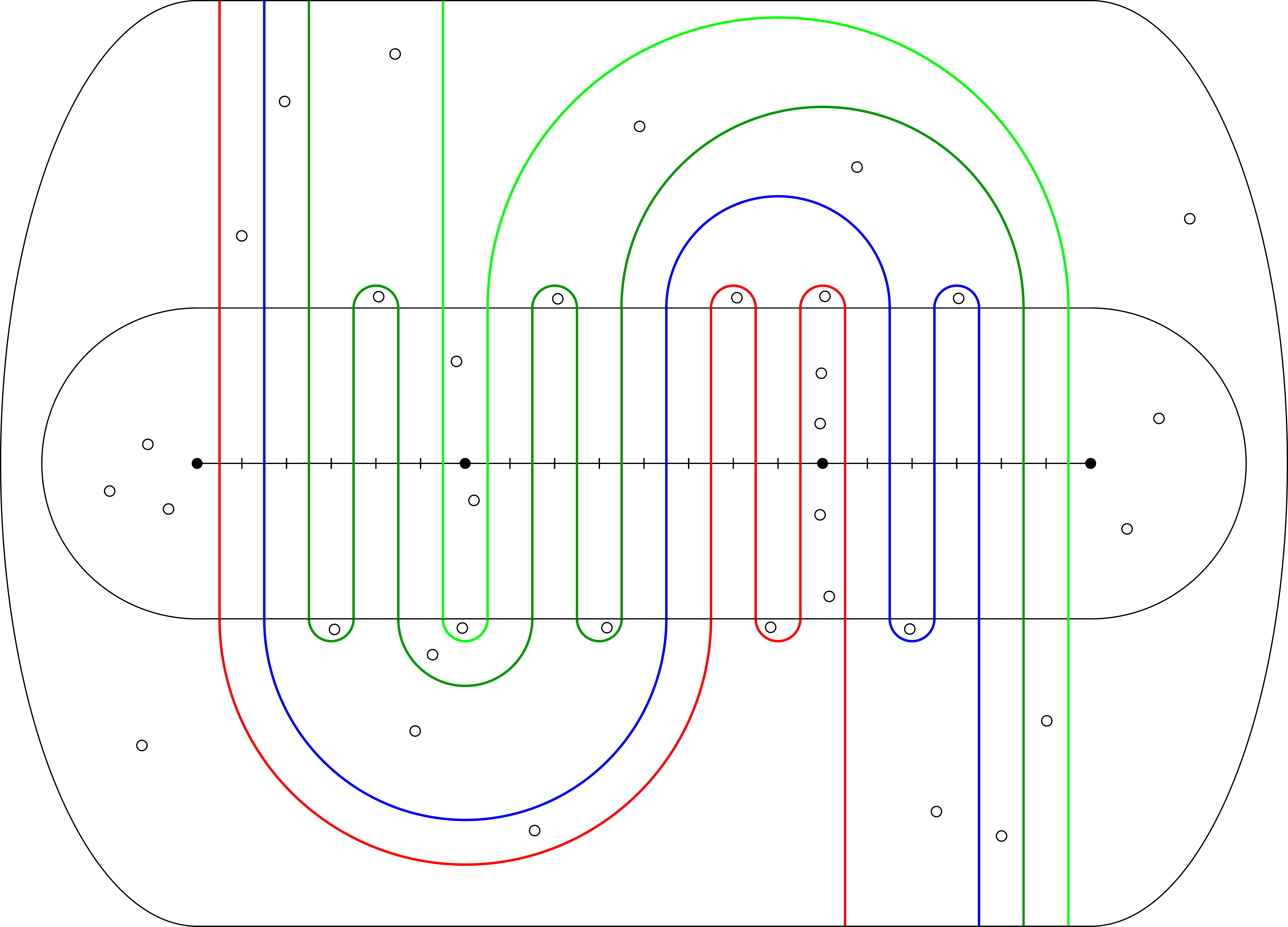

\caption{Step 2 of constructing the arcs $B$. The punctured disk $V$ is obtained by adding an annulus to $U$. For each $1\leq i\leq s$, one arc is drawn through all the edges of $T$ with endpoints labeled by $i$ and $i+1$. There are four homotopy classes in $B$ in this case.}
\label{fig:choquetinduction}
\end{subfigure}

\caption{Constructing a punctured disk $V$ and a system of arcs $B$ from the punctured disk $U$ together with the arcs $A$ and the transition matrix $\protect\begin{pmatrix} 1 & 1 & 3 & 1 \\ 3 & 1 & 3 & 1 \\ 1 & 3 & 1 & 1 \protect\end{pmatrix}$.}
\end{figure}

Now, embed $U$ into a larger disk $V$ by gluing on a closed annulus to the boundary component of $U$. We construct a set $B$ of arcs on $V$ with the desired properties. To do this, we construct for each $1\leq j\leq s$ an arc $\ell_j^V$ on $V$. The arc $\ell_j^V$ will intersect $U$ in $\sum_{i=1}^r a_{ij}$ arcs. The first time $\ell_j^V$ traverses $U$ it will cross through the first edge of $T$ (closest to $v_0$) with endpoints labeled $j$ and $j+1$. The next time it traverses $U$ it will cross through the second edge of $T$ (second closest to $v_0$) with endpoints labeled $j$ and $j+1$, and moreover do so in the \emph{opposite direction} to the first traversal. We continue this process inductively from the lowest edge with endpoints $j$ and $j+1$ to the highest, switching directions through $U$ each time. Each time $\ell_j^V$ traverses $U$ it crosses an edge contained in $e_i$ for some $i$ and when it does so we require the arc of intersection of $\ell_j^V$ with $U$ to be homotopic to $\ell_i^U$. See Figure \ref{fig:choquetinduction} for an example. We claim that one may use this recipe for each $1\leq j\leq s$ to construct the arcs $\ell_j^V$ in such a way that they are \emph{pairwise disjoint}.

To see this last claim, construct the arc $\ell_1^V$ as described. The intersection $\ell_1^V\cap U$ separates the vertices labeled 1 from the vertices labeled $2,\ldots,s+1$ in $U$. Moreover, between any two vertices labeled 1 there are an \emph{even} number of edges with vertices labeled 2 and 3. This ensures that $\ell_2^V$ may be constructed, disjoint from $\ell_1^V$, using the same recipe. The intersection $\ell_2^V\cap U$ separates the vertices labeled 1 and 2 from the vertices labeled $3,\ldots,s+1$ and between any two vertices labeled 2 there are an \emph{even} number of edges with vertices labeled 3 and 4. This ensures that $\ell_3^V$  may be constructed disjoint from $\ell_1^V$ and $\ell_2^V$. Inductively, we construct $\ell_4^V,\ldots, \ell_s^V$. We denote by $B$ the union of $\ell_1^V,\ldots,\ell_s^V$. Then for each $j$, $\ell_j^V$ intersects $U$ in $a_{ij}$ arcs homotopic to $\ell_i^U$, for each $i$. By construction, $\ell_i^V$ separates $\ell_{i-1}^V$ from $\ell_{i+1}^V$. Thus, the dual tree to $B$ is an interval. Finally, we need to ensure that the $\ell_j^V$ are homotopically non-trivial and pairwise non-homotopic. To do this, we add one puncture to each component of $(V\setminus U)\setminus B$. This completes the proof.
\end{proof}

We will use one other easy technical lemma to prove Theorem \ref{mainthm:choquetrealization}. If $\Delta$ is a finite-dimensional simplex then we may endow it with the $\ell^\infty$-distance $d_\infty$ defined as follows. Denoting by $v_1,\ldots,v_k$ the vertices of $\Delta$, we may represent each point $p\in \Delta$ uniquely as a convex combination $p=\sum_{i=1}^k a_iv_i$ where $0\leq a_i\leq 1$ for each $i$ and $\sum a_i=1$. The numbers $a_1,\ldots,a_k$ are the \emph{barycentric coordinates} of the point $p$. If $p=\sum_{i=1}^k a_i v_i$ and $q=\sum_{i=1}^k b_i v_i$ are points of $\Delta$, then their $\ell^\infty$-distance is \[d_\infty(p,q)=\sup \{ |a_i-b_i| : 1\leq i\leq k\}.\] As in Theorem \ref{thm:brownapprox}, $D(\cdot,\cdot)$ denotes the $L^\infty$-distance between maps with the same domain and codomain. The proof of the following lemma is standard and left to the reader.

\begin{lem}
\label{lem:linfinitydist}
Let $\Delta_1$ and $\Delta_2$ be finite-dimensional simplices endowed with their $\ell^\infty$ metrics. Let $F,G:\Delta_1\to \Delta_2$ be affine maps. Suppose that $d_\infty(F(v),G(v))\leq \epsilon$ for all vertices $v$ of $\Delta_1$. Then $D(F,G) \leq \epsilon$.
\end{lem}

Let $f:\Delta\to\Delta'$ be an affine map from a $(q-1)$-simplex $\Delta$ to a
$(p-1)$-simplex $\Delta'$.
Choosing an ordering $v_1,\ldots,v_q$ of the vertices of $\Delta$ and an ordering $w_1,\ldots,w_p$ of the vertices of $\Delta'$ gives a representation of $f$ by a $p\times q$ matrix $M$.
Namely, the $i^{\text{th}}$ column of $M$ gives the barycentric coordinates of $f(v_i)$ with respect to $w_1,\ldots,w_p$.
In particular,
all entries are non-negative and all column sums are 1.
We will consider matrices representing affine maps which have nice properties as described below:

\begin{lem}\label{matrices}
  Let $M$ be a $p\times q$ matrix with non-negative entries, all of whose
  column sums are equal to one. For any $\epsilon>0$, there is a $p\times q$
  matrix $M'$ such that
  \begin{itemize}
  \item all entries of $M'$ are positive and all of its column sums are one;
    \item all entries of $M'$ differ from the corresponding entries of
      $M$ by $<\epsilon$; and
      \item there is a scalar multiple of $M'$ which is a matrix with odd integer entries.
        \end{itemize}
  \end{lem}

  \begin{proof}
    If $p=1$ there is nothing to do, so we assume $p\geq 2$. Let
    $K>\max\{p,\frac 1{\epsilon}\}$ be an odd integer. The matrix
    $pKM$ has column sums $pK$. The largest entry in each column is
    $\geq K$. In each column, replace the smaller $p-1$ entries by nearest odd
    positive integers, and replace the largest entry by the integer
    that keeps the column sum $pK$. This last entry is then odd and
    positive and differs from the original entry by $\leq p-1$. After
    dividing by $pK$ we get the desired matrix $M'$.
  \end{proof}

  \begin{prop}
  \label{prop:oddmatrixlim}
    Every Choquet simplex $\Delta$ is affinely homeomorphic to the
    limit of an inverse system $(\Delta_n,g_n)_{n=1}^\infty$ where: $\Delta_n$ is a finite-dimensional simplex, $g_n:\Delta_{n+1}\to \Delta_n$ is affine, and, choosing orderings for the vertices of $\Delta_n$, the matrix representative for $g_n$ in barycentric coordinates has a scalar multiple with positive odd entries.

  \end{prop}

  \begin{proof}
    Represent $\Delta$ as the inverse limit of $(\Delta_n,f_n)$ given by Theorem
    \ref{thm:llinverselims}. Using Lemma \ref{matrices} with $\epsilon=L(g_1,\ldots,g_{n-1})$, inductively approximate each
    bonding map $f_n:\Delta_{n+1}\to\Delta_n$ by an affine map
    $g_n:\Delta_{n+1}\to\Delta_n$ so that the conditions of Theorem
    \ref{thm:brownapprox} are satisfied. Namely, Lemma \ref{matrices} gives $d_\infty(f_n(v),g_n(v))<L(g_1,\ldots,g_{n-1})$
    at the vertices
	and therefore $D(f_n,g_n)<L(g_1,\ldots,g_{n-1})$ by Lemma
    \ref{lem:linfinitydist}. Let $\Delta'$ be the inverse limit of $(\Delta_n,g_n)$. It is a Choquet simplex by \cite[Theorem 13]{tensor_prod}.
    The map $F:\Delta\to\Delta'$ from the conclusion of Theorem
    \ref{thm:brownapprox} is an affine homeomorphism.
    This follows, since each $F_n$ is the limit of $g_{nm}\circ f_{m\infty}$, both maps in the composition are affine, and a limit of affine maps is affine.
  \end{proof}

The proof of Theorem \ref{mainthm:choquetrealization} now follows quickly:

\begin{proof}[Proof of Theorem \ref{mainthm:choquetrealization}]
Let $\Delta$ be a compact metrizable Choquet simplex. By Proposition \ref{prop:oddmatrixlim} we may represent $\Delta$ as an inverse limit of \[\Delta_1 \xleftarrow{f_1} \Delta_2 \xleftarrow{f_2} \Delta_3 \xleftarrow{f_3} \ldots\] where $f_n$ is represented in barycentric coordinates by a matrix which has a scalar multiple whose entries are positive odd integers. Choose $K_n$ a scalar such that $K_n f_n$ has positive odd integer entries. Let $D_n$ be the dimension of $\Delta_n$. We will construct a sequence of punctured disks $X_n$ together with systems of arcs $A_n$, such that $\{A_n\}_{n=1}^\infty$ is directed and such that the transition map $\pi_n:C(X_{n+1})\to C(X_n)$ is exactly $K_n f_n$. To do this, we choose $X_1$ to be any punctured disk containing $D_1+1$ pairwise disjoint, homotopically distinct, homotopically non-trivial arcs forming the system $A_1$. Given the pair $(X_n,A_n)$ for any $n$, we may use Lemma \ref{lem:oddmatrixrealization} to construct a punctured disk $X_{n+1}$ containing $X_n$ and a system of arcs $A_{n+1}$ on $X_{n+1}$ for which the transition map $\pi_n:C(X_{n+1})\to C(X_n)$ is exactly $K_n f_n$.

Denote $C_n=C(X_n)$. Now, we choose as a base $\Sigma_1$ for the cone $C_1=\R_+^{D_1+1}$ the convex hull of the standard basis vectors $e_i=(0,\ldots,0,1,0,\ldots,0)$. We choose as a base for $C_n=\R_+^{D_n+1}$ the inverse image $\Sigma_n\vcentcolon=\pi_{1n}^{-1}(\Sigma_1)$. Since the column sums of $\pi_n$ are $K_n$, $\Sigma_n$ is the convex hull of the multiples $\frac{1}{K_1\cdots K_n}e_i$ and the induced map $\Sigma_{n+1}\to \Sigma_n$ is exactly $f_n$, in barycentric coordinates. Thus the inverse limit of $(\Sigma_n,\pi_n)_{n=1}^\infty$ is affinely homeomorphic to $\Delta$.

To finish the proof, for an arbitrary infinite type hyperbolic surface $X$ of the first kind, we need to produce an example of a minimal lamination $\Lambda$ on $X$ whose cone of transverse measures has a base affinely homeomorphic to $\Delta$. First consider the case that $X$ is the flute surface (genus zero with countably many punctures, exactly one of which is non-isolated). Note that such a hyperbolic metric of the first kind $X$ on the flute surface exists, e.g. by \cite[Theorem 4]{basmajian}. We may choose an exhaustion $Y_1\subset Y_2\subset Y_3\subset \ldots$ where $Y_n$ is a disk with the same number of punctures as $X_n$ for each $n$. By choosing a homeomorphism of $X_1$ with $Y_1$, we may embed $X_1$ in $X$. By choosing a homeomorphism from the annulus $Y_2\setminus Y_1$ to the annulus $X_2\setminus X_1$, we may extend the embedding of $X_1$ to an embedding of $X_2$. Continue this process inductively. Thus, $X_1\subset X_2\subset \ldots$ is identified with the exhaustion $Y_1\subset Y_2\subset \ldots$ and the directed system of arcs $\{A_n\}_{n=1}^\infty$ pushes forward to a directed system of arcs on $X$. Thus, we may assume that $X_1\subset X_2\subset \ldots$ is an exhaustion of $X$ and that $\{A_n\}_{n=1}^\infty$ is a directed system of collections of arcs on the subsurfaces $X_n$ of $X$. Let $\Lambda\subset X$ be the inverse limit lamination of $\{A_n\}_{n=1}^\infty$. Then $\mathcal M(\Lambda)$ is affinely homeomorphic to the limit of the inverse system $C_1 \longleftarrow C_2 \longleftarrow C_3 \longleftarrow \ldots$ described earlier. Hence $\mathcal M(\Lambda)$ has a base affinely homeomorphic to $\Delta$. By Proposition \ref{prop:inverselimminimal}, $\Lambda$ is minimal.

Finally, we consider the case of an arbitrary $X$. Replace the isolated punctures of the flute surface $Y$ by boundary components. There is a topological embedding of $Y$ into $X$ (see e.g. \cite[Lemma 3.2]{triangulation}). The exhaustion of $Y$ described in the last paragraph pushes forward to an exhaustion $Y_1\subset Y_2\subset \ldots$ of the subsurface $Y$ in $X$. This also embeds the collections of arcs $\{A_n\}_{n=1}^\infty$ as collections of arcs on the subsurfaces $Y_n$ of $X$. Extend the exhaustion $Y_1\subset Y_2\subset \ldots$ of $Y$ to an exhaustion of $X$ as follows. The surface $X$ is the union of $Y$ with at most countably many pairwise disjoint subsurfaces $Z^1,Z^2,\ldots$ meeting $Y$ only along their boundary components. Choose an exhaustion $Z_1^i\subset Z_2^i \subset \ldots$ of each $Z^i$ by punctured compact subsurfaces. Then we define an exhaustion $X_1\subset X_2\subset \ldots$ by taking $X_n$ to be the union of $Y_n$ with $Z_n^i$ for all $i$ such that $Z_n^i$ meets $Y_n$ (along the boundary). One may verify that this does define an exhaustion of $X$ by punctured compact subsurfaces and by \cite[Proposition 3.1]{complete} we may assume the $X_n$ have geodesic boundary (after a homotopy). Then for each $n$, $A_n$ is a finite collection of arcs on $X_n$ and $\{A_n\}_{n=1}^\infty$ is directed. Taking the inverse limit lamination $\Lambda$ of $\{A_n\}_{n=1}^\infty$, we see that the cones $C_n=C(X_n)$ are unchanged, as are the transition maps $C_{n+1}\to C_n$. Hence, $\mathcal M(\Lambda)$ has a base affinely homeomorphic to $\Delta$ and $\Lambda$ is minimal, as desired.
\end{proof}

\section{Cones of measures without Choquet simplex bases}
\label{sec:nocompactbase}

In the case that there is no compact subsurface intersecting every leaf of $\Lambda$, the cone of transverse measures $\mathcal M(\Lambda)$ may be badly behaved. In fact, the cone may have no compact base or no non-zero elements at all. We explore examples of these properties in this section.

\subsection{A lamination with no non-zero transverse measures}
\label{sec:zeromeasure}

In this subsection we construct a lamination with no non-zero transverse measures. The idea is to construct an inverse limit lamination with the following properties. Consider an exhaustion $X_1\subset X_2 \subset\ldots$ by punctured compact subsurfaces together with a system of arcs $\{A_n\}_{n=1}^\infty$ which is directed. Going from $(X_{n+1},A_{n+1})$ to $(X_{n+2},A_{n+2})$, most of the traversals of arcs of $A_{n+1}$ by arcs of $A_{n+2}$ are on the arc of $A_{n+1}$ which is disjoint from $X_n$. This effect is compounded at later stages so that a huge majority of the transversals of arcs of $A_{n+1}$ by arcs of $A_{n+k}$ for $k\gg 0$ are on the arc of $A_{n+1}$ disjoint from $X_n$. This will force any transverse measure to assign measure 0 to any leaf of the inverse limit lamination $\Lambda$ which intersects $X_n$ (for any $n$).

We will construct nested punctured disks $X_1\subset X_2\subset \ldots$ together with systems of arcs $A_n$ such that $\{A_n\}_{n=1}^\infty$ is directed and such that the inverse system $C(X_1) \longleftarrow C(X_2) \longleftarrow \ldots$ is exactly \[\R_+\xleftarrow{\pi_1 } \R_+^2 \xleftarrow{\pi_2 } \R_+^3 \xleftarrow{ \pi_3  } \ldots \tag{\textasteriskcentered \textasteriskcentered \textasteriskcentered} \label{eq:measurezero} \text{ where }\] \[\pi_1=\begin{pmatrix} 1 & 0 \end{pmatrix}, \ \ \ \pi_2= \begin{pmatrix} 1 & 0 & 0 \\ 2 & 1 & 0 \end{pmatrix}, \ \ \ \pi_3= \begin{pmatrix} 1 & 0 & 0 & 0 \\ 2& 1 & 0 & 0 \\ 2 & 2 & 1 & 0 \end{pmatrix}, \ \ \ \pi_4= \begin{pmatrix} 1 & 0 & 0 & 0 & 0 \\ 2& 1 & 0 & 0 & 0 \\ 2 & 2 & 1 & 0 & 0 \\ 2 & 2 &2 & 1 & 0 \end{pmatrix}, \ \ \ \ldots.\] The construction of the disks $X_n$ and arcs $A_n$ is analogous to that of Lemma \ref{lem:oddmatrixrealization} and Theorem \ref{mainthm:choquetrealization}.

We start by defining $X_1$ to be a punctured disk containing a single homotopically non-trivial arc, which forms the collection $A_1$. Suppose that $X_1,\ldots,X_n$ and $A_1,\ldots,A_n$ have been constructed such that the transition map from $A_{i+1}$ to $A_i$ is exactly $\pi_i$ for each $i<n$ and such that the dual tree to $A_i$ is an interval $[0,1]$ for $i\leq n$. We wish to construct $X_{n+1}$ and $A_{n+1}$ such that $\{A_i\}_{i=1}^{n+1}$ is directed and the transition map from $A_{n+1}$ to $A_n$ is $\pi_n$. Embed the dual interval $T_n$ to $A_n$ in $X_n$ in such a way that the vertices of $T_n$ lie in the regions of $X_n \setminus A_n$ that they represent and an edge of $T_n$ intersects $A_n$ in exactly one point, which lies on the arc that it represents. Denote by $\ell_1^n,\ldots,\ell_n^n$ the arcs of $A_n$. Order the vertices $v_0,v_1,\ldots,v_n$ of $T_n$ and edges $e_1,\ldots,e_n$ of $T_n$ such that $e_i$ joins $v_{i-1}$ to $v_i$ for each $i$ and $e_i$ represents $\ell_i^n$. See Figure \ref{fig:measurezerointerval}. Sub-divide $e_i$ into $2i-1$ edges for $1\leq i\leq n$. Label the vertices of $e_i$, in order from $v_{i-1}$ to $v_i$, by \[i, i-1, i-2,\ldots,2, 1,2,\ldots,i-1,i,i+1.\] We form a larger disk by gluing a closed annulus to the boundary component of $X_n$. On $X_n$ we place one arc $\ell_i^{n+1}$ for each $1\leq i\leq n+1$ which crosses through all the edges of $T_n$ with endpoints labeled by $i,i+1$. The arc $\ell_i^{n+1}$ will alternate directions through $X_n$ each time it crosses it, and whenever it passes through a sub-edge of $e_j$, it will traverse an arc homotopic to $\ell_j^n$. See Figure\ref{fig:measurezero}. The $\ell_1^{n+1},\ldots,\ell_n^{n+1}$ may be constructed such that $\ell_j^{n+1}$ separates $\ell_{j-1}^{n+1}$ from $\ell_{j+1}^{n+1}$. We additionally construct one more arc $\ell_{n+1}^{n+1}$ on $X_{n+1}$ which is disjoint from $X_n$ and is separated from $\ell_{n-1}^{n+1}$ by $\ell_n^{n+1}$. We set $A_{n+1}=\{\ell_i^{n+1}\}_{i=1}^{n+1}$ and place a puncture in each component of $(X_{n+1}\setminus X_n)\setminus A_{n+1}$. The arcs of $A_{n+1}$ are then disjoint, homotopically non-trivial, and homotopically distinct, the dual tree to $A_{n+1}$ is an interval, and the transition map is exactly $\pi_n$, as desired.

\begin{figure}[h]
\centering

\begin{subfigure}[t]{\textwidth}

\centering
\def\svgwidth{0.55\textwidth}
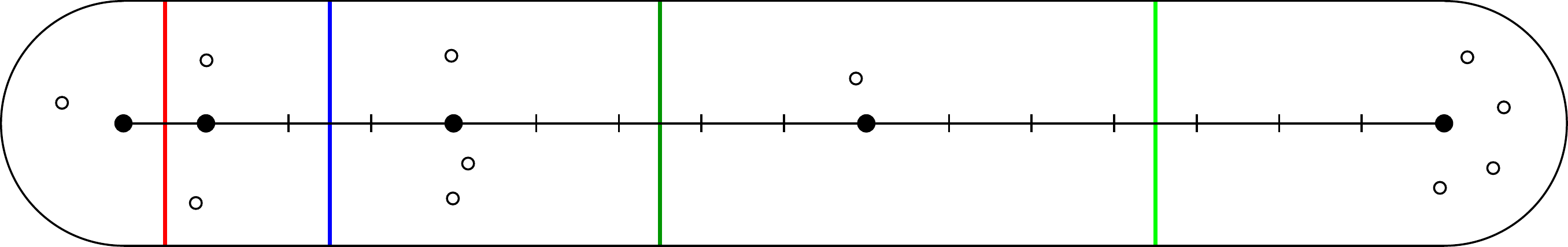

\caption{Realizing the transition matrix $\pi_4$. The dual interval $T_4$ is embedded in $X_4$ and its edges are sub-divided.}
\label{fig:measurezerointerval}
\end{subfigure}

\begin{subfigure}[t]{\textwidth}

\centering
\def\svgwidth{0.55\textwidth}
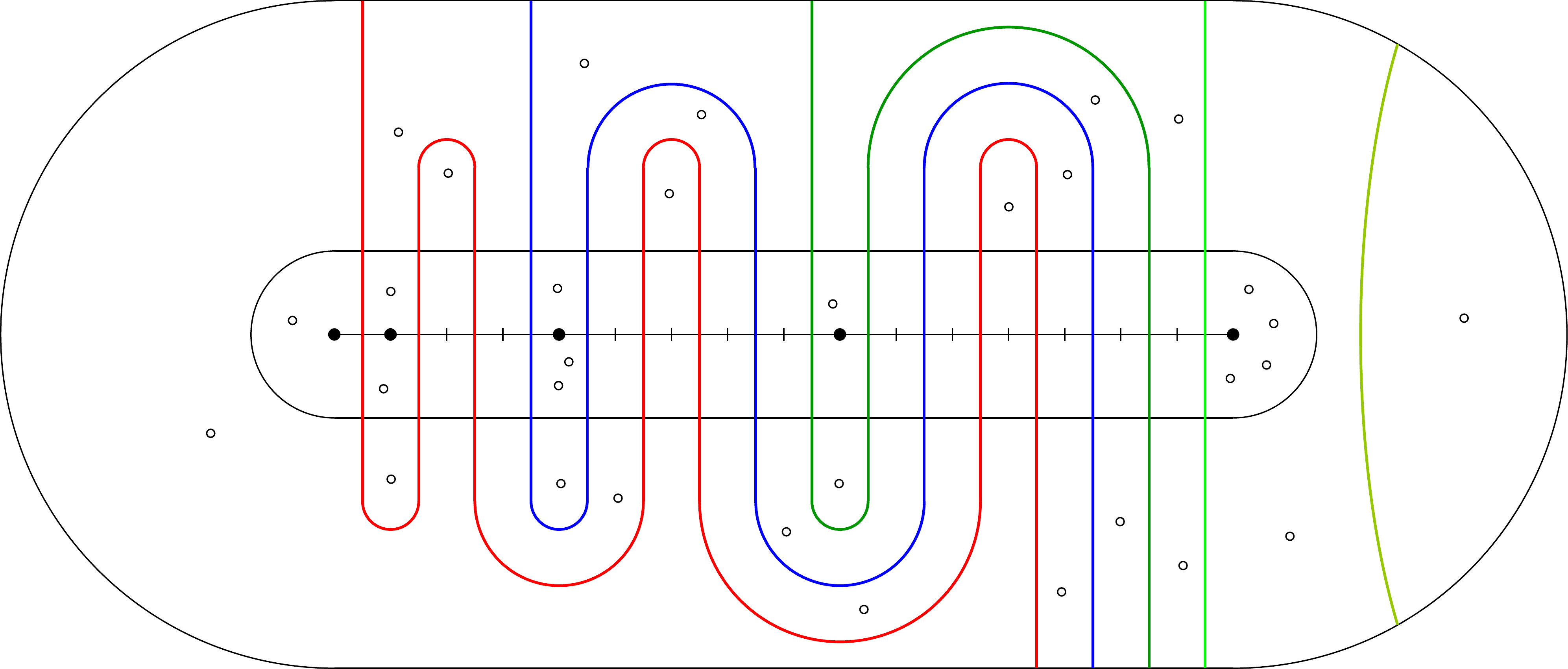

\caption{The arcs $A_5$ on $X_5$. The arc $\ell_j^{n+1}$ crosses through the edges labeled $[j,j+1]$ for $1\leq j\leq n$, while $\ell_{n+1}^{n+1}$ is disjoint from $X_n$.}
\label{fig:measurezero}
\end{subfigure}

\caption{Constructing $(X_5,A_5)$ inductively from $(X_4,A_4)$ and the transition matrix $\pi_4$.}
\end{figure}

Now we can prove Theorem \ref{mainthm:zeromeasure} from the introduction:

\begin{proof}[Proof of Theorem \ref{mainthm:zeromeasure}]

Set $\Lambda$ to be the inverse limit lamination of $\{A_n\}_{n=1}^\infty$. The cone $\mathcal M(\Lambda)$ is the limit of the inverse system (\ref{eq:measurezero}). The transition matrix $\pi_{nm}=\pi_n \circ \pi_{n+1} \circ \cdots \circ \pi_{m-1}$ for $m\geq n$ may be determined as follows. Set $i=m-n$. Then \[\pi_{nm}=\begin{pmatrix} 1 & 0 & 0 & \ldots & 0 & 0 & \ldots & 0 \\ a_1^i & 1 & 0 & \ldots & 0 & 0 & \ldots & 0 \\ a_2^i & a_1^i & 1 & \ldots & 0 & 0 & \ldots & 0 \\ a_3^i & a_2^i & a_1^i & \ldots & 0 & 0 & \ldots & 0 \\  \vdots & \vdots & \vdots & \ddots & \vdots & \vdots & \ddots & \vdots \\ a_{n-1}^i & a_{n-2}^i & a_{n-3}^i & \ldots & 1 & 0 & \ldots & 0\end{pmatrix}\] where $\pi_{nm}$ is $n\times m$, there are $m-n$ columns of zeroes on the right, and $a_j^i$ is the $(j,1)$-entry of the $i^{\text{th}}$ power of the infinite matrix \[A \vcentcolon = \begin{pmatrix} 1 & 0 & 0 & 0 & \ldots \\ 2 & 1 & 0 & 0 & \ldots \\ 2 & 2 & 1 & 0 & \ldots \\ 2 & 2 & 2 & 1 & \ldots \\ \vdots & \vdots & \vdots & \vdots & \ddots \end{pmatrix}.\]
We claim that $a_j^i$ is a polynomial in the variable $i$ of degree $j$. To see this, write $A=I+T$ where $I$ is an infinite identity matrix and $T=A-I$ is lower triangular. Then $A^i = \sum_{k=0}^i \binom{i}{k} T^k$ and by examining the powers $T^k$ we see that $a_j^i = c_0 \binom{i}{0} + c_1 \binom{i}{1} + \ldots + c_j \binom{i}{j}$ where $c_k$ is a fixed (positive) entry of $T^k$. Since $\binom{i}{k}$ is a polynomial in $i$ of degree $k$, the claim follows.


Hence for $j'>j$ we have $\frac{a_j^i}{a_{j'}^i}=\frac{P_j(i)}{P_{j'}(i)}\to 0 \text{ as } i\to \infty$. Fixing $n$, we therefore have that the non-zero columns of $\pi_{nm}$ converge projectively to $(0,0,\ldots,0,1)\in \R_+^n$ as $m\to \infty$. Thus, the intersections of the images of the maps $\pi_{nm}$ in $\R_+^n$ are contained in the sub-cone $0^{n-1} \times \R_+=\{(0,0,\ldots,0,x) : x\geq 0\}$. So the inverse limit $\mathcal M(\Lambda)$ coincides with the limit of the inverse system \[\R_+ \xleftarrow{\pi_1 } 0 \times \R_+ \xleftarrow{\pi_2 } 0^2 \times \R_+ \xleftarrow{\pi_3 } 0^3 \times \R_+\xleftarrow{\pi_4 } \ldots.\] However, the map $\pi_n$ restricted to the cone $0^n \times \R_+$ is just 0. So the inverse limit $\mathcal M(\Lambda)$ is 0. As in Theorem \ref{mainthm:choquetrealization}, the lamination $\Lambda$ may be realized on any infinite type surface.
\end{proof}

We mention an alternative way to prove that $\mathcal M(\Lambda)=0$. Construct $\Lambda$ as above. By studying the construction, one may notice that $\Lambda$ consists of a countable collection $L_1,L_2,\ldots$ of leaves. The leaf $L_i$ accumulates onto $L_{i+1}$ for each $i$ and $L_1$ is isolated, $L_2$ becomes isolated after removing $L_1$, $L_3$ becomes isolated after removing $L_2$, and so on. Since $L_1$ is isolated but accumulates onto $L_2$, it lies outside the support of any measure. The same holds inductively for $L_i$, by removing $L_1,\ldots,L_{i-1}$. Thus $\mathcal M(\Lambda)=0$.

\subsection{Laminations without compact bases}
\label{sec:nobase}

In this section we give a couple of examples of laminations $\Lambda$ which support non-zero transverse measures but for which the cones $\mathcal M(\Lambda)$ nonetheless lack compact bases. Necessarily, each such lamination $\Lambda$ contains sub-laminations disjoint from any given finite type subsurface. In particular, $\Lambda$ cannot be a minimal lamination in any of these cases.

\begin{ex}
Consider the lamination $\Lambda$ of Example \ref{ex:nobase1}. The cone $\mathcal M(\Lambda)\cong \R_+^\N$ has no compact base. For if $B$ were a base then $B$ would contain a multiple of $e_i=(0,\ldots,0,1,0\ldots)$ for each $i\geq 1$ (where $e_i$ has a 1 in entry $i$). Thus, $a_ie_i\in B$ for some $a_i>0$. We see that $a_ie_i\to 0$ as $i\to \infty$, regardless of the values of $a_i$. Since 0 does not lie in $B$, it cannot be compact.
\end{ex}

\begin{ex}
Consider the lamination $\Lambda$ on the infinite type surface $X$ in Figure \ref{fig:nobase2}. Thus, $X$ has a Cantor set of ends, all accumulated by genus. The lamination $\Lambda$ consists of countably many isolated simple closed curves $\Gamma_1,\Gamma_2,\ldots$ plus countably many proper leaves $L_1,L_2,\ldots$ such that $\Gamma_j$ converges to the union $\bigcup_{i=1}^\infty L_i$ as $j\to \infty$. Using an appropriate exhaustion of $X$ shows that $\mathcal M(\Lambda)$ is affinely homeomorphic to the cone $C\subset \ell^1 \times \R_+^{\mathbb{N}}$ defined by \[C=\left\{\left((y_i)_{i=1}^\infty,x_1,x_2,\ldots\right) \in \ell^1 \times \R_+^{\mathbb{N}} : y_i\geq 0 \text{ for all } i \text{ and } x_j \geq \sum y_i \text{ for all } j\right\}.\] Here $\ell^1$ is endowed with its weak${}^*$ topology as the dual of $c_0$, $\ell^1\times \R_+^{\mathbb{N}}$ with its product topology, and $C$ with the subspace topology. The $\ell^1$ coordinates $y_i$ of the cone correspond to weights on the simple closed curves $\Gamma_i$ while the coordinates $x_i$ correspond to weights on the proper leaves $L_i$. As in the last example, $C$ has no compact base.
\end{ex}

\begin{figure}[h]
\centering
\def\svgwidth{0.5\textwidth}
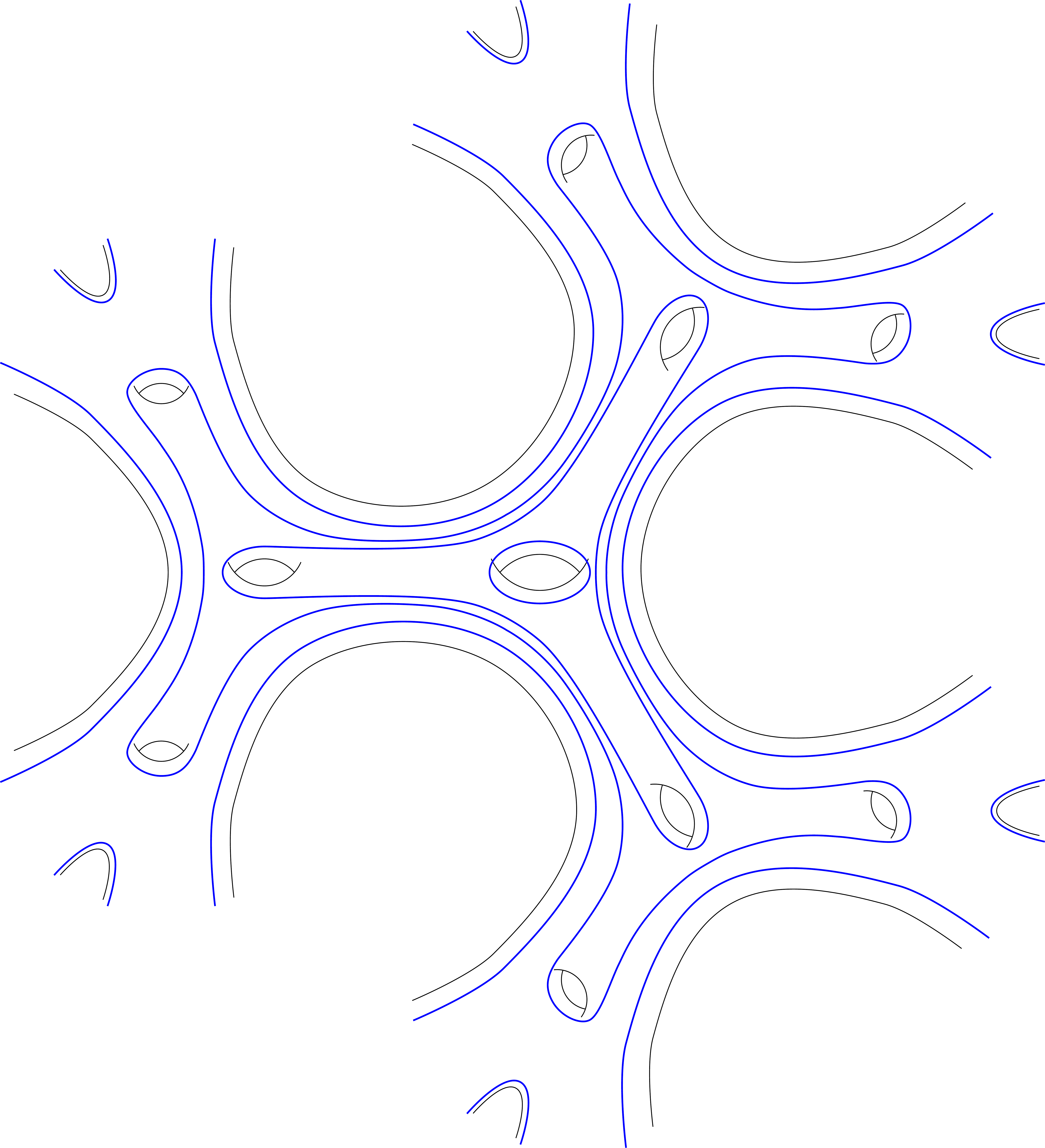

\caption{A lamination on the surface with a Cantor set of ends accumulated by genus, consisting of countably many isolated simple closed curves which limit to a countable union of proper leaves.}
\label{fig:nobase2}
\end{figure}

\bibliographystyle{plain}
\bibliography{transverse}

{\normalsize
\noindent \textbf{M. Bestvina.} Department of Mathematics, University of Utah, Salt Lake City, UT 84112. \\
E-mail: \href{mailto:bestvina@math.utah.edu}{bestvina@math.utah.edu}

\bigskip

\noindent \textbf{A. J. Rasmussen.} Department of Mathematics, University of Utah, Salt Lake City, UT 84112. \\
E-mail: \href{mailto:rasmussen@math.utah.edu}{rasmussen@math.utah.edu}

\end{document}